\documentclass[10pt,oneside]{amsart}
\usepackage[a4paper, textwidth=16.0cm]{geometry}
\usepackage[T1]{fontenc}
\usepackage[english]{babel}
\usepackage{aliascnt}

\usepackage{xcolor} 
\usepackage{tikz}
\usepackage[draft=false]{hyperref} 
\usepackage{mathtools}
\usepackage{amsmath,amsthm,amssymb}
\usepackage[all,cmtip]{xy}
\usepackage{tikz-cd}

\usepackage[capitalise]{cleveref}

\newtheorem{theorem}{Theorem}[subsection]

\newaliascnt{lemma}{theorem}
\newtheorem{lemma}[lemma]{Lemma}
\aliascntresetthe{lemma}
\crefname{Lemma}{Lemma}{Lemmas}
\Crefname{Lemma}{Lemma}{Lemmas}

\newaliascnt{proposition}{theorem}
\newtheorem{proposition}[proposition]{Proposition}
\aliascntresetthe{proposition}
\crefname{Proposition}{Proposition}{Propositions}
\Crefname{Proposition}{Proposition}{Propositions}

\newaliascnt{corollary}{theorem}
\newtheorem{corollary}[corollary]{Corollary}
\aliascntresetthe{corollary}
\crefname{Corollary}{Corollary}{Corollaries}
\Crefname{Corollary}{Corollary}{Corollaries}

\theoremstyle{definition}
\newaliascnt{definition}{theorem}
\newtheorem{definition}[definition]{Definition}
\aliascntresetthe{definition}
\crefname{Definition}{Definition}{Definitions}
\Crefname{Definition}{Definition}{Definitions}

\newaliascnt{notation}{theorem}
\newtheorem{notation}[notation]{Notation}
\aliascntresetthe{notation}
\crefname{Notation}{notation}{Notations}
\Crefname{Notation}{notation}{Notations}

\newaliascnt{construction}{theorem}

\aliascntresetthe{construction}
\crefname{Construction}{construction}{Constructions}
\Crefname{Construction}{construction}{Constructions}

\newaliascnt{example}{theorem}
\newtheorem{example}[example]{Example}
\aliascntresetthe{example}
\crefname{Example}{example}{Examples}
\Crefname{Example}{example}{Examples}

\newaliascnt{remark}{theorem}
\newtheorem{remark}[remark]{Remark}
\aliascntresetthe{remark}
\crefname{Remark}{Remark}{Remarks}
\Crefname{Remark}{Remark}{Remarks}

\usepackage{amsmath,amsthm,amssymb,amsfonts,extarrows}
\usepackage{enumerate,amscd,amsxtra,MnSymbol}
\usepackage{mathrsfs}
\usepackage{color}
\usepackage[cmtip,all]{xy}
\usepackage{eucal}
\usepackage{bbold}
\usepackage{upgreek}
\usepackage{paralist}
\usepackage{tikz-cd}
\usepackage{dsfont}
\usepackage{bbm}
\usepackage[bbgreekl]{mathbbol}

\DeclareMathSymbol\bbDelta \mathord{bbold}{"01}
\DeclareMathSymbol\bDelta \mathord{bbold}{"01}

\newcommand{\bD}{{\mathbb D}}
\newcommand{\bE}{{\mathbb E}}

\renewcommand{\P}{{\mathbb P}}

\newcommand{\bN}{{\mathbb N}}

\newcommand{\mA}{{\mathcal A}}
\newcommand{\mB}{{\mathcal B}}
\newcommand{\mC}{{\mathcal C}}
\newcommand{\mD}{{\mathcal D}}
\newcommand{\mE}{{\mathcal E}}

\newcommand{\mJ}{{\mathcal J}}

\newcommand{\mM}{{\mathcal M}}
\newcommand{\mN}{{\mathcal N}}

\newcommand{\mP}{{\mathcal P}}

\newcommand{\mS}{{\mathcal S}}
\newcommand{\mT}{{\mathcal T}}

\newcommand{\mV}{{\mathcal V}}
\newcommand{\mW}{{\mathcal W}}

\newcommand{\A}{{\mathrm A}}
\newcommand{\B}{{\mathrm B}}

\newcommand{\E}{{\mathrm E}}
\newcommand{\F}{{\mathrm F}}
\newcommand{\G}{{\mathrm G}}

\renewcommand{\L}{{\mathrm L}}

\renewcommand{\P}{{\mathrm P}}
\newcommand{\Q}{{\mathrm Q}}
\newcommand{\R}{{\mathrm R}}
\newcommand{\rS}{{\mathrm S}}

\newcommand{\Z}{{\mathrm Z}}

\newcommand{\nat}{{\mathrm {nat}}}

\newcommand{\rc}{\mathrm{c}}

\newcommand{\bj}{\mathrm{j}}
\newcommand{\bi}{\mathrm{i}}
\newcommand{\m}{\mathrm{m}}
\newcommand{\bk}{\mathrm{k}}

\newcommand{\g}{\mathrm{g}}
\newcommand{\n}{\mathrm{n}}

\newcommand{\op}{\mathrm{op}}
\newcommand{\adj}{\mathrm{adj}}
\newcommand{\Sk}{\mathrm{Sk}}

\newcommand{\Ch}{\mathsf{Ch}}

\newcommand{\colim}{\mathrm{colim}}

\newcommand{\LMod}{{\mathrm{LMod}}}
\newcommand{\RMod}{{\mathrm{RMod}}}

\newcommand{\Bord}{{\mathrm{Bord}}}

\newcommand{\rev}{{\mathrm{rev}}}

\newcommand{\Cmon}{{\mathrm{Cmon}}}

\newcommand{\oker}{{\overset{\to}{\ker}}}

\newcommand{\Env}{{\mathrm{Env}}}  
\newcommand{\ot}{\otimes}

\newcommand{\co}{\mathrm{co}}

\newcommand{\id}{\mathrm{id}}

\newcommand{\Cat}{\mathrm{Cat}}

\newcommand{\Set}{\mathrm{Set}}

\newcommand{\Alg}{\mathrm{Alg}}
\newcommand{\Calg}{\mathrm{Calg}}

\newcommand{\coCart}{\mathrm{coCart}}

\newcommand{\Mon}{\mathrm{Mon}}
\newcommand{\Fun}{\mathrm{Fun}}

\newcommand{\Op}{{\mathrm{Op}}}

\newcommand{\Cocart}{{\mathrm{Cocart}}}

\newcommand{\surj}{{\mathrm{surj}}}

\newcommand{\lax}{{\mathrm{lax}}}
\newcommand{\oplax}{{\mathrm{oplax}}}

\newcommand{\tu}{{\mathbb 1}}

\newcommand{\ev}{{\mathrm{ev}}}

\newcommand{\Map}{{\mathrm{Map}}}

\newcommand{\Enr}{{\mathrm{Enr}}}

\newcommand{\bZ}{{\mathbb{Z}}}

\newcommand{\Mor}{{\mathrm{Mor}}}

\newcommand{\cl}{\mathrm{cl}}

\newcommand{\Free}{\mathrm{Free}}

\newcommand{\cop}{\mathrm{cop}}

\newcommand{\Sp}{{\mathrm{Sp}}} 

\newcommand{\Fin}{{\mathrm{Fin}}}

\newcommand{\mon}{{\mathrm{mon}}}

\newcommand{\Ho}{{\mathrm{Ho}}}

\newcommand{\vectimes}{{\overset{\to}{\times}}}

\newcommand{\cube}{{\,\vline\negmedspace\square}}

\begin{document}

\title{Stable homotopy theory of higher categories with adjoints}

\author{Hadrian Heine, \\ Max Planck Institute Bonn, Germany, \\ heine@mpim-bonn.mpg.de}
\maketitle

\begin{abstract}

We demonstrate that higher categories with adjoints behave like homotopy types. By the homotopy hypothesis, homotopy types can be regarded as higher categories in which all cells are invertible. 
We show that replacing invertibility by the existence of higher adjoints still supports a homotopy theory that is at least as rich as classical homotopy theory.

The basic invariants of this theory are the homotopy monoidal categories
\[
\Pi_n(X,x)=\operatorname{Ho}\bigl(\Omega^n(X,x)\bigr)
\]
associated to a based \((\infty,\infty)\)-category with adjoints \((X,x)\). These invariants are monoidal categories with duals and replace iterated loops with iterated endomorphisms. The classical homotopy groups of the classifying space are recovered as the isomorphism classes of objects in the groupoid completion.
We prove a Whitehead theorem for homotopy monoidal categories and show that fibrations of higher categories give rise to long oriented exact sequences of homotopy monoidal categories.

To develop homological algebra of higher categories with adjoints,
we prove an equivalence between skeletal sequences of categorical spectra with adjoints and categorical chain complexes, extending a result outlined by Lurie in the proof of the cobordism hypothesis.

\end{abstract}

\tableofcontents

\section{Introduction}

\subsubsection{Higher categories with adjoints behave like homotopy types}

Duality is one of the fundamental organizing principles of modern mathematics. In higher category theory, duality is encoded by adjoints: instead of requiring morphisms to be invertible, one asks them to admit left and right adjoints, and iterates this condition through all higher dimensions. Higher categories with adjoints arise naturally in topology, geometry, representation theory, and mathematical physics. Their most prominent examples are bordism higher categories, and their importance is particularly visible in the theory of fully extended topological quantum field theories.

The cobordism hypothesis gives a conceptual classification of fully extended framed topological quantum field theories: symmetric monoidal functors out of the framed bordism higher category are controlled by fully dualizable objects in the target. From this perspective, higher categories with adjoints are not merely organizational devices. They are among the basic geometric objects underlying extended field theories. The fully extended framed bordism category is itself an infinity category with adjoints, while natural targets of field theories are higher categories of dualizable algebraic or geometric objects. This provides a compelling reason to understand higher categories with adjoints intrinsically, rather than only through invariants obtained after passage to spaces or spectra.

A classical paradigm of geometry is that geometric objects give rise to invariants valued in homotopy theory, where the tools of algebraic topology make them accessible to computation. An emerging perspective is that many such invariants naturally refine to higher categories before passing to spaces, spectra, or K-theory. Fully extended cobordism categories retain geometric and compositional information that is invisible after passage to their classifying spaces, while higher categories of representations or modules refine coarser invariants such as K-theory. This raises a basic question: can one develop the computational machinery of algebraic topology directly at the level of higher categories, without first discarding their genuinely higher-categorical structure?

The perspective of this paper is that the classical passage
\[
\text{geometry}\longrightarrow\text{homotopy theory}
\]
admits a refinement
\[
\text{geometry}
\longrightarrow
\text{higher categories}
\longrightarrow
\text{homotopy types}.
\]

The present paper addresses this question for higher categories with adjoints. Its central viewpoint is that higher categories with adjoints behave like homotopy types. By the homotopy hypothesis, homotopy types are higher categories in which every cell is invertible. Higher categories with adjoints arise by replacing invertibility with the weaker condition that every cell admits a left and right adjoint.

A first indication of this principle is provided by the classifying space functor
\[
\tau_0: \infty\mathrm{Cat}\to\infty\mathrm{Grp}.
\]

For an arbitrary $\infty$-category, passage to the classifying space discards the directed structure. In the presence of adjoints, however, this passage retains the entire homotopy theory in a precise sense.

We prove that $\tau_0$ is $\infty$-connected: for every $A,B \in\infty\mathrm{Cat}^{\mathrm{adj}}$ the canonical functor $$ \operatorname{Mor}_{\infty\mathrm{Cat}^{\mathrm{adj}}}(A,B)\longrightarrow \operatorname{Map}_{\mathcal S}(\tau_0A,\tau_0B)$$ is $\infty$-connected. Thus every map between the associated homotopy types lifts to the higher-categorical level, as does every homotopy between such maps, every higher homotopy between those homotopies, and so on in all dimensions. In this sense, the homotopy theory of spaces is obtained by decategorifying a richer homotopy theory of $\infty$-categories with adjoints.

The basic invariants of this theory are the homotopy monoidal categories
\[
\Pi_n(X,x)=\operatorname{Ho}_1\bigl(\Omega^n(X,x)\bigr).
\]
They are the analogue of classical homotopy groups. Since $\Omega^n(X,x)$ is naturally $\bE_n$-monoidal, $\Pi_n(X,x)$ is monoidal for $n=1$, braided monoidal for $n=2$, and symmetric monoidal for $n>2$. These invariants retain genuinely categorical information. If $X$ has adjoints, then after inverting all morphisms in $\Pi_n(X,x)$ and passing to isomorphism classes of objects, one recovers the homotopy groups of the classifying space
\[
\pi_n(\tau_0X,x).
\]
Thus the classical homotopy groups arise as a decategorification of the richer homotopy monoidal categories.

\subsubsection*{Homotopy theory of $\infty$-categories with adjoints}

The main goal of the paper is to substantiate the preceding viewpoint by developing an intrinsic homotopy theory of $\infty$-categories with adjoints which closely parallels the classical homotopy theory of spaces.

A first decisive result is a categorical analogue of Whitehead's theorem. We prove that the homotopy monoidal categories detect equivalences: a functor
\[
F: C\longrightarrow D
\]
between $\infty$-categories with adjoints is an equivalence if and only if it induces equivalences
\[
\Pi_*(C,X)\longrightarrow \Pi_*(D,F(X))
\]
for every object $X\in C$. There is an analogous stable Whitehead theorem for categorical spectra with adjoints: a map of categorical spectra with adjoints is an equivalence precisely when it induces equivalences on all stable homotopy monoidal categories. Thus homotopy monoidal categories do not merely refine classical homotopy groups; they detect equivalences between $\infty$-categories and categorical spectra with adjoints.

The analogy extends beyond detection of equivalences to the structure of fiber sequences. A fibration of based $\infty$-categories with adjoints gives rise to a long oriented sequence
\[
\cdots\longrightarrow
\Omega Y
\longrightarrow
\mathop{\operatorname{fib}}\limits^{\rightarrow}(\varphi)
\longrightarrow
X
\xrightarrow{\varphi}
Y.
\]

We show that applying homotopy monoidal categories produces a categorical analogue of the classical long exact sequence of homotopy groups. 

\subsubsection*{Homological algebra of $\infty$-categories with adjoints}

The homotopy theory of $\infty$-categories with adjoints has a homological manifestation. We develop the homological side of the theory through categorical spectra with adjoints and categorical chain complexes.

A skeletal filtration of a categorical spectrum with adjoints gives rise to a categorical chain complex, whose terms are symmetric monoidal categories with duals rather than abelian groups.
We prove that the functor from skeletal sequences of categorical spectra with adjoints to categorical chain complexes is an equivalence. 
This reflects again the idea that symmetric monoidal categories with duals are a suitable categorical analogue of abelian groups.

\section{Main results}

We state the main results of this paper.

\subsubsection{Homotopy theory of higher categories with adjoints}

We prove that the classical homotopy theory of homotopy types lifts to the homotopy theory of $\infty$-categories with adjoints.

\begin{theorem}(\cref{superfull}) The oriented functor $$\infty\Cat^\adj \to \mS$$ is $\infty$-connected, i.e. it is essentially surjective and 
for every $A, B \in \infty\Cat^\adj$ the canonical functor $$ \Mor_{\infty\Cat^\adj}(A,B) \to \Map_{\mS}(\tau_0(A),\tau_0(B)) $$
is $\infty$-connected.

The oriented functor $$\Cat\Sp^\adj \to \Sp$$ is $\infty$-connected, i.e. it is essentially surjective and 
for every $A, B \in \Cat\Sp^\adj$ the canonical functor $$ \Mor_{\Cat\Sp^\adj}(A,B) \to \Map_{\Sp}(\tau_0(A),\tau_0(B)) $$
is $\infty$-connected.
 
\end{theorem}

In this sense higher category theory with
adjoints is at least as rich as classical homotopy theory.

\subsubsection{Homotopy monoidal categories}

We prove that the homotopy monoidal categories refine the classical homotopy groups before passing to classifying spaces.

\begin{theorem}(\cref{superfull})
Let $X$ be an infinity category with adjoints and $m\geq 0$. The canonical functor
\[
\Pi_m(X)^{\mathrm g}
\longrightarrow
\Pi_m(\tau_0X)
\]
induces an isomorphism
\[
\pi_0\bigl(\Pi_m(X)^{\mathrm g}\bigr)
\xrightarrow{\cong}
\pi_0\bigl(\Pi_m(\tau_0X)\bigr)
\simeq
\pi_m(\tau_0X).
\]
Here $\Pi_m(X)^{\mathrm g}$ denotes the groupoid obtained from $\Pi_m(X)$ by inverting all morphisms.
\end{theorem}

We next prove a categorical analogue of Whitehead's theorem. A functor $ F: C\longrightarrow D $
is an $n$-equivalence if it induces an equivalence on maximal subspaces and an
$(n-1)$-equivalence on morphism $\infty$-categories.

\begin{theorem}(\cref{homten})
Let $C$ be an $\infty$-categories with adjoints and $n\geq 1$.
A functor 
\[
F: C\longrightarrow D
\]
is an $n$-equivalence if and only if for every object $X\in C$, the induced
functor
\[
\Pi_m(C,X)
\longrightarrow
\Pi_m(D,F(X))
\]
is an equivalence for every $m<n$.

In particular, $F$ is an equivalence if and only if, for every $X\in C$, the
induced graded functor
\[
\Pi_*(C,X)
\longrightarrow
\Pi_*(D,F(X))
\]
is an equivalence.
\end{theorem}

We obtain a Whitehead theorem for categorical spectra
with adjoints:

\begin{corollary} (\cref{maf})
Let $C$ and $D$ be categorical spectra with adjoints. A map
\[
F: C\longrightarrow D
\]
is an equivalence if and only if, for every object $X\in C$, the induced
$\mathbb Z$-graded functor
\[
\Pi_*(C,X)
\longrightarrow
\Pi_*(D,F(X))
\]
is an equivalence.
\end{corollary}

We prove a categorical analogue of the long exact sequence of homotopy groups.

\begin{theorem}(\cref{seqfib})
Let
\[
\varphi: X\longrightarrow Y
\]
be a fibration of based $\infty$-categories with adjoints. Then the associated long sequence
\[
\cdots
\longrightarrow
\Pi_{1}(Y)
\xrightarrow{\Pi_0(\delta)}
\Pi_0\!\left(
\mathop{\operatorname{fib}}\limits^{\rightarrow}(\varphi)
\right)
\xrightarrow{\Pi_0(p)}
\Pi_0(X)
\xrightarrow{\Pi_0(\varphi)}
\Pi_0(Y)
\]
alternates between exactness and oriented exactness.

More precisely, the canonical functors
\[
\Pi_{n+1}(Y)
\longrightarrow
\operatorname{fib}\!\bigl(\Pi_n(p)\bigr)
\]
and
\[
\Pi_n\!\left(
\mathop{\operatorname{fib}}\limits^{\rightarrow}(\varphi)^{\mathrm{nat}}
\right)
\longrightarrow
\mathop{\operatorname{fib}}\limits^{\rightarrow}
\!\bigl(\Pi_n(\varphi)\bigr)
\]
are essentially surjective, full and isofull, and the functor 
\[
\Pi_{n+1}\!\left(X^{\mathrm{nat}}\right)
\longrightarrow
\mathop{\operatorname{fib}}\limits^{\rightarrow}
\!\bigl(\Pi_n(\delta)\bigr)
\]
is essentially surjective and isofull.

Here $(-)^{\mathrm{nat}}$ denotes the maximal wide subcategory on which the corresponding lax null-transformation becomes a natural transformation.

\end{theorem}

We show next that the Gray tensor product of $\infty$-categories descends to $\infty$-categories with adjoints.

\begin{theorem} (\cref{Grayadj}) Let $n,m \geq 0.$ For any $n$-category with adjoints $A$ and $n+m$-category $B$ the
$n+m$-categories 
$$\operatorname{Fun}^{\mathrm{lax}}(A,B), \\ \operatorname{Fun}^{\mathrm{oplax}}(A,B)$$ admit adjoints of dimension smaller $m.$

In particular, for any $\infty$-categories with adjoints $A$ and $B$ the
$\infty$-categories 
$$\operatorname{Fun}^{\mathrm{lax}}(A,B), \\ \operatorname{Fun}^{\mathrm{oplax}}(A,B)$$ admit adjoints.

\end{theorem}

\begin{corollary}(\cref{Grayloc}) The full subcategory
\[
\infty\mathrm{Cat}^{\mathrm{adj}}
\subset
\infty\mathrm{Cat}
\]
of $\infty$-categories with adjoints is a monoidal localization, where $\infty\mathrm{Cat}$ carries the Gray monoidal structure.

\end{corollary}

\subsubsection*{Homological algebra}

We prove the following Dold--Kan correspondence:

\begin{theorem}(\cref{DK})
For every $1 \leq n\leq\infty$ there is an equivalence
\[
\mathrm{Sk}_n
\longrightarrow
\mathrm{CatCh}_n
\]
sending a skeletal sequence
\[
X_0
\longrightarrow
X_1
\longrightarrow
\cdots
\longrightarrow
X_n
\]
of connective categorical spectra with adjoints of length $n$ to the categorical chain complex
\[
\left(
0\mathop{\times}\limits^{\rightarrow}_{X_1}X_0
\longrightarrow X_0,\;
0\mathop{\times}\limits^{\rightarrow}_{\Omega(X_2)}
\Omega(X_1)
\longrightarrow
\Omega(X_1),\;
\ldots,\;
0\mathop{\times}\limits^{\rightarrow}_{\Omega^{n-1}(X_n)}
\Omega^{n-1}(X_{n-1})
\longrightarrow
\Omega^{n-1}(X_{n-1})
\right)
\]
of length $n.$
\end{theorem}

As a consequence, connective categorical spectra with adjoints can be reconstructed
from categorical chain complexes.

\begin{corollary}(\cref{DK4})
There is a canonical equivalence
\[
\Sk\Cat\Sp
\longrightarrow
\mathrm{CatCh}
\]
sending a skeletal sequence
\[
X_0
\longrightarrow
X_1
\longrightarrow
\cdots
\]
of categorical spectra with adjoints to the categorical chain complex
$$ \{\iota_1(0 \overset{\to}{\times}_{\Omega^\infty(\Omega^{i}(X_{i+1}))} \Omega^\infty(\Omega^i(X_i))) \to \iota_1(\Omega^\infty(\Omega^i(X_i))), i \in \bZ \}. $$

\end{corollary}

\begin{corollary}(\cref{DK3})
There is a canonical embedding
\[
\Cat\Sp^\adj
\longrightarrow
\mathrm{CatCh}
\]
sending a categorical spectrum with adjoints $X$ to the categorical chain complex
$$ \{\iota_1(0 \overset{\to}{\times}_{\Omega^\infty(\Omega^{i}(X))} \Omega^\infty(\Omega^i(X))) \to \iota_1(\Omega^\infty(\Omega^i(X))), i \in \bZ \}$$
associated to the categorical Puppe sequence
\[
\cdots
\longrightarrow
\Omega^2(X)
\longrightarrow
\Omega\!\left(
0\mathop{\times}\limits^{\rightarrow}_X X
\right)
\longrightarrow
\Omega(X)
\longrightarrow
\Omega(X)
\longrightarrow
0\mathop{\times}\limits^{\rightarrow}_{X} X \longrightarrow X \longrightarrow X \longrightarrow ...
\]

\end{corollary}

\subsubsection{Relation to other work}

Lurie \cite{lurie2008classification} studies higher categories with adjoints in the proof of the cobordism hypothesis. He outlines a proof for an equivalence between skeletal categorical sequences of length $n$ and categorical chain complexes of length $n$  \cite[Theorem 3.3.33.]{lurie2008classification}. A rigorous proof
was announced by Haugseng in joint work with Nikolaus \cite{HaugsengCT2025}:

Masuda \cite{masuda2026algebra} studies categorical spectra with adjoints.
He  constructs a Gray smash product for categorical spectra and shows that the Gray smash product descends to categorical spectra with adjoints
\cite[Theorem 6.1.8.]{masuda2026algebra}.

\subsubsection{Acknowledgements}

We thank David Gepner and Markus Spitzweck for helpful discussions.

\subsubsection{Notation and terminology}

We fix a hierarchy of Grothendieck universes whose objects we call small, large, very large, etc.
We call a space small, large, etc. if its set of path components and its homotopy groups are small, large, etc. for any choice of base point. 

We refer to weak $(\infty,n)$-categories for $0 \leq n \leq \infty$ as $n$-categories and refer to weak $(n,n)$-categories as $(n,n)$-categories.
In particular, we refer to $(\infty,1)$-categories as 1-categories or simply categories.

\vspace{1mm}
We write 
\begin{itemize}
\item $\Set$ for the category of small sets.
\item $\Delta$ for (a skeleton of) the category of finite, non-empty, totally ordered sets and order preserving maps, whose objects we denote by $[n] = \{0 < ... < n\}$ for $n \geq 0$.
\item $\mS$ for the $\infty$-category of small spaces.
\item $ \Cat$ for the category of small categories.	
\item $\Map_\mC(A,B)$ for the space of maps $A \to B$ in $\mC$ for any category $\mC$ containing objects $A,B.$
\item $\Fun(\mC,\mD)$ for the category of functors $\mC \to \mD$ between two categories $\mC,\mD$, the internal hom of $\Cat$ for the cartesian product.

\item We write $\ast$ for the final space.

\item We write $\bD^1$ for the walking arrow, the strict category with two objects and a unique non-identity arrow.

\item We write $\partial\bD^1$ and $S^0$ for the maximal subspace in $\bD^1$, the set with two elements.

\end{itemize}

\vspace{1mm}

We indicate $\infty$-categories of large objects by $\widehat{(-)}$, for example we write $\widehat{\mS}, \widehat{\Cat}$ for the $\infty$-categories of large spaces, large $\infty$-categories.

We often call a fully faithful functor $\mC \to \mD$ an embedding.
We call a functor $\mC \to \mD$ an inclusion if it induces an embedding on maximal subspaces and on all mapping spaces. The latter is equivalent to ask that for any category $\mB$ the induced map
$\Map_{\Cat}(\mB,\mC) \to \Map_{\Cat}(\mB,\mD)$ is an embedding.

\section{Higher categories with adjoints}

\subsection{Higher categories}

In this section we recall the basic higher category theory needed for our work.

\subsubsection{Enriched categories}

To define higher categories, we recall the theory of (homotopy coherent) enriched categories \cite{GEPNER2015575}, \cite{Heine2023AnEB}, \cite{heine2025equivalence}, \cite{HINICH2020107129}.
For every presentably monoidal category $\mV$ there is a presentable 2-category $ \mV\mathrm{-}\Cat$ of $\mV$-enriched categories and $\mV$-enriched functors and a forgetful functor $ \iota:\mV\mathrm{-}\Cat \to \Cat$ to the presentable 2-category $\Cat$ of categories.
For $\mV= \mS$ the category of homotopy types, the forgetful functor $ \iota:\mS\mathrm{-}\Cat \to \Cat$ is an equivalence by \cite[Corollary 3.23.]{heine2024bienrichedinftycategories}.

\begin{notation}Let $\mV$ be a presentably monoidal category.
A $\mV$-enriched category $\mC$ provides an underlying category $\iota(\mC)$, for all objects $X,Y \in \iota(\mC)$ a morphism object 
$$\Mor_\mC(X, Y) \in \mV$$
and for all objects $X,Y,Z \in \iota(\mC)$ a composition morphism in $\mV:$
$$\Mor_\mC(Y, Z) \ot \Mor_\mC(X, Y) \to \Mor_\mC(X, Z).$$

We write $ X \in \mC$ for $ X \in \iota(\mC)$ and usually notationally identify $\mC$ with $\iota(\mC).$

\end{notation}

\begin{notation}
Let $\mV$ be a presentably monoidal category and $\mC, \mD$ be $\mV$-enriched categories. Let $ \mV \mathrm{-}\Fun(\mC,\mD)$ denote the category of $\mV$-enriched functors $\mC \to \mD$ using that $\mV \mathrm{-}\Cat$ is a 2-category.

\end{notation}

For every enriched category there is an opposite one by \cite[Proposition 4.60.]{heine2024bienrichedinftycategories}:

\begin{proposition}
There is an involution $$(-)^\circ: \mV\mathrm{-}\Cat \simeq \mV^\rev\mathrm{-}\Cat$$ forming the opposite enriched category.

For every $\mC \in \mV\mathrm{-}\Cat$ and $X,Y \in \mC$
there are canonical equivalences
$$ \iota(\mC^\circ) \simeq \iota(\mC)^\op$$
and $$ \Mor_{\mC^\circ}(X,Y) \simeq \Mor_{\mC}(Y,X).$$
\end{proposition}

The following is \cite[Definition 2.51.]{heine2024higheralgebraweightedcolimits}:

\begin{definition}Let $\mV$ be a presentably monoidal category, $\mC$ a $\mV$-enriched category and $X \in \mC, V \in \mV.$	 
	
\begin{enumerate}
	
\item The tensor of $V$ and $X$ in $\mC$ is the object $V \ot X \in \mC $ such that there is a morphism
$V \to \Mor_\mC(X, V \ot X) $ in $\mV$ that induces for every $Y \in \mC$ an equivalence
$$ \Mor_\mC(V \ot X,Y) \to \Mor_\mV(V, \Mor_\mC(X,Y)). $$ 

\item The cotensor of $V$ and $X$ in $\mC$ is the object ${^V X} \in \mC $ that is the tensor of $V $ and $X$ in the opposite $\mV^\rev$-enriched category $\mC^\circ.$

\end{enumerate}
	
\end{definition}

\subsubsection{Higher categories}

Next we recall the basic theory of $\infty$-categories following \cite[2.4.]{gepner2026oriented}, \cite[2.5.]{gepner2025oriented}. 

\begin{definition}
For every $\n \geq 0$ we inductively define the presentable cartesian closed category $\n\Cat$ of small (univalent) $\n$-categories by setting:
$$0\Cat :=\mS,\ (\n+1)\Cat:= \n\Cat\mathrm{-}\Cat.$$
	
\end{definition}

\begin{notation}
For every $\n \geq 0$ we inductively define colocalizations
$\n\Cat \rightleftarrows (\n+1)\Cat: \iota_\n,$
where both adjoints  preserve finite products and filtered colimits.
Let
$$0\Cat= \mS \rightleftarrows 1\Cat = \mS\mathrm{-}\Cat : \iota_0 $$ be the unique colocalization whose left adjoint preserves finite products. Let $$(\n+1)\Cat= \n\Cat\mathrm{-}\Cat \rightleftarrows (\n+2)\Cat= (\n+1)\Cat \mathrm{-}\Cat:\iota_{\n+1}:= (\iota_\n)_!. $$

\end{notation}

\begin{definition}The presentable cartesian closed category $\infty\Cat$ of small (univalent) $\infty$-categories is the limit
$$\infty\Cat:= \lim(...\xrightarrow{\iota_{\n}} \n\Cat \xrightarrow{\iota_{\n-1}} ... \xrightarrow{\iota_0} 0 \Cat) $$
of presentable cartesian closed categories and right adjoint functors. 
	
\end{definition}

\begin{notation}

Since $\infty\Cat$ is cartesian closed, for every $\infty$-category $\mC$ the functor $$ \mC \times (-): \infty\Cat \to \infty\Cat$$ admits a right adjoint $\Fun(\mC,-): \infty\Cat \to \infty\Cat$.
\end{notation}

Since forming category objects preserves limits, we obtain the following,
which is \cite[Proposition 2.5.4.]{gepner2025oriented}:

\begin{remark}\label{fix}
	
There is a canonical equivalence
$ \infty\Cat \simeq \infty\Cat\mathrm{-}\Cat. $
	
\end{remark}

\begin{notation}\label{mormoro}

Let $\Mor: \infty\Cat_{\partial\bD^1/} \to \infty\Cat$
be the functor $$\infty\Cat_{\partial\bD^1/} \simeq \Cat_{\partial\bD^1/} \times_{\Cat} \infty\Cat\mathrm{-}\Cat \to \infty\Cat $$ sending $(\mC,X,Y)$ to $\Mor_\mC(X,Y).$

\end{notation}

\begin{remark}\label{homfil}
The functor $\Mor: \infty\Cat_{\partial\bD^1/} \to \infty\Cat$ preserves small limits and small filtered colimits.

\end{remark}

\vspace{1mm}

By \cite[Definition 3.4.2.]{gepner2025oriented} there is an $\infty$-category satisfying the following:  

\begin{notation}\label{suspi}

The functor $\Mor: \infty\Cat_{\partial\bD^1/} \to \infty\Cat$
of (\cref{mormoro}) admits a left adjoint $S: \infty\Cat \to \infty\Cat_{\partial\bD^1/}$, the categorical suspension.

Let $\mC$ be an $\infty$-category. The $\infty$-category $S(\mC)$ has two objects $0,1$ and morphism $\infty$-categories: $$\Mor_{S(\mC)}(1,0)\simeq \emptyset, \ \Mor_{S(\mC)}(0,1)\simeq \mC, \ \Mor_{S(\mC)}(0,0)\simeq \Mor_{S(\mC)}(1,1) \simeq *. $$

\end{notation}

\begin{definition}Let $\n \geq 0$.
The $\n$-disk (or walking $\n$-morphism) is $\bD^\n:= S^{n}(*).$
	
\end{definition}

\begin{definition}Let $\n \geq 0$.
The boundary of the $\n$-disk is $\partial\bD^\n:= S^{n}(\emptyset).$
	
\end{definition}

\begin{remark}
The functor $\emptyset \subset *$ induces inclusions $\partial\bD^\n \subset \bD^\n$ for every $\n \geq 0.$
	
\end{remark}

\begin{example}
Then $\partial\bD^0=\ast, \partial\bD^1=S(\emptyset)=*\coprod*$ is the set with two elements.	
\end{example}

\begin{notation}
For every $0 \leq \n \leq \m$ the left adjoint embeddings $\n\Cat \leftrightarrows \m\Cat$ preserve small limits and thus induce a left adjoint embedding $\n\Cat \leftrightarrows \infty\Cat: \iota_\n$ that preserves small limits and so admits a left adjoint $$\tau_\n: \infty\Cat \to \n\Cat$$ by presentability.

\end{notation}

By the limit definition of $\infty\Cat$ we have the following filtration \cite[Lemma 2.4.17.]{gepner2026oriented}:

\begin{lemma}\label{decom}
Let $\mC$ be an $\infty$-category.
The sequential diagram $$\iota_0(\mC) \to ... \to \iota_\n(\mC) \to \iota_{\n+1}(\mC) \to ... \to \mC$$ exhibits $\mC$ as the colimit in $\infty\Cat$ of the diagram $\iota_0(\mC) \to ... \to \iota_\n(\mC) \to \iota_{\n+1}(\mC) \to ....$
\end{lemma}

\begin{definition}\label{ruik} 

Let $\n \geq 0.$
We inductively define involutions $(-)^\op_\n, (-)^\co_\n$ of $ \n\Cat$ by setting
$(-)^\co_0, (-)^\op_0$ are the identities and
$$(-)^\co_{\n+1}: {(\n+1)}\Cat \xrightarrow{((-)^\op_\n)_!}{(\n+1)}\Cat, $$$$
(-)^\op_{\n+1}: {(\n+1)}\Cat \xrightarrow{(-)^\circ} {(\n+1)}\Cat \xrightarrow{((-)^\co_\n)_!}{(\n+1)}\Cat. $$

There are commutative squares
$$\begin{xy}
\xymatrix{
{(\n+1)}\Cat \ar[d]^{\iota_\n}  \ar[rr]^{(-)_{\n+1}^\op} && {(\n+1)}\Cat  \ar[d]^{\iota_\n}
\\ 
 \n\Cat \ar[rr]^{(-)_{n}^\op} &&  \n\Cat
}
\end{xy}
\begin{xy}
\xymatrix{
{(\n+1)}\Cat \ar[d]^{\iota_\n}  \ar[rr]^{(-)_{\n+1}^\co} && {(\n+1)}\Cat  \ar[d]^{\iota_\n}
\\ 
 \n\Cat \ar[rr]^{(-)_{n}^\co} &&  \n\Cat
}
\end{xy}
$$
Hence the latter involutions give rise to involutions 
\begin{equation}\label{eqinvo}
(-)^\op, (-)^\co: \infty\Cat \to \infty\Cat
\end{equation}
on the limit.
	
\end{definition}

Now we are ready to define higher categories with adjoints.

\begin{definition}
\begin{itemize}
\item An $\infty$-category $\mC$ admits left adjoints, right adjoints, adjoints for morphisms of dimension smaller 2 if every morphism in the underlying 2-category of $\mC$ admits a left adjoint,
right adjoint, both a left and a right adjoint, respectively.

\item Let $\bk \geq 2$.
An $\infty$-category $\mC$ admits (left) (right) adjoints for morphisms of dimension smaller $\bk$ if $\mC$ admits (left) (right) adjoints for morphisms of dimension smaller 2 and all morphism $\infty$-categories of $\mC$ admit (left) (right) adjoints for morphisms of dimension smaller $\bk-1$.

\item An $\infty$-category $\mC$ admits (left) (right) adjoints if it admits (left) (right) adjoints for morphisms of dimension smaller $\bk$ for every $\bk \geq 2$.

\end{itemize}	
	
\end{definition}

\begin{notation}
Let $$ \infty\Cat^\adj \subset \infty\Cat$$ be the full subcategory of small $\infty$-categories having adjoints.

\end{notation}

\begin{remark}Let $ \bk \geq 2.$
An $\infty$-category $\mC$ admits left (right) adjoints for morphisms of dimension smaller $\bk$ if and only if for every $\ell > \bk$ there is an $n > \ell$ such that the underlying $n$-category of $\mC$ admits left (right) adjoints of dimension smaller $\bk.$
An $\infty$-category $\mC$ admits left (right) adjoints if and only if for every $n \geq 2$ the underlying $n$-category of $\mC$ admits left (right) adjoints of dimension smaller $n.$

\end{remark}

\begin{example}\label{aadk}
Every full subcategory of an $\infty$-category having left (right) adjoints admits left (right) adjoints.
	
\end{example}

The following lemma is wellknown to experts and for instance proven in \cite[Proposition 4.4.]{Scholze2025}: 
We provide a proof for the reader's convenience.

\begin{lemma}\label{adjjj}

Let $\mC$ be an $\infty$-category, $f: X  \to Y, g : Y \to X$ morphisms in $\mC$ and $\eta: \id \to g \circ f$ and $\epsilon: f \circ g \to \id $ be 2-morphisms in $\mC$ that exhibit $f $ as left adjoint of $g.$

If $\eta: \id \to g \circ f$ and $\epsilon: f \circ g \to \id $ admit right adjoints $\eta^R, \epsilon^R$, then $\eta^R: g \circ f \to \id $ and $\epsilon^R: \id \to f \circ g $ exhibit $g: Y \to X$ as left adjoint of $f: X  \to Y $.

\end{lemma}

\begin{proof}

The composition $$ f \xrightarrow{f \circ \eta} f \circ g \circ f \xrightarrow{\epsilon \circ f} f $$ 
is left adjoint to the composition
$$ f \xrightarrow{\epsilon^R \circ f} f \circ g \circ f \xrightarrow{f \circ \eta^R} f. $$ 

The composition $$ g \xrightarrow{\eta \circ g} g \circ f \circ g \xrightarrow{g \circ \epsilon} g $$ 
is left adjoint to the composition
$$ g \xrightarrow{g \circ \epsilon^R} g \circ f \circ g \xrightarrow{\eta^R \circ g} g. $$ 

We conclude by the fact that adjoints are unique if they exist, and that the identity is left adjoint to the identity.

\end{proof}

\begin{corollary}

Let $\n \geq 2$ and $\mC$ an $\infty$-category such that every
$n+1$-morphism admits a right adjoint.
Then any right adjoint of an $n$-morphism is also a left adjoint.

Then any left adjoint of an $n$-morphism is also a right adjoint.

\end{corollary}

\begin{definition} By \cite[Definition 2.2.10., Variant 2.2.11.]{lurie2008classification} for every $n \geq 0$ there is an $n$-category $\Bord_n$
whose objects are $0$-dimensional oriented framed manifolds, whose $i$-morphisms for $ 1 \leq \bi \leq n$ are $i$-dimensional oriented framed cobordisms whose $n+1$-morphisms are diffeomorphisms between $n$-dimensional oriented framed cobordisms and whose
$i$-morphisms for $ i \geq n+2 $ are smooth homotopies.
Moreover there is a functor $\Bord_\n  \to \Bord_{\n+1}$ \cite[Example 3.3.12.]{lurie2008classification}.
Let	$$\mathrm{Bord}_\infty$$ be the colimit in $\infty\Cat$ of the following sequential diagram:
$$\Bord_1 \to ... \to \Bord_\n \to ...$$

\end{definition}

\begin{example}
For every $2 \leq n \leq \infty $ the $n$-category $\Bord_\n$ admits adjoints for morphisms of dimension smaller $n$ by \cite[Example 2.3.23.]{lurie2008classification}.
	
\end{example}

\subsubsection{Fibrations}

Next we introduce fibrations of higher categories following \cite{gepner2026fibrations}:

\begin{notation}Let $\phi: \mC \to \mD$ be a functor and $X,Y \in \mC.$ 
We write
\begin{align*}
\phi_{X,Y}:\Mor_\mC(X,Y)\to\Mor_\mD(\phi(X),\phi(Y))
\end{align*}
for the induced functor on morphism $\infty$-categories.
\end{notation}

\begin{definition}Let $1 \leq \n \leq \infty$ and $\phi: \mC \to \mD$ a functor of $\infty$-categories.

\begin{enumerate}[\normalfont(1)]\setlength{\itemsep}{-2pt}

\item A 1-morphism $f: X \to Y$ in $\mC$ is $\phi$-cocartesian if for every $Z\in \mC$ the commutative square 
\begin{equation}\label{filler}
\begin{xy}
\xymatrix{
\Mor_\mC(Y,Z) \ar[d] \ar[r]
& \Mor_\mC(X,Z) \ar[d]^\phi
\\ 
\Mor_\mD(\phi(Y), \phi(Z)) \ar[r] & \Mor_\mD(\phi(X), \phi(Z))
}
\end{xy}\end{equation}
is a pullback square.

\item An $n$-morphism $\alpha:\bD^n\to\mC$ is $\phi$-cocartesian if
for every pair of morphisms $(X \to \alpha(0), \alpha(1) \to Y)$ in $\mC$ the composite $n-1$-morphism
\[
\bD^{n-1}\to\Mor_\mC(\alpha(0),\alpha(1)) \to \Mor_\mC(X,Y)
\]
is $\phi_{X,Y}$-cocartesian.
\end{enumerate}
\end{definition}

\begin{example}Let $\phi: \mC \to \mD$ be a functor and $n \geq 1$.
Every $n$-morphism in $\mC$ that is an equivalence, is $\phi$-cocartesian.
An $n$-morphism in $\mC$ is an equivalence if and only if it is 
cocartesian for the functor $\mC \to \bD^0$.
    
\end{example}

\begin{definition}
A functor $\phi: \mC \to \mD$ is an anticocartesian fibration if for every $k \leq 1$ every commutative square 
\begin{equation}\label{filler2}
\begin{xy}
\xymatrix{
\bD^{k-1} \ar[d] \ar[r]
& \mC \ar[d]^\phi
\\ 
\bD^k \ar[r] & \mD
}
\end{xy}\end{equation}
admits a filler by a $\phi$-cocartesian $k$-morphism,
where the left vertical functor is the source inclusion.

\end{definition}

\begin{definition}
A functor $\phi: \mC \to \mD$ is a cocartesian fibration if $\phi^\co$ is an anticocartesian fibration.

\end{definition}

\begin{definition}Let $\sigma$ be a commutative square
\begin{equation}\label{sqco}
\begin{xy}
\xymatrix{
X \ar[d]^\psi \ar[r]^\kappa
& Y \ar[d]^\phi
\\ 
S \ar[r]^\rho & T
}
\end{xy}\end{equation}

\begin{enumerate}[\normalfont(1)]\setlength{\itemsep}{-2pt}
\item A commutative square \ref{sqco} is a map of anticocartesian fibrations if $\psi, \phi$ are anticocartesian fibrations and $\kappa$ sends $\psi$-cocartesian morphisms to $\phi$-cocartesian morphisms.

\vspace{1mm}
\item A commutative square \ref{sqco}
is a map of cocartesian fibrations if $\sigma^\co$ is a map of anticocartesian fibrations.

\vspace{1mm}

\item A map of (anti)cocartesian fibrations over $T$ is a map of (anti)cocartesian fibrations such that $\rho$ is the identity.

\end{enumerate}

\end{definition}

\begin{notation}\label{notfib1}

Let $$\coCart$$ be the respective subcategory of cocartesian fibrations and maps of such. 

\end{notation}

\subsection{The Gray tensor product}

Next we construct the Gray tensor product following \cite[3.5.]{gepner2026oriented}, \cite[3.1.]{gepner2025oriented}.

\begin{notation}

By \cite[Definition 3.5.19.]{gepner2025oriented} for every $\n \geq 0 $ there is an $\n$-category $\cube^\n$,
the oriented $\n$-cube, whose 1-truncation is the $\n$-fold product $(\bD^1)^{\times \n}.$
	
Let $$\cube \subset \infty\Cat $$ be the full subcategory of oriented cubes.
	
\end{notation}

\vspace{1mm}
By \cite[Remark 3.5.18.]{gepner2025oriented} the category $\cube$ carries a monoidal structure $\boxtimes$ whose tensor unit is $\cube^0$ and such that for every $\n,\m \geq 0$ there is an equivalence $\cube^\n \boxtimes \cube^\m \simeq \cube^{\n+\m}.$

\vspace{1mm}

The next is \cite[Theorem 3.14., Example 3.16.]{campion2023gray}
and \cite[Corollary 3.4.2.]{gepner2026oriented}:

\begin{theorem}\label{dense}
	
There is a unique presentably monoidal structure on $\infty\Cat$ such that the embedding
$$\cube \subset \infty\Cat$$ is monoidal.

\end{theorem}

We call the monoidal structure on $\infty\Cat$ of Theorem \ref{dense}
the Gray monoidal structure and write $\boxtimes$ for the Gray tensor product.

\begin{notation}

Since the Gray tensor product defines a presentably monoidal structure on $\infty\Cat$, it is biclosed: for every $\infty$-category $\mC$ the functor $$ \mC \boxtimes (-): \infty\Cat \to \infty\Cat$$ admits a right adjoint $\Fun^\lax(\mC,-): \infty\Cat \to \infty\Cat$ and the functor $$ (-) \boxtimes \mC : \infty\Cat \to \infty\Cat$$ admits a right adjoint $ \Fun^\oplax(\mC,-): \infty\Cat \to \infty\Cat.$
\end{notation}

\begin{definition}Let $\F,\G:\mC \to \mD $ be functors of $\infty$-categories.
An (op)lax natural transformation $\F \to \G$ is a morphism in $\Fun^{(\op)\lax}(\mC,\mD)$.

\end{definition}

\begin{remark}\label{lao}

If $\mC$ is an $\n$-category and $\mD$ an $\m$-category for $\n,\m \geq 0$,
then $\mC \boxtimes \mD$ is an $\n+\m$-category.
This holds since $\n\Cat$ is closed under small colimits in $\infty\Cat$ and $\n\Cat$ is generated under small colimits by the oriented $\ell$-cubes for $1 \leq \ell \leq \n$. 
\end{remark}

The following is \cite[Remark 3.12.7.]{gepner2025oriented}:

\begin{remark}\label{grayspace}

The full subcategory $\mS \subset \infty\Cat$ is closed under the Gray tensor product and the restricted Gray monoidal structure on $\mS$ is the cartesian structure.
The left adjoint $\tau_0: \infty\Cat \to \mS$ of the resulting monoidal embedding $\mS \subset \infty\Cat$ is monoidal.

\end{remark}

The next is \cite[Theorem 3.7.2.]{gepner2026oriented}:
\begin{proposition}\label{thas2}

For every $\infty$-category $\mC$ there is a canonical equivalence
$$ S(\mC) \simeq  \mC \boxtimes \bD^1 +_{ \mC \boxtimes \partial\bD^1} \partial\bD^1.$$

\end{proposition}

\begin{corollary}\label{hom}
	
For every $\infty$-category $\mC$ and $X,Y \in \mC$ there is a canonical 
equivalence
$$ \Mor_\mC(X,Y) \simeq \{X \} \times_\mC \Fun^\oplax(\bD^1,\mC) \times_\mC \{Y\}.$$

\end{corollary}

The next is \cite[Proposition 3.4.8.]{gepner2026oriented}:

\begin{proposition}\label{dua}
	
The involutions (\ref{eqinvo}) underlie monoidal involutions
$$(-)^\op, (-)^\co: (\infty\Cat, \boxtimes^\rev) \simeq (\infty\Cat, \boxtimes).$$
	
\end{proposition}

\begin{corollary}\label{grayhoms}
Let $\mC,\mD \in \infty\Cat$.
There are canonical equivalences $$ \Fun^\oplax(\mC,\mD)^\op \simeq \Fun^\lax(\mC^\op,\mD^\op), \ \Fun^\oplax(\mC,\mD)^\co \simeq \Fun^\lax(\mC^\co,\mD^\co).$$

\end{corollary}

Next we consider a variant of the Gray tensor product for $\infty$-categories with distinguished object that plays the role of the smash product in homotopy theory.

\begin{definition}
The category $\infty\Cat_*$ of small $\infty$-categories with distinguished object is the full subcategory of $\Fun(\bD^1, \infty\Cat)$ spanned by the functors whose source is final.
	
\end{definition}

\begin{definition}

The Gray smash monoidal structure on $\infty\Cat_{*}$,
denoted by $\wedge$, is the smash monoidal structure of \cite[Lemma 3.69.]{gepner2025oriented}
applied to $(\infty\Cat, \boxtimes).$
For every $X, Y \in \infty\Cat_*$ the Gray smash product $X \wedge Y$ is the cofiber of the canonical functor $$ X \vee Y \to X \boxtimes Y. $$

\end{definition}

\begin{notation}
Let $\mC,\mD \in \infty\Cat_*.$
Let $ \Fun_*^{(\op)\lax}(\mC,\mD) $ be the fiber of the functor $$ \Fun^{(\op)\lax}(\mC,\mD) \to \Fun^{(\op)\lax}(\ast,\mD)\simeq \mD $$ over the base point.

\end{notation}

\cite[Lemma 2.3.11.]{gepner2025oriented} gives the following:

\begin{lemma}
Let $\mC \in \infty\Cat_*$. 
\begin{enumerate}
\item The functor $$ \mC \wedge (-): \infty\Cat_* \to \infty\Cat_*$$ is left adjoint to the functor $\Fun_*^\lax(\mC,-): \infty\Cat_* \to \infty\Cat_*$. 

\vspace{1mm}
\item The functor $$ (-) \wedge \mC : \infty\Cat_* \to \infty\Cat_*$$ is left adjoint to the functor $\Fun_*^\oplax(\mC,-): \infty\Cat_* \to \infty\Cat_*.$

\end{enumerate}

\end{lemma}

By functoriality of the smash monoidal structure, \cref{dua} gives the following:

\begin{corollary}\label{dua2}
	
The involutions $$(-)^\op, (-)^\co: \infty\Cat_* \to \infty\Cat_* $$
induced by the involutions (\ref{eqinvo}) underlie monoidal involutions
$$(-)^\op, (-)^\co: (\infty\Cat_*, \wedge^\rev) \simeq (\infty\Cat_*, \wedge).$$
	
\end{corollary}

\subsection{Categorical spheres}

Next we define categorical spheres.

\begin{notation}Let $\mC\in \infty\Cat_*$.
The reduced categorical suspension of $\mC$ is
$$\Sigma(\mC):= S(\mC) +_{\bD^1} \bD^0.$$ 	
	
\end{notation}

We often call the reduced categorical suspension simply suspension.

\begin{definition}
Let $\Omega: \infty\Cat_{*} \to \infty\Cat$
be the composition $$ \infty\Cat_{*} \to \infty\Cat_{\partial\bD^1/} \xrightarrow{\Mor}\infty\Cat, \ (\mC,X) \mapsto \Mor_\mC(X,X), $$
where the first functor is induced by the functor $\partial\bD^1 \to *.$

The functor $\Omega$ preserves the final object and so lifts to a functor $\Omega: \infty\Cat_{*} \to \infty\Cat_*.$

\end{definition}

\begin{remark}

There is an adjunction $$ \Sigma : \infty\Cat_* \rightleftarrows \infty\Cat_*: \Omega.$$

\end{remark}

\cref{homfil} implies the following:

\begin{remark}\label{endfil}
The functor $\Omega: \infty\Cat_{*} \to \infty\Cat_*$
preserves small filtered colimits.

\end{remark}

\begin{definition}\label{conn} Let $\n \geq 0$.
An $\infty$-category $\mC$ is $\n$-connected if the underlying space of $\mC$ is $\n$-connected and all morphism $\infty$-categories are $\n$-1-connected. Every $\infty$-category is -1-connected. 

\end{definition}

\begin{remark}
Let $\n \geq 0$. By induction an $\infty$-category is $\n$-connected if it is $\n$+1-connected.

\end{remark}

\cite[Theorem 6.3.2.]{GEPNER2015575} implies the following proposition:

\begin{proposition}\label{hhnl}
The functor $\Omega: \infty\Cat_* \to \infty\Cat_*$ lifts to a
functor $$\Omega: \infty\Cat_* \to \Mon(\infty\Cat),$$ 
which admits a fully faithful left adjoint $$B:\Mon(\infty\Cat) \to \infty\Cat_* $$ that preserves finite products.
The essential image of $B$ precisely consists of the connected $\infty$-categories with distinguished object.

\end{proposition}

\begin{notation}
For every $\n \geq 0$ let $ \Mon_{\bE_\n}(\infty\Cat)$
be the category of $\bE_\n$-monoids \cite[Definition 5.1.0.4.]{lurie.higheralgebra} in $\infty\Cat$.

\end{notation}

\begin{remark}

By \cite[Theorem 5.1.2.2]{lurie.higheralgebra} for every $\n \geq 1$ there is a canonical equivalence $$ \Mon_{\bE_\n}(\infty\Cat) \simeq \Mon_{\bE_{\n-1}}(\Mon(\infty\Cat)).$$	

\end{remark}

The following is \cite[Corollary 4.1.23.]{heine2026stable}:

\begin{corollary}\label{deloop} Let $\n \geq 1.$ There is an adjunction
$$B^\n: \Mon_{\bE_\n}(\infty\Cat) \rightleftarrows \infty\Cat_*: \Omega^\n$$
whose right adjoint lifts $\Omega^\n: \infty\Cat_* \to \infty\Cat_*$ and whose left adjoint preserves finite products and induces an equivalence to the full subcategory of $\n-1$-connected $\infty$-categories with distinguished object.

\end{corollary}

\begin{definition} 
Let $\n \geq 0$. The categorical $\n$-sphere $S^\n$ is the cofiber of the functor $\partial\bD^\n \subset \bD^\n$ in $\infty\Cat$.
\end{definition}

\begin{remark}By definition the categorical $\n$-sphere is an $\n$-category with distinguished object.

\end{remark}

\begin{notation}Let $\n \geq 0$. Let $$ \Free_{\bE_\n}: \infty\Cat \to \Mon_{\bE_\n}(\infty\Cat)$$ be the free functor.
 
\end{notation}

The following is \cite[Proposition 4.1.28.]{heine2026stable}:

\begin{proposition}\label{freesp} Let $\n \geq 0$. There is an equivalence of $\n$-categories with distinguished object $$S^\n \simeq B^\n \Free_{\bE_\n}(*).$$ In particular, there is an equivalence $S^1 \simeq B\bN$ of categories with distinguished object. 

\end{proposition}

The following is \cite[Corollary 4.1.29.]{heine2026stable}:

\begin{corollary}\label{autoeq} Let $\n \geq \m \geq 0.$

\begin{enumerate}
\item There is a canonical isomorphism of monoids $$\pi_0(\Map_{\infty\Cat_*}(S^\n,S^\n)) \cong \bN.$$

\item There is a canonical isomorphism $$\pi_0(\Map_{\infty\Cat_*}(S^\n,S^\m)) \cong *.$$

\item There is a canonical isomorphism $$\pi_0(\Map_{\infty\Cat_*}(S^\m,S^\n)) \cong \pi_{\n-\m}(\Free_{\bE_\m}(*)).$$

\item There is a canonical isomorphism of monoids $$\pi_0(\Map_{\iota_0\infty\Cat_*}(S^\n,S^\n)) \cong \bN^\times \cong 0.$$
\end{enumerate}	

\end{corollary}

The following is \cite[Lemma 4.1.30.]{heine2026stable}:

\begin{lemma}\label{smarem}\label{uuupo}
Let $\mC\in \infty\Cat_*$.
There is a canonical equivalence
$$\Sigma(\mC) \simeq \mC \wedge S^1.$$  

\end{lemma}

The following is \cite[Corollary 4.1.31.]{heine2026stable}:

\begin{corollary}\label{freespar} Let $\n \geq 0$.	
There is a canonical equivalence of $\n$-categories with distinguished object $$(-)_+\wedge (S^1)^{\wedge n} \simeq B^\n \circ \Free_{\bE_\n}.$$

\end{corollary}

\begin{corollary}\label{freespart} Let $\n \geq 0$.	
There is a canonical equivalence of $\n$-categories with distinguished object $$(S^1)^{\wedge n} \simeq B^\n \Free_{\bE_\n}(*).$$

\end{corollary}

\cref{freesp} implies the following:

\begin{corollary}Let $\n \geq 0$. There is an  equivalence of $\n$-categories $(S^1)^{\wedge \n} \simeq S^\n$.

\end{corollary}	

\subsection{Oriented fiber sequences}

\subsubsection{Oriented categories}

Next we consider categories enriched in the Gray tensor product following \cite{gepner2025oriented}. 

\begin{definition}\emph{}
\begin{enumerate}
\item An oriented category is a category enriched in $(\infty\Cat,\boxtimes^\rev).$

\item An oriented functor is a functor enriched in $(\infty\Cat,\boxtimes^\rev).$

\item An antioriented category is a category enriched in $(\infty\Cat,\boxtimes).$

\item An antioriented functor is a functor enriched in $(\infty\Cat,\boxtimes).$

\item A bioriented category is a category enriched in $(\infty\Cat,\boxtimes) \ot (\infty\Cat,\boxtimes^\rev).$

\item A bioriented functor is a functor enriched in $(\infty\Cat,\boxtimes) \ot (\infty\Cat,\boxtimes^\rev).$

\end{enumerate}

\end{definition}

\begin{example}

The Gray monoidal structure on $\infty\Cat$ is biclosed and so exhibits $\infty\Cat$ as enriched in $(\infty\Cat,\boxtimes) \ot (\infty\Cat,\boxtimes^\rev).$ This way we see
$\infty\Cat$ as a large bioriented category, which we denote by the same name.

\end{example}

\begin{notation}\emph{}

\begin{itemize}
\item Let $\mC$ be an oriented category and $X,Y \in \mC$.
We write $\R\Mor_\mC(X,Y)$ for the $\infty$-category of morphisms $X \to Y$ in $\mC$.

\item Let $\mC$ be an antioriented category and $X,Y \in \mC$.
We write $\L\Mor_\mC(X,Y)$ for the $\infty$-category of morphisms $X \to Y$ in $\mC$.

\item Let $\mC$ be a bioriented category and $X,Y \in \mC$.
We write $\L\Mor_\mC(X,Y)$ for the $\infty$-category of morphisms $X \to Y$ in the underlying antioriented category of $\mC$ and 
$\R\Mor_\mC(X,Y)$ for the $\infty$-category of morphisms $X \to Y$ in the underlying oriented category of $\mC$.
    
\end{itemize}
    
\end{notation}

We refer to adjunctions of oriented, antioriented, bioriented categories as oriented, antioriented, bioriented adjunctions.

\begin{definition}\emph{}
\begin{enumerate}

\item A weakly reduced oriented category is a category enriched in $(\infty\Cat_*,\wedge^\rev).$

\item A weakly reduced oriented functor is a functor enriched in $(\infty\Cat_*,\wedge^\rev).$

\item A weakly reduced antioriented category is a category enriched in $(\infty\Cat_*,\wedge).$

\item A weakly reduced antioriented functor is a functor enriched in $(\infty\Cat_*,\wedge).$

\item A weakly reduced bioriented category is a category enriched in $(\infty\Cat_*,\wedge) \ot (\infty\Cat_*,\wedge^\rev).$

\item A weakly reduced bioriented functor is a functor enriched in $(\infty\Cat_*,\wedge) \ot (\infty\Cat_*,\wedge^\rev).$

\end{enumerate}

\end{definition}

\begin{example}

The Gray smash monoidal structure on $\infty\Cat_*$ is biclosed and so exhibits $\infty\Cat_*$ as enriched in $(\infty\Cat_*,\wedge) \ot (\infty\Cat_*,\wedge^\rev).$ This way we see
$\infty\Cat_*$ as a large weakly reduced bioriented category, which we denote by the same name.

\end{example}

We refer to adjunctions of weakly reduced oriented, antioriented, bioriented categories as weakly reduced oriented, antioriented, bioriented adjunctions.

\begin{notation}We fix the following notation:

\begin{itemize}
\item Let $\mC$ be an antioriented category, $X \in \mC$ and $K \in \infty\Cat.$
We refer to (co)tensors of $K$ and $X$ by left (co)tensors and write $K \ot X$ for the left tensor and ${^K X}$
for the left cotensor of $K$ and $X$.

\vspace{1mm}

\item Let $\mC$ be an oriented category, $X \in \mC$ and $K \in \infty\Cat.$
We refer to (co)tensors of $K$ and $X$ by right (co)tensors and write $X \ot K$ for the right tensor and ${X^K}$ for the right cotensor of $K$ and $X$.

\vspace{1mm}

\item Let $\mC$ be a bioriented category, $X \in \mC$ and $K \in \infty\Cat.$
We refer to (co)tensors of $K \ot * $ and $X$ by left (co)tensors and write $K \ot X$ for the left tensor and ${^K X}$ for the left cotensor of $K$ and $X$.
We refer to (co)tensors of $* \ot K $ and $X$ by right (co)tensors and write $X \ot K$ for the right tensor and ${X^K}$ for the right cotensor of $K$ and $X$.

\end{itemize}
	
\end{notation}

\begin{notation}We fix the following notation:

\begin{itemize}

\item Let $\mC$ be a weakly reduced antioriented category, $X \in \mC$ and $K \in \infty\Cat_*.$
We refer to (co)tensors of $K$ and $X$ by left (co)tensors and write $K \wedge X$ for the left tensor and ${^K X_*}$ for the left cotensor of $K$ and $X$.

\vspace{1mm}

\item Let $\mC$ be a weakly reduced oriented category, $X \in \mC$ and $K \in \infty\Cat_*.$
We refer to (co)tensors of $K$ and $X$ by right (co)tensors and write $X \wedge K$ for the right tensor and ${X^K_*}$ for the right cotensor of $K$ and $X$.

\vspace{1mm}

\item Let $\mC$ be a weakly reduced bioriented category, $X \in \mC$ and $K \in \infty\Cat_*.$
We refer to (co)tensors of $K \ot * $ and $X$ by left (co)tensors and write $K \wedge X$ for the left tensor and ${^K X_*}$
for the left cotensor of $K$ and $X$.
We refer to (co)tensors of $* \ot K $ and $X$ by right (co)tensors and write $X \wedge K$ for the right tensor and ${X^K_*}$ for the right cotensor of $K$ and $X$.

\end{itemize}
	
\end{notation}

An initial object, final object, zero object of a (weakly reduced) oriented, antioriented, bioriented category is an initial object, final object, zero object in the enriched sense.

\begin{definition}
An oriented, antioriented, bioriented category is reduced if it admits a zero object.
An oriented, antioriented, bioriented functor is reduced if it preserves the zero object.
\end{definition}

\vspace{1mm}

By \cite[Proposition 3.8.]{gepner2025oriented} we 

can identify reduced oriented, antioriented, bioriented categories with weakly reduced 
oriented, antioriented, bioriented categories that admit an initial or final object. 
Similarly, we identify reduced oriented, antioriented, bioriented functors between reduced oriented, antioriented, bioriented categories with weakly reduced oriented, antioriented, bioriented functors. 

\subsubsection{Opposite and conjugate oriented categories}

\begin{notation}
The equivalences of \cref{dua} give rise to the following equivalences
$$(-)^\co:= (-)^\op_!: {\boxtimes\Cat} \simeq {\Cat\boxtimes}, $$$$ (-)^\co:= (-)^\op_!: {\Cat\boxtimes} \simeq {\boxtimes\Cat}, $$
$$ (-)^\co:= ((-)^\op, (-)^\op)_! : {\boxtimes\Cat\boxtimes} \simeq {\boxtimes\Cat\boxtimes},$$
where the first two equivalences are inverse to each other and the third equivalence is an involution.
\end{notation}

\begin{notation}
Let $$(-)^\circ: {\Cat\boxtimes} \simeq {\boxtimes\Cat}, $$$$ (-)^\circ: {\boxtimes\Cat} \simeq {\Cat\boxtimes} $$$$ (-)^\circ:  {\boxtimes\Cat\boxtimes} \simeq {\boxtimes\Cat\boxtimes}$$ be the opposite enriched category involutions.	

\end{notation}

\begin{notation}
The equivalences of \cref{dua} give rise to the following equivalences
$$ (-)^\op:= (-)^\circ\circ (-)^\co_!: {\boxtimes\Cat} \simeq {\boxtimes\Cat}, $$$$ (-)^\op:= (-)^\circ\circ (-)^\co_!: {\Cat\boxtimes} \simeq {\Cat\boxtimes},$$
$$  (-)^\op :=(-)^\circ \circ ((-)^\co, (-)^\co)_!: {\boxtimes\Cat\boxtimes} \simeq {\boxtimes\Cat\boxtimes}, $$
where the first two equivalences are inverse to each other and the third equivalence is an involution.
\end{notation}

\begin{notation}

Moreover we set
$$ {(-)^{\co\op}}:=  {(-)^{\co}} \circ {(-)^{\op}} \simeq {(-)^{\op}} \circ {(-)^{\co}}:{\boxtimes\Cat} \simeq {\Cat\boxtimes}, {\Cat\boxtimes} \simeq {\boxtimes\Cat}, {\boxtimes\Cat\boxtimes} \simeq {\boxtimes\Cat\boxtimes},$$
$$ {(-)^{\cop}}:=  {(-)^{\co\op}} \circ {(-)^{\circ}} \simeq {(-)^{\circ}} \circ {(-)^{\co\op}} \simeq (-)^{\co\op}_!: {\boxtimes\Cat} \simeq {\boxtimes\Cat}, {\Cat\boxtimes} \simeq {\Cat\boxtimes},$$
$${^{\cop}(-)}:= ((-)^{\co\op}, \id)_!, \ (-)^\cop:= (\id, (-)^{\co\op})_!: {\boxtimes\Cat\boxtimes} \simeq {\boxtimes\Cat\boxtimes},$$
$${^{\cop}(-)}^\cop:=(-)^\cop \circ  {^{\cop}(-)} \simeq $$$$ {^{\cop}(-)}\circ (-)^\cop \simeq ((-)^{\co\op}, (-)^{\co\op})_! \simeq$$$$ {(-)^{\co\op}} \circ {(-)^{\circ}} \simeq {(-)^{\circ}} \circ {(-)^{\co\op}}: {\boxtimes\Cat\boxtimes} \simeq {\boxtimes\Cat\boxtimes}.$$

\end{notation}

The equivalences of \cref{dua} gives rise to the following equivalences of bioriented categories:
\begin{corollary}\label{duae}There are canonical equivalences of bioriented categories:
$$(-)^\op: \infty\Cat^\co \simeq \infty\Cat, $$$$ (-)^\co: \infty\Cat^\op \simeq \infty\Cat^\circ, $$$$ (-)^{\co\op} : {^\cop\infty\Cat^{\cop}} \simeq \infty\Cat. $$	

\end{corollary}	

Next we define oriented pushouts and oriented pullbacks following \cite[4.8.]{gepner2025oriented}.

\begin{definition}
Let $\mC$ be an oriented category.
The oriented pullback of a diagram $X \to Z\leftarrow Y$ in $\mC$ is a diagram
\[
\xymatrix{& W\ar[rd]\ar[ld] &\\
X\ar[rd] & \Longrightarrow & Y\ar[ld]\\
& Z &}
\]
in $\mC$ such that for all objects $T$ of $\mC$ the induced functor
\begin{equation}\label{indmapl}
\R\Mor_{\mC}(T,W)\to \R\Mor_{\mC}(T,X) \underset{\R\Mor_{\mC}(T,Z)}{\times}
\Fun^\oplax(\bD^1,\R\Mor_{\mC}(T,Z)) \underset{\R\Mor_{\mC}(T,Z)}{\times} \R\Mor_{\mC}(T,Y)
\end{equation}
is an equivalence in $\infty\Cat$.
We write $X\overset{\to}{\times}_Z Y$ or $ Y\overset{\leftarrow}{\times}_Z X$ for $W$ if it exists.

\end{definition}

\begin{definition}Let $\mC$ be an oriented category.
The oriented pushout of a diagram $X \leftarrow Z\to Y$ in $\mC$
is the oriented pullback in the oriented category $\mC^\op$ of the corresponding diagram.
We write $ X\overset{\to}{+}_Z Y $ or $Y\overset{\leftarrow}{+}_Z X$ for the oriented pushout.

\end{definition}

The next is \cite[Lemma 3.8.6.]{gepner2025oriented}:

\begin{lemma}\label{0desc}\label{adesc}
Let $\mC$ be an oriented category and $Z \to X, Z \to Y, X \to Z, Y \to Z$ morphisms in $\mC$.

\begin{enumerate}

\item If $\mC$ admits pushouts and right tensors with $\bD^1$, 
there is a canonical equivalence $$ X \overset{\to}{+}_Z Y \simeq X +_{\{0\}\otimes Z} (Z \ot \bD^1) +_{\{1\}\otimes Z} Y.$$	

\item If $\mC$ admits pullbacks and right cotensors with $\bD^1$, 
there is a canonical equivalence $$ X \overset{\to}{\times}_Z Y \simeq X \times_{Z^{\{0\}}} {Z^{\bD^1}} \times_{Z^{\{1\}}} Y.$$	
\end{enumerate}

\end{lemma}

\begin{definition}Let $\mC$ be a reduced oriented category.

\begin{enumerate}
\item Let $\phi: A \to B$ be a morphism in $\mC.$ The oriented left cofiber of $\phi$ is $ 0 \overset{\to}{+}_A B $.

\item Let $\phi: A \to B$ be a morphism in $\mC.$ The oriented right cofiber of $\phi$ is $ B \overset{\to}{+}_A 0 $.

\item Let $\phi: B \to A$ be a morphism in $\mC.$ 
The oriented left fiber of $\phi$ is $ 0 \overset{\to}{\times}_A B  $.

\item Let $\phi: B \to A$ be a morphism in $\mC.$ 
The oriented right fiber of $\phi$ is $ B \overset{\to}{\times}_A 0 $.

\end{enumerate} 
\end{definition}

\begin{notation}
Let $\bD^1$ be always equipped with the initial object.
Let $(\bD^1)^\op$ denote $\bD^1$ equipped with the final object.

\end{notation}

The following is \cite[Corollary 3.8.8.]{gepner2025oriented}:

\begin{corollary}\label{redfib}

\begin{enumerate}

\item Let $\mC$ be a reduced oriented category that admits pushouts and reduced right tensors with $\bD^1$ and $C \to A$ a morphism in $\mC$.
There are canonical equivalences $$ 0 \overset{\to}{+}_C A  \simeq  C \wedge \bD^1 +_{C \wedge S^0}A, \ A \overset{\to}{+}_C 0 \simeq A +_{C \wedge S^0} C \wedge (\bD^1)^\op .$$	

\item Let $\mC$ be a reduced oriented category that admits pullbacks and reduced right cotensors with $\bD^1$ and $A \to C$ a morphism in $\mC$.
There are canonical equivalences $$0 \overset{\to}{\times}_C A \simeq C^{\bD^1}_*\times_{C^{S^0}_*} A, \ A \overset{\to}{\times}_C 0  \simeq A \times_{C^{S^0}_*} C^{(\bD^1)^\op}_*. $$

\end{enumerate}

\end{corollary}

\begin{corollary}Let $\mC$ be a reduced oriented category and $X \in \mC$.
The oriented pushout $ 0 \overset{\rightarrow}{+}_C 0 $ and the right tensor $X \wedge S^1$
both satisfy the same universal property.

\end{corollary}

By \cite[Lemma 4.4.12.]{gepner2025oriented} there is the following pasting law:

\begin{lemma}\label{pasting}

Consider the following diagram in any oriented category $\mC$, where the left hand square is a commutative square:
\[
\begin{tikzcd}
\Q \ar{d} \ar{r} & \P \ar{r}{} \ar{d}[swap]{} & B  \ar{d}{} \\
\E \ar{r} & A \ar[double]{ur}{}  \ar{r}[swap]{} & C
\end{tikzcd}
\]
If the right hand square is an oriented pullback square, the left hand square is a pullback square if and only if the outer square is an oriented pullback square.

\end{lemma}

The following is \cite[Corollary 3.8.13.]{gepner2025oriented}:

\begin{proposition}\label{homs} 
Let $\F: \mA \to \mC,\G: \mB \to \mC$ be functors and $\A,\A'\in \mA, \B,\B' \in \mB$ and $\sigma: \F(\A)\to \G(\B), \sigma': \F(\A')\to \G(\B')$ morphisms. There is a canonical equivalence $$\Mor_{\mA\underset{\mC}{\bar{\vec{\times}}} \mB}((\A,\B, \sigma),(\A',\B', \sigma')) \simeq \Mor_\mA(\A,\A')\underset{\Mor_\mC(\F(\A),\G(\B'))}{\bar{\vec{\times}}} \Mor_\mB(\B,\B').$$
\end{proposition}

\cref{homs} implies the following:

\begin{corollary}\label{homst} Let $\F: \mA \to \mC,\G: \mB \to \mC$ be functors, $A,A'\in \mA, B,B' \in \mB$ and $\sigma: \F(A)\to \G(B), \sigma': \F(A')\to \G(B')$, $\alpha, \alpha': A \to A', \beta, \beta': B \to B'$ morphisms and $$\tau: \sigma' \circ \F(\alpha) \to \G(\beta)\circ \sigma, \ \tau': \sigma' \circ \F(\alpha') \to \G(\beta') \circ \sigma$$ be 2-morphisms in $\mC$.
There is a canonical equivalence
$$\Mor_{\Mor_{\mA\underset{\mC}{\bar{\vec{\times}}} \mB}((\A,\B, \sigma),(\A',\B', \sigma'))}((\alpha,\beta, \tau), (\alpha',\beta',\tau')) \simeq $$$$ 
\Mor_{\Mor_\mA(A,A')(\alpha,\alpha') \underset{\Mor_{\Mor_\mC(\F(A),\G(B'))}(\sigma'\F(\alpha),\G(\beta')\sigma)}{\bar{\vec{\times}}} \Mor_{\Mor_\mB(B,B')}(\beta,\beta')}$$
\end{corollary}

\subsection{The Gray tensor product descends to higher categories with adjoints}

\vspace{1mm}

Next we prove that the Gray tensor product descends to the full subcategory of $\infty$-categories with adjoints.

\vspace{1mm}

The next proposition is a generalization of \cite[Theorem 4.6.]{haugseng2020lax} and
\cite[Proposition 6.14.]{heine2017monadicity} from dimension 2 to $\infty.$

\begin{proposition}\label{adjjo}
Let $\mC$ be an $\infty$-category and 
\[
\begin{tikzcd}
A \ar{r}{\alpha} \ar{d}[swap]{f} & X\ar[double]{dl}{\sigma} \ar{d}{g} \\
B \ar{r}[swap]{\beta} & Y
\end{tikzcd}
\]
a morphism $\theta: f \to g$ in $\Fun^\lax(\bD^1,\mC)$.

\begin{enumerate}
\item If $\alpha, \beta$, $\sigma$ admit right adjoints, then $\theta$ admits a right adjoint in $\Fun^\lax(\bD^1,\mC)$.

\item If $\alpha, \beta$ admit left adjoints $\alpha', \beta'$, respectively, and the composition of 2-morphisms \begin{equation*}\label{ttt}
\zeta: \beta' g \to \beta' g \alpha \alpha' \to \beta' \beta f \alpha' \to f \alpha'\end{equation*} admits a left adjoint, then $\theta$ admits a left adjoint in $\Fun^\lax(\bD^1,\mC)$.
\end{enumerate}

\end{proposition}

\begin{proof}

In the following we will heavily apply \cref{homst}.
(1): Let $\alpha'$ be a right adjoint of
$\alpha$ and $\beta'$ a right adjoint of $\beta$ and $\sigma'$ a right adjoint of $\sigma.$
Let $\theta':g \to f$ be the lax commutative square
$$(\alpha',\beta', \tau: f \alpha' \to \beta' \beta f \alpha' \xrightarrow{\beta'\sigma'\alpha'} \beta' g \alpha \alpha' \to \beta' g)$$
and $\eta: \id \to \theta' \theta$ the object of  $$\Mor_{\Mor_{\Fun^\lax(\bD^1,\mC)}(f,f)}(\id,\theta'\theta)
\simeq\Mor_{\Mor_{\mC}(A,A)}(\id,\alpha'\alpha) \underset{\Mor_{\Mor_{\mC}(A,B)}(f, \beta'\beta f)}{\bar{\vec{\times}}} \Mor_{\Mor_{\mC}(B,B)}(\id, \beta'\beta)$$
corresponding to the three units.
We show that $\eta$ exhibits $\theta: f \to g$ as left adjoint to $\theta':g \to f$.
By \cite[Lemma 5.13.]{heine2017monadicity} it is enough to prove that for every morphism $h: Z \to T$ in $\mC$
the induced transformation $$\eta_*: \Mor_{\Fun^\lax(\bD^1,\mC)}(h,\eta): \id \to \theta'_* \circ \theta_* $$ exhibits $\theta_*: \Mor_{\Fun^\lax(\bD^1,\mC)}(h,f) \to \Mor_{\Fun^\lax(\bD^1,\mC)}(h,g)$
as left adjoint to $$\theta'_*: \Mor_{\Fun^\lax(\bD^1,\mC)}(h,g) \to \Mor_{\Fun^\lax(\bD^1,\mC)}(h,f).$$

In other words, we like to see that for every
$$ \hspace{10mm} (U,V,\kappa: f U \to V h) \in {\Mor_{\Fun^\lax(\bD^1,\mC)}(h,f) \simeq \Mor_{\mC}(Z,A) \underset{{\Mor_{\mC}(Z,B)}}{\bar{\vec{\times}}} \Mor_{\mC}(T,B)},$$ 
$$(R,W,\lambda: g \circ R\to W\circ h) \in  \Mor_{\Fun^\lax(\bD^1,\mC)}(h,g) \simeq {\Mor_{\mC}(Z,X) \underset{\Mor_{\mC}(Z,Y)}{\bar{\vec{\times}}} \Mor_{\mC}(T,Y)}$$
the following functor induced by $\eta$ is an equivalence:
$$ \Mor_{\Mor_{\Fun^\lax(\bD^1,\mC)}(h,g)}(\theta_*(U,V,\kappa),(R,W,\lambda))\to $$$$ \Mor_{\Mor_{\Fun^\lax(\bD^1,\mC)}(h,f)}(\theta'_*\theta_*(U,V,\kappa),\theta'_*(R,W,\lambda)) \to $$
$$ \Mor_{\Mor_{\Fun^\lax(\bD^1,\mC)}(h,f)}((U,V,\kappa),\theta'_*(R,W,\lambda)).$$
The latter functor identifies with the canonical equivalence:
$${\Mor_{\Mor_{\mC}(Z,X)}(\alpha \circ U,R) {\bar{\vec{\times}}}_{\Mor_{\Mor_{\mC}(Z,Y)}(g\circ \alpha \circ U, W \circ h)} \Mor_{\Mor_{\mC}(T,Y)}(\beta \circ V,W)} \simeq $$
$${\Mor_{\Mor_{\mC}(Z,X)}(\alpha \circ U,R) {\bar{\vec{\times}}}_{\Mor_{\Mor_{\mC}(Z,Y)}(\beta \circ f \circ U, W \circ h)} \Mor_{\Mor_{\mC}(T,Y)}(\beta \circ V,W)}\simeq $$
$${ \Mor_{\Mor_{\mC}(Z,A)}(U,\alpha'\circ R) {\bar{\vec{\times}}}_{\Mor_{\Mor_{\mC}(Z,B)}(f \circ U,\beta'\circ W \circ h)} \Mor_{\Mor_{\mC}(T,B)}(V,\beta'\circ W)}.$$

\vspace{1mm}
(2): Let $\alpha'$ be a left adjoint of $\alpha$ and $\beta'$ a left adjoint of $\beta$ and $\zeta'$ a left adjoint of $\zeta.$
Let $\theta':g \to f$ be the lax commutative square
$(\alpha',\beta', \zeta')$
and $\epsilon: \theta' \theta \to \id$ the object of  $$\Mor_{\Mor_{\Fun^\lax(\bD^1,\mC)}(f,f)}(\theta'\theta,\id)
\simeq {\Mor_{\Mor_{\mC}(A,A)}(\alpha'\alpha,\id) \underset{\Mor_{\Mor_{\mC}(A,B)}(f\alpha'\alpha,f)}{\bar{\vec{\times}}} \Mor_{\Mor_{\mC}(B,B)}(\beta'\beta,\id)}$$
corresponding to the three counits.
We will show that $\epsilon$ exhibits $\theta': g \to f$ as left adjoint to $\theta:f \to g$.
For that by \cite[Lemma 5.13.]{heine2017monadicity} it is enough to prove that for every morphism $h: Z \to T$ in $\mC$
the induced natural transformation $$\epsilon_*: \Mor_{\Fun^\lax(\bD^1,\mC)}(h,\eta): \theta'_* \circ \theta_* \to \id $$ exhibits $$\theta'_*: \Mor_{\Fun^\lax(\bD^1,\mC)}(h,g) \to \Mor_{\Fun^\lax(\bD^1,\mC)}(h,f)$$
as left adjoint to $$\theta_*: \Mor_{\Fun^\lax(\bD^1,\mC)}(h,f) \to \Mor_{\Fun^\lax(\bD^1,\mC)}(h,g).$$ 

In other words, we like to see that for every
$$(U,V,\kappa: f U \to V h) \in \Mor_{\Fun^\lax(\bD^1,\mC)}(h,f) \simeq {\Mor_{\mC}(Z,A) \underset{\Mor_{\mC}(Z,B)}{\bar{\vec{\times}}} \Mor_{\mC}(T,B)},$$ 
$$(R,W,\lambda: g \circ R\to W\circ h) \in  \Mor_{\Fun^\lax(\bD^1,\mC)}(h,g) \simeq {\Mor_{\mC}(Z,X) \underset{\Mor_{\mC}(Z,Y)}{\bar{\vec{\times}}} \Mor_{\mC}(T,Y)}$$
the following functor induced by $\epsilon$ yields an equivalence
on maximal subspaces:
$$ \Mor_{\Mor_{\Fun^\lax(\bD^1,\mC)}(h,g)}((R,W,\lambda),\theta_*(U,V,\kappa)) \to \Mor_{\Mor_{\Fun^\lax(\bD^1,\mC)}(h,f)}(\theta'_*(R,W,\lambda),\theta'_*\theta_*(U,V,\kappa))$$$$\to \Mor_{\Mor_{\Fun^\lax(\bD^1,\mC)}(h,f)}(\theta'_*(R,W,\lambda),(U,V,\kappa)).$$
The latter functor identifies with the canonical functor:
$${\Mor_{\Mor_{\mC}(Z,X)}(R,\alpha \circ U) \underset{\Mor_{\Mor_{\mC}(Z,Y)}(g \circ R, \beta \circ V \circ h)}{\bar{\vec{\times}}} \Mor_{\Mor_{\mC}(T,Y)}(W,\beta \circ V)} \simeq $$$${\Mor_{\Mor_{\mC}(Z,A)}(\alpha'\circ R, U) \underset{\Mor_{\Mor_{\mC}(Z,B)}(\beta' \circ g \circ R,V \circ h)}{\bar{\vec{\times}}} \Mor_{\Mor_{\mC}(T,B)}(\beta'\circ W,V) \xrightarrow{(\zeta'R)^*}} $$$$ {\Mor_{\Mor_{\mC}(Z,A)}(\alpha'\circ R, U) \underset{\Mor_{\Mor_{\mC}(Z,B)}(f\circ \alpha'\circ R, V \circ h)}{\bar{\vec{\times}}} \Mor_{\Mor_{\mC}(T,B)}(\beta'\circ W,V)}$$
induced by precomposition with $\zeta'R.$
So it remains to show that $$(\zeta'R)^*:{ \Mor_{\Mor_{\mC}(Z,A)}(\alpha'\circ R, U) \underset{\Mor_{\Mor_{\mC}(Z,B)}(\beta' \circ g \circ R,V \circ h)}{\bar{\vec{\times}}} \Mor_{\Mor_{\mC}(T,B)}(\beta'\circ W,V)} \xrightarrow{} $$$$ {\Mor_{\Mor_{\mC}(Z,A)}(\alpha'\circ R, U) \underset{\Mor_{\Mor_{\mC}(Z,B)}(f\circ \alpha'\circ R, V \circ h)}{\bar{\vec{\times}}} \Mor_{\Mor_{\mC}(T,B)}(\beta'\circ W,V)}$$
is an equivalence.
The latter is a functor over $$ \Mor_{\Mor_{\mC}(Z,A)}(\alpha'\circ R, U) \times \Mor_{\Mor_{\mC}(T,B)}(\beta'\circ W,V) $$ that induces on the fiber over any $(\gamma,\delta)$ the following functor induced by the unit $\id \to \zeta R\circ \zeta'R: $$$\Mor_{\Mor_{\Mor_{\mC}(Z,B)}(\beta' \circ g \circ R,V \circ h)}(\kappa \circ f \gamma \circ \zeta R,\delta h\circ \beta'\lambda) \to $$$$ \Mor_{\Mor_{\Mor_{\mC}(Z,B)}(f\circ \alpha'\circ R, V \circ h)}(\kappa \circ f \gamma \circ \zeta R\circ \zeta'R,\delta h\circ \beta'\lambda\circ \zeta'R) \to$$$$ \Mor_{\Mor_{\Mor_{\mC}(Z,B)}(f\circ \alpha'\circ R, V \circ h)}(\kappa \circ f\gamma,\delta h\circ \beta'\lambda \circ \zeta' R).$$
The latter is an equivalence since the adjunction $$\zeta': f \alpha' \rightleftarrows \beta' g:\zeta $$ in $\Mor_\mC(X,B)$ gives rise to an adjunction $\zeta'R: f \alpha'R \rightleftarrows \beta' g R:\zeta R$ in $\Mor_\mC(Z,B)$ that induces an adjunction $$\Mor_{\mC}(Z,B)(\zeta R, V \circ h): \Mor_{\mC}(Z,B)(f \alpha'R, V \circ h) \rightleftarrows \Mor_{\mC}(Z,B)(\beta' g R, V \circ h): \Mor_{\mC}(Z,B)(\zeta'R, V \circ h).$$

\end{proof}

\begin{theorem}\label{Grayadj}
Let $0 \leq n \leq \infty$ and $\bk \geq n$ and $\mB$ an $n$-category and $\mC$ an $\infty$-category that admits left (right) adjoints for morphisms of dimension smaller $\bk$. The $\infty$-categories $$ \Fun^\lax(\mB,\mC),$$$$ \Fun^\oplax(\mB,\mC) $$ admit left (right) adjoints
for morphisms of dimension smaller $\bk-n$. 

\end{theorem}

\begin{proof}
By \cref{grayhoms} there is a canonical equivalence $$ \Fun^\oplax(\mB,\mC) \simeq \Fun^\lax(\mB^\op,\mC^\op)^\op. $$ So it suffices to prove the statement for $ \mD.$
Since $\n\Cat$ is generated under small colimits by the disks of dimension smaller $n+1$ and the full subcategory of $\infty\Cat$ of small $\infty$-categories having left (right) adjoints 
for morphisms of dimension smaller $\bk-n$ is closed under small limits, and the functor $\Fun^\lax(-,\mC): \n\Cat^\op \to \infty\Cat$ preserves limits, we can assume that $\mB=\bD^\ell$ for $0 \leq \ell \leq n.$
So it suffices to show that for every $0 \leq n \leq \infty$ and $\bk \geq n$ and $\infty$-category $\mC$ that admits left (right) adjoints for morphisms of dimension smaller $\bk$ the $\infty$-category $ \Fun^\lax(\bD^n,\mC) $ admits left (right) adjoints for morphisms of dimension smaller $\bk-n$. 

We proceed by induction on $n \geq 0.$ The case $n=0$ is trivial since $\Fun^\lax(\bD^0,\mC) \simeq \mC.$ Assume the statement for $n$ holds.

By \cref{thas2} there is a canonical equivalence $$ \bD^{n+1} \simeq \bD^n \boxtimes \bD^1 \coprod_{\bD^n \boxtimes \partial\bD^1} \partial\bD^1$$
that induces an equivalence 
$$\Fun^\lax(\bD^{n+1},\mC) \simeq \mC \times \mC \times_{\Fun^\lax(\bD^{n},\mC) \times \Fun^\lax(\bD^{n},\mC)} \Fun^\lax(\bD^{1},\Fun^\lax(\bD^{n},\mC)).$$
Since the full subcategory of $\infty\Cat$ of small $\infty$-categories having left (right) adjoints
for morphisms of dimension smaller $\bk-n$ is closed under small limits, by induction on $n \geq 0$ it is enough to show the statement for $n=1,$ i.e. $\mB=\bD^1.$

To complete the proof we prove by induction on $\bk \geq 3$ that for every $\infty$-category $\mC$ 
having left (right) adjoints for morphisms of dimension smaller $\bk$ the $\infty$-category  
$\Fun^\lax(\bD^1,\mC)$ admits left (right) adjoints for morphisms of dimension smaller $\bk-1.$
The case $\bk=3$ is \cref{adjjo}. We assume that the statement holds for $\bk.$
Let $\mC$ be an $\infty$-category having left (right) adjoints for morphisms of dimension smaller $\bk+1.$
We need to see that $ \Fun^\lax(\bD^1,\mC)$ admits left (right) adjoints for morphisms of dimension smaller $\bk$. 
By assumption $ \Fun^\lax(\bD^1,\mC)$ admits left (right) adjoints for morphisms of dimension smaller $\bk-1$. So it remains to see that for every morphisms $f:A \to B, g: X \to Y$ in $\mC$ corresponding to objects of $\Fun^\lax(\bD^1,\mC)$ the morphism $\infty$-category $\Mor_{\Fun^\lax(\bD^1,\mC)}(f,g)$ admits left (right) adjoints for morphisms of dimension smaller $\bk-1.$
By \cref{homs} there is an equivalence
$$\Mor_{\Fun^\lax(\bD^1,\mC)}(f,g) \simeq \Mor_{\mC}(A,X) \times \Mor_{\mC}(B,Y) \times_{\Mor_{\mC}(A,Y) \times \Mor_{\mC}(A,Y)} \Fun^\lax(\bD^1,\Mor_{\mC}(A,Y)).$$
By induction hypothesis the latter $\infty$-category admits left (right) adjoints for morphisms of dimension smaller $\bk-1$ because the full subcategory of $\infty\Cat$ spanned by the small $\infty$-categories having left (right) adjoints for morphisms of dimension smaller $\bk-1$ is closed under small limits.

\end{proof}

\begin{lemma}\label{gtr}
Let $\mD$ be a presentable category and  $\mC \subset \mD $ a full subcategory such that the embedding $\iota: \mC \subset \mD $ admits a right adjoint $\G: \mD \to \mC. $ 
The following are equivalent:

\begin{enumerate}

\item The functor $\iota: \mC \subset \mD $ admits a left adjoint 
$\F: \mD \to \mC  $ such that the composition $\iota \circ \F: \mD \to \mC \subset \mD $ is an accessible functor.

\item The full subcategory $\mC $ is closed in $\mD$ under small limits and $\G: \mD \to \mC  $ is an accessible functor.

\item The functor $ \iota \circ \G : \mD \to \mC \subset \mD $ is accessible and preserves small limits.

\end{enumerate}

In particular, if (1) or (2) holds, $\mC$ is presentable.

\end{lemma}

\begin{proof}
The category $\mC $ admits small limits and small colimits as a colocalization of a category that admits small limits and small colimits, where small limits in $\mC$ are coreflected via $\G$ from small limits in $\mD$.
So (2) and (3) are equivalent.

If (1) holds, $\mC $ is closed in $\mD$ under small limits. 
If (1) holds, $\mC \subset \mD $ is an accessible localization so that 
$\mC$ is presentable. 
Hence $\G: \mD \to \mC  $ is an accessible functor as a right adjoint between presentable categories. 

Assume that (2) holds. The composition $R:= \iota \circ \G : \mD \to \mC \subset \mD $ is an accessible functor that preserves small limits.  
Since $\mD$ is presentable, $R : \mD \to \mD $ admits a left adjoint $ T: \mD \to \mD $ by the adjoint functor theorem \cite[Corollary 5.5.2.9]{lurie.HTT}. 
Let $ \varepsilon : R = \iota \circ \G  \to  \id_\mD $ be the counit. 
Because $\iota: \mC \subset \mD $ is fully faithful, the unit $\id_\mC \to \G \circ \iota $ is an equivalence so that the composition 
$  \G \circ \epsilon :   \G \circ \iota \circ \G \to  \G $ and thus 
$ R \circ \varepsilon : R  \circ R   \to  R  $
are equivalences by the triangle identities. 

Let $ X, Y $ be objects of $\mD. $  The map $$ \Map_\mD(T(X), R(Y)) \to \Map_\mD(T(X), Y) $$ induced by the counit $ \varepsilon(Y) : R(Y) \to Y $ is equivalent to the map $$ \Map_\mD(T(X), R(Y)) \simeq \Map_\mD(X, R  \circ R(Y )) \to \Map_\mD(X, R(Y) ) \simeq \Map_\mD(T(X), Y) $$ induced by the counit 
$ R(\varepsilon(Y)) : R  \circ R(Y)  \to  R(Y) $ and thus is an equivalence.
Thus $T(X) $ is a colocal object of $\mD$ and so belongs to $\mC. $
So $ T: \mD \to \mD $ factors as $ \F: \mD \to \mC \subset \mD. $
Since $ T: \mD \to \mD $  is left adjoint to $ R: \mD \to \mD,  $ the functor 
$ \F: \mD \to \mC $ is left adjoint to $ \iota \simeq R \circ \iota: \mC \to \mD.$  

\end{proof}

\begin{theorem}\label{Grayloc}

The full subcategory $$ \infty\Cat^\adj \subset \infty\Cat$$ is an accessible monoidal localization,
where $\infty\Cat$ carries the Gray monoidal structure.

\end{theorem}

\begin{proof}

We prove that the full subcategory $\infty\Cat^\adj \subset \infty\Cat$ is a localization. If this is shown, \cref{Grayadj} implies that this localization is monoidal.
In view of \cref{gtr} it suffices to show that the embedding $\iota: \infty\Cat^\adj \subset \infty\Cat$ admits a right adjoint $\G$ such that $\iota \circ \G$ preserves small limits and small filtered colimits.
Let $\G(\mC) \subset \mC$ be the subcategory of morphisms of arbitrary dimension that admit a left and right adjoint.
For every $\mB \in \infty\Cat^\adj $ the induced embedding
$$ \Map_{\infty\Cat}(\mB, \G(\mC)) \to \Map_{\infty\Cat}(\mB,\mC)$$ is an equivalence since $\G(\mB)=\mB.$
So the embedding $\iota: \infty\Cat^\adj \subset \infty\Cat$ admits a right adjoint $\G$.
The functor $\iota \circ \G$ preserves small limits and small filtered colimits because adjoints in limits are formed componentwise and adjoints in filtered colimits are formed summandwise.

\end{proof}

\begin{corollary}\label{corgri}
The category $ \infty\Cat^\adj$
carries a presentably monoidal structure such that the embedding $\infty\Cat^\adj \subset \infty\Cat$ is lax monoidal and admits a monoidal left adjoint.
	
\end{corollary}

\begin{corollary}\label{corgri2}
The category $ \infty\Cat^\adj$ is a bioriented localization
of $\infty\Cat$. 
	
\end{corollary}

\begin{definition}

Let $\mC$ be a monoidal $\infty$-category and $X \in \mC$.
A left (right) dual of $X$ in a left (right) adjoint of $X$
in $B\mC.$

\end{definition}

\begin{proposition}\label{dualos}

Let $\mC$ be a monoidal $\infty$-category and $X \in \mC.$
The following are equivalent:

\begin{enumerate}

\item The object $X$ in $\mC$ admits a left (right) dual.

\item The object $X$ in $\iota_1(\mC)$ admits a left (right) dual.

\item The right $\mC$-linear functor $X \ot (-): \mC \to \mC$ admits a right $\mC$-linear left (right) adjoint, i.e. a left (right) adjoint
in $\RMod_\mC(\infty\Cat)$.

\vspace{1mm}

\item The right $\iota_1(\mC)$-linear functor $X \ot (-): \iota_1(\mC) \to \iota_1(\mC)$ admits a right $\iota_1(\mC)$-linear left (right) adjoint, i.e. a left (right) adjoint
in $\RMod_\mC(\Cat)$.

\end{enumerate}

\begin{proof}

By definition a left (right) dual of $X$ in $\mC$ is a left (right) adjoint of $X$ in $B\mC.$
A left (right) adjoint of $X$ in $B\mC$ is a left (right) adjoint of $X$ in $\iota_2(B\mC) \simeq B\iota_1(\mC),$ which is by definition a left (right) dual of $X$ in $\iota_1(\mC)$.
Hence (1) and (2) are equivalent.

The category $\RMod_\mC(\infty\Cat)$ is presentably left tensored over $\infty\Cat$ and so enriched in $\infty\Cat,$ i.e. an $\infty$-category.
The $\infty$-category $B\mC$ is the full subcategory of the $\infty$-category $\RMod_\mC(\infty\Cat)$ spanned by $\mC.$
Hence an object $X \in \mC$ admits a left (right) dual in $\mC$ if and only if the right $\mC$-linear functor $X \ot (-): \mC \to \mC$ admits a right $\mC$-linear left (right) adjoint. 
Hence (1) and (3) are equivalent.
Replacing $\mC$ by $\iota_1(\mC)$, we see that (2) and (4) are equivalent.

\end{proof}

\end{proposition}

\begin{remark}
Let $\mC$ be a monoidal $\infty$-category and $X \in \mC.$
\cref{dualos} implies that if $X$ admits a right dual $X^\vee $,
the functor $X \ot (-): \mC \to \mC$ of $\infty$-categories admits a right adjoint. Hence for every $A,B \in \mC$ there is a canonical
equivalence $$\Mor_\mC(X \ot A,B) \simeq \Mor_\mC(A, X^\vee \ot B).$$

\end{remark}

\begin{notation}

Let $n \geq 0.$
Let $$ \Mon_{\bE_n}(\infty\Cat)^\vee \subset \Mon_{\bE_n}(\infty\Cat)$$
be the full subcategory of $\bE_n$-monoidal $\infty$-categories whose objects admit left and right duals and whose morphisms admit adjoints.

\end{notation}

\begin{remark}

Since the embedding $\infty\Cat^\adj \subset \infty\Cat$ preserves finite products, it induces an embedding $\Mon_{\bE_n}(\infty\Cat^\adj) \subset \Mon_{\bE_n}(\infty\Cat) $ whose essential image precisely consists of the $\bE_n$-monoidal $\infty$-categories whose underlying
$\infty$-category admits adjoints.
Hence $ \Mon_{\bE_n}(\infty\Cat)^\vee \subset \Mon_{\bE_n}(\infty\Cat)$ is likewise the full subcategory of $\Mon_{\bE_n}(\infty\Cat^\adj) $ of 
$\bE_n$-monoids in $\infty\Cat^\adj$ whose objects admit left and right duals.

\end{remark}

\begin{proposition}\label{dualcomp}Let $n \geq 0.$
The embedding $\Mon_{\bE_n}(\infty\Cat)^\vee \subset \Mon_{\bE_n}(\infty\Cat)$ admits a left adjoint.

\end{proposition}

\begin{proof}

By \cref{Grayloc} the embedding $\infty\Cat^\adj \subset \infty\Cat$
is accessible and admits a left adjoint. Hence $\infty\Cat^\adj$ is presentable. Hence $\Mon_{\bE_n}(\infty\Cat^\adj)$ is presentable and the induced embedding 
$$ \Mon_{\bE_n}(\infty\Cat^\adj) \subset \Mon_{\bE_n}(\infty\Cat)$$ is accessible and preserves small limits and so admits a left adjoint by the adjoint functor theorem. 
So it suffices to prove that the embedding $ \Mon_{\bE_n}(\infty\Cat)^\vee \subset \Mon_{\bE_n}(\infty\Cat^\adj)$
admits a left adjoint.
To see this, we apply \cref{gtr}. We verify the hypotheses. 
The full subcategory $\Mon_{\bE_n}(\infty\Cat)^\vee \subset \Mon_{\bE_n}(\infty\Cat)$ and so also the full subcategory $\Mon_{\bE_n}(\infty\Cat)^\vee \subset \Mon_{\bE_n}(\infty\Cat^\adj)$
are closed under small limits. 
The embedding $\Mon_{\bE_n}(\infty\Cat)^\vee \subset \Mon_{\bE_n}(\infty\Cat^\adj)$ admits a right adjoint $G$ that sends an $\bE_n$-monoid $X$ in $\infty\Cat^\adj$ to the full subcategory of $X$
spanned by the objects that admit left and right duals.
To apply \cref{gtr}, it suffices to see that the composition $$\Mon_{\bE_n}(\infty\Cat^\adj) \xrightarrow{G} \Mon_{\bE_n}(\infty\Cat)^\vee \subset \Mon_{\bE_n}(\infty\Cat^\adj)$$
sending an $\bE_n$-monoid $X$ in $\infty\Cat^\adj$ to the full subcategory of $X$ spanned by the objects that admit left and right duals, preserves small filtered colimits.
For every filtered category $K$ and functor $\beta: K \to \Mon_{\bE_n}(\infty\Cat^\adj)$ there is a commutative triangle

\[
\begin{xy}
\xymatrix{
\colim(G\circ\beta) \ar[rr] \ar[dr]
&& \G(\colim(\beta)) \ar[dl] \\
& \colim(\beta), &
}
\end{xy}
\]
where both diagonal functors are fully faithful.
Hence the top horizonal functor is fully faithful.
The top horizonal functor is also essentially surjective and so an equivalence since every object of $\colim(\beta)$ that admits left and right duals is the image of an object of $\beta(k)$ that admits left and right duals for $k \in K$ because $K$ is filtered.

\end{proof}

\begin{proposition}\label{symsym}

Let $\mC$ be an $\infty$-category with adjoints and $\mA \to \mC, \mB \to \mC$ functors.
There is a canonical equivalence
$$ \mA \overset{\to}{\times}_\mC \mB \simeq \mB \overset{\to}{\times}_\mC \mA.$$ 

\end{proposition}

\begin{proof}

We construct a canonical involution
$$ \theta: \Fun^\oplax(\bD^1,\mC) \simeq \Fun^\oplax(\bD^1,\mC)$$
that covers the involution $\mC \times \mC \simeq \mC \times \mC$ 
switching the factors, and sends a morphism to its right adjoint.
Let $\mT$ be an $\infty$-category.
To construct $\theta,$ it suffices by the Yoneda lemma to give an involution 
$$ \pi_0(\iota_0(\Fun^\lax(\mT, \Fun^\oplax(\bD^1,\mC)))) \simeq \pi_0(\iota_0(\Fun^\lax(\mT, \Fun^\oplax(\bD^1,\mC)))) $$
that covers the induced involution $\pi_0(\iota_0(\Fun^\lax(\mT, \mC \times \mC))) \simeq \pi_0(\iota_0(\Fun^\lax(\mT, \mC \times \mC))) $ .

Let $\mD:= \Fun^\lax(\mT, \mC),$ which admits adjoints by \cref{Grayadj} since $\mC$ admits adjoints.
By the defining universal properties there is a canonical equivalence 
$$ \Fun^\lax(\mT, \Fun^\oplax(\bD^1,\mC)) \simeq  \Fun^\oplax(\bD^1,\mD). $$

So it suffices to give an involution
$$ \rho: \pi_0(\iota_0(\Fun^\oplax(\bD^1,\mD))) \simeq \pi_0(\iota_0(\Fun^\oplax(\bD^1,\mD))) $$
that covers the involution $\pi_0(\iota_0(\mD)) \times \pi_0(\iota_0(\mD)) \simeq \pi_0(\iota_0(\mD)) \times \pi_0(\iota_0(\mD)) $ switching the factors. 
We define $\rho$ to be the map sending a morphism $f: A \to B$ in $\mD$ to its right adjoint $g: B \to A.$
\cref{adjjj} implies that $\rho(\rho(f)) \simeq f $ for every morphism $f: A \to B$ in $\mD.$ Hence $\rho$ is an involution.
By the Yoneda lemma we obtain an involution
$$ \theta: \Fun^\oplax(\bD^1,\mC) \simeq \Fun^\oplax(\bD^1,\mC)$$
that covers the involution $\mC \times \mC \simeq \mC \times \mC$ 
switching the factors. This involution induces an equivalence
$$ \mA \overset{\to}{\times}_\mC \mB \simeq \mA \times_\mC \Fun^\oplax(\bD^1,\mC) \times_\mC \mB \simeq \mB \times_\mC \Fun^\oplax(\bD^1,\mC) \times_\mC \mA \simeq \mB \overset{\to}{\times}_\mC \mA. $$ 

\end{proof}

\subsection{Categorical spectra with adjoints}

Next we define categorical spectra with adjoints.

\begin{definition}\label{spectraa} 
The presentable bioriented category $\Cat\Sp$ of categorical spectra with adjoints is the limit of the following diagram of reduced bioriented categories and right adjoint bioriented functors: 
$$...\xrightarrow{\Omega} \infty\Cat_*^\cop \xrightarrow{\Omega} \infty\Cat_* \xrightarrow{\Omega} \infty\Cat_*^\cop \xrightarrow{\Omega} \infty\Cat_*.$$

\end{definition}

\begin{definition}
The reduced bioriented functor of infinite endomorphisms $$\Omega^\infty: \Cat\Sp\to \infty\Cat_* $$ is the canonical bioriented functor projecting to the value of the final object of the diagram.

\end{definition}

\begin{notation}\label{infsuspi0}
By \cite[Notation 5.2.12.]{heine2026stable} the bioriented functor $\Omega^\infty: \Cat\Sp \to \infty\Cat_* $ admits a bioriented left adjoint $\Sigma^\infty: \infty\Cat_* \to \Cat\Sp $, which we call infinite suspension.
By \cite[Lemma 5.2.14.]{heine2026stable} the bioriented forgetful functor $\infty\Cat_* \to \infty\Cat $ admits a bioriented left adjoint $(-)_+$ that sends $X$ to $ X \coprod *.$ 

We define the bioriented adjunction $$\Sigma^\infty_+:= \Sigma^\infty \circ (-)_+: \infty\Cat \rightleftarrows \infty\Cat_* \rightleftarrows \Cat\Sp:\Omega^\infty.$$

\end{notation}

Every categorical spectrum gives rise to a classical spectrum.
The following is \cite[Lemma 5.2.33.]{heine2026stable}:

\begin{proposition}

The embedding $\Sp \subset \Cat\Sp$ of the full subcategory of spectra, admits a right adjoint. The right adjoint sends a categorical spectrum $X$ to the spectrum whose $n$-th term is the full subspace of $\iota_0(X_n)$ spanned by the tensor invertible objects.
	
\end{proposition}

The following is \cite[Proposition 5.2.19.]{heine2026stable}: 

\begin{proposition}\label{rstab}
The endomorphisms functor $\Omega: \Cat\Sp \to \Cat\Sp$ is equivalent to the evident induced functor on limits $$ \lim(... \xrightarrow{\Omega}\mC \xrightarrow{\Omega} \mC) \to \lim(... \xrightarrow{\Omega}\mC \xrightarrow{\Omega} \mC), \ (X_0, X_1,X_2, ...) \mapsto (X_{-1}, X_0,X_1, ...),$$
which is inverse to the shift functor $$ \lim(... \xrightarrow{\Omega}\mC \xrightarrow{\Omega} \mC) \to \lim(... \xrightarrow{\Omega}\mC \xrightarrow{\Omega} \mC), \ (X_0, X_1,X_2, ...) \mapsto (X_1, X_2,X_3, ...),$$ and so is an equivalence.

\end{proposition}

The following is \cite[Corollary 5.3.6.]{heine2026stable}:

\begin{theorem}
The category $ \Cat\Sp$
carries a unique presentably monoidal structure such that the functor
$\Sigma_+^\infty: \infty\Cat \to \Cat\Sp $ is monoidal, where 
$\infty\Cat$ carries the Gray tensor product. 

\end{theorem}

\begin{definition}

A categorical spectrum with adjoints is a categorical spectrum $X$ whose levels are $\infty$-categories with adjoints.

\end{definition}

\begin{definition}\label{spectraa} 
Let $$\Cat\Sp^\adj \subset \Cat\Sp$$ be the full subcategory of categorical spectra with adjoints.

\end{definition}

\begin{remark}

The bioriented functor $ \Omega: \infty\Cat_*^\cop \to \infty\Cat_* $ restricts to a bioriented functor $$ \Omega: (\infty\Cat_*^\adj)^\cop \to \infty\Cat_*^\adj.$$ 
The bioriented category $\Cat\Sp^\adj$ is the limit of the following diagram of reduced bioriented categories and reduced bioriented functors:
$$...\xrightarrow{\Omega} (\infty\Cat^\adj)^\cop \xrightarrow{\Omega} \infty\Cat^\adj \xrightarrow{\Omega} (\infty\Cat^\adj)^\cop \xrightarrow{\Omega} \infty\Cat^\adj.$$

\end{remark}

\vspace{2mm}

\begin{theorem}\label{corgri3}
The category $ \Cat\Sp^\adj$
carries a presentably monoidal structure such that the embedding $\Cat\Sp^\adj \subset \Cat\Sp$ is lax monoidal and admits a monoidal left adjoint.
	
\end{theorem}

\begin{proof}

The right adjoint embedding $\infty\Cat^\adj \subset \infty\Cat$
induces a right adjoint embedding $\Cat\Sp^\adj \subset \Cat\Sp$.
We prove that $\Cat\Sp^\adj \subset \Cat\Sp$ is a monoidal localization.
For this, we have to see that for every $X \in \Cat\Sp$, the right adjoints of the functors $$ X \wedge (-), (-) \wedge X : \Cat\Sp \to \Cat\Sp$$ preserve the full subcategory $\Cat\Sp^\adj.$
We prove the case of the right adjoint $\R\Mor_{\Cat\Sp}(X,-) $ of the functor $ X \wedge (-) : \Cat\Sp \to \Cat\Sp$. The other case is similar.

By \cite[6.1.22.]{heine2026stable} the category $\Cat\Sp$ is generated under small colimits by endomorphims of the essential image of $\Sigma_+^\infty: \infty\Cat \to \Cat\Sp $. Since $\Cat\Sp^\adj \subset \Cat\Sp$ is closed under small limits, this implies that we can reduce to the case that $X \simeq \Omega^n(\Sigma^\infty_+(Y)) $ for $n \geq 0.$ We want to see that  for every $m \geq 0$ and $Z \in \infty\Cat^\adj$ the $\infty$-category 
$ \Omega^\infty(\Sigma^m(\R\Mor_{\Cat\Sp}(X,Z))) $ admits adjoints.
There is a canonical equivalence
$$ \Omega^\infty(\Sigma^m(\R\Mor_{\Cat\Sp}(X,Z))) \simeq $$$$ \Omega^\infty(\R\Mor_{\Cat\Sp}(X, \Sigma^m(Z))
\simeq $$$$ \Omega^\infty(\R\Mor_{\Cat\Sp}(\Sigma^\infty_+(Y),\Sigma^{n+m}(Z)))
\simeq $$$$ \Fun^\lax(Y,\Omega^\infty(\Sigma^{n+m}(Z))).$$
Hence the result follows from \cref{Grayadj}.

\end{proof}

\begin{corollary}\label{corgri4}
The category $ \Cat\Sp^\adj$ is a bioriented localization
of $ \Cat\Sp$. 

\end{corollary}

\begin{remark}\label{infsuspi}

The bioriented functor $\Omega^\infty: \Cat\Sp \to \infty\Cat_* $
restricts to a bioriented functor $\Omega^\infty: \Cat\Sp^\adj \to \infty\Cat^\adj_*.$
\cref{corgri4} and \cref{infsuspi0} imply that the latter admits a bioriented left adjoint $\infty\Cat^\adj_* \to \Cat\Sp^\adj, $ which factors as $$ \infty\Cat^\adj_*  \subset \infty\Cat_* \xrightarrow{\Sigma^\infty} \Cat\Sp \to \Cat\Sp^\adj. $$

\end{remark}

\subsubsection{Connective categorical spectra}

Next we consider connective categorical spectra.

\begin{definition}

A categorical spectrum $\mC$ is connective if for every $\n \geq 0$ the $\infty$-category
$\Omega^\infty(\Sigma^\n(\mC)) $ is $\n$-1-connected in the sense of \cref{conn}.

\end{definition}

\begin{notation}
Let $\Cat\Sp_{\geq0} \subset \Cat\Sp $ be the full subcategory of connective categorical spectra.	

\end{notation}

\begin{example}\label{connex}
	
Let $\mC$ be a symmetric monoidal $\infty$-category. There is a  categorical spectrum $$B^\infty(\mC):= (\mC, B\mC, B^2\mC, ...)$$ with the natural bonding maps.
The categorical spectrum $ B^\infty(\mC)$ is connective.
	
\end{example}

The next result is \cite[Proposition 5.2.26.]{heine2026stable}, a categorical version of May's recognition principle identifying infinite loop spaces with grouplike $\bE_\infty$-spaces \cite{may1974}:

\begin{proposition}\label{symsp0}

The induced functor \begin{equation}\label{eqojj}
\Cat\Sp \simeq \Cmon(\Cat\Sp) \xrightarrow{\Cmon(\Omega^\infty)} \Cmon(\infty\Cat) \end{equation}
admits a left adjoint that sends a symmetric monoidal $\infty$-category $\mC$ to $B^\infty(\mC)$ and induces an equivalence
$$ \Cmon(\infty\Cat) \simeq \Cat\Sp_{\geq0}.$$

\end{proposition}

\section{Homotopy theory} 

\subsection{Homotopy monoidal categories}

In this section we prove a categorical analogue of Whitehead's theorem \cite[Theorem 1.1.,1.2.]{whitehead1949combinatorial} for $\infty$-categories with adjoints (\cref{homten})
and categorical spectra with adjoints (\cref{maf}).

Next we define connectivity, which is \cite[Definition 3.3.4.]{gepner2026homotopy}:

\begin{definition}\label{defcon} Let $n \geq 0.$
\begin{itemize}
\item Every functor between $\infty$-categories is $-1$-connective.
		
\item A functor between $\infty$-categories is $n$-connective if it is essentially surjective and induces $n$-1-connective functors on morphism $\infty$-categories.
		
\item A functor between $\infty$-categories is $\infty$-connective if it is $n$-connective for every $n \geq 0.$

\end{itemize}
	
\end{definition}

\begin{definition}Let $-2 \leq n \leq \infty.$
A functor between $\infty$-categories is $n$-connected if it is
$n$+1-connective.
	
\end{definition}

\begin{example}\label{exfull} For every $n \geq 0$ and $\infty$-category $\mC$ the induced functor
$\iota_{n}(\mC) \to \iota_{n+1}(\mC) $ is $n$-connective.

This follows by induction on $n \geq 0 $
since $$ \Mor_{\iota_{n+1}(\mC)}(A,B) \simeq \iota_{n}(\Mor_{\mC}(A,B))$$ for every $A,B \in \mC.$
    
\end{example}

The following is \cite[Definition 3.2.12.]{gepner2026homotopy}:

\begin{definition}Let $\n \geq 0.$

\begin{itemize}
\item A functor $\F: \mC \to \mD$ is a $0$-equivalence
if it induces a bijection on equivalence classes of objects.

\item A functor $\F: \mC \to \mD$ is an $\n$-equivalence
if it induces an $\n-1$-equivalence on morphism $\infty$-categories. 

\end{itemize}

\end{definition}

The following is \cite[Proposition 3.2.19.]{gepner2026homotopy}:

\begin{lemma}\label{homt} A functor $\F: \mC \to \mD$ is an equivalence if and only if it is an $n$-equivalence for every $n \geq 0.$
\end{lemma}

\begin{definition}Let $ n \geq 0.$
A functor is a weak $n$-equivalence if it induces under $n$-truncation an $n$-equivalence.
	
\end{definition}

\begin{lemma}\label{fuleqo}
	
Let $ n \geq 0.$
Every $n$+1-connective functor is a weak $n$-equivalence.
	
\end{lemma}

\begin{proof}
	
We proceed by induction on $n \geq 0.$
Let $\phi: \mC \to \mD$ be an $n$+2-connective functor.
Then $\phi$ induces on morphism $\infty$-categories $n$+1-connective functors and so weak $n$-equivalences by induction hypothesis. This is equivalent to say that the $n$+1-truncation of $\psi$
induces $n$-equivalences on morphism $\infty$-categories. Since $\phi$ is essentially surjective
and the truncation functors are essentially surjective, also the $n$+1-truncation of $\psi$ is essentially surjective and so an $n$+1-equivalence. 
So it suffices to prove the induction start $n=0.$
Let $\phi: \mC \to \mD$ be a 1-connective functor.
Since the truncation functors are essentially surjective,
$\psi$ induces under zero truncation an essentially surjective functor.
To see that $\psi$ induces under zero truncation a zero equivalence, it suffices to see that
two objects $X,Y \in \mC$ are connected by a zig-zag of morphisms in $\mC$ if the images under $\phi$ are connected by a zig-zag of morphisms in $\mD$.
This follows immediately from the assumptions on $\phi.$ 
	
\end{proof}

\begin{remark}

The functor $\tau_n: \infty\Cat \to \n\Cat$ does not preserve $n$-equivalences.
In other words $n$-equivalences are generally not weak $n$-equivalences.
However every $n$+1-equivalence is $n$+1-connective and so by \cref{fuleqo} a weak $n$-equivalence. 

\end{remark}

\begin{example}Let $n \geq 0$.
A functor between $n$-categories is a weak $n$-equivalence if and only if it is an $n$-equivalence.

A functor between spaces is a weak $n$-equivalence if and only if it is an $n$-equivalence
if and only if it is inverted by the left adjoint of the embedding $(n,0)\Cat \subset 0\Cat =\mS.$
This holds since the left adjoint $\Ho_n: \n\Cat \to(n, n)\Cat $ of the embedding $(n, n)\Cat \subset \n\Cat$ sends spaces to $(n,0)$-categories as one immediately sees by induction, and so restricts to the left adjoint of the embedding $(n,0)\Cat \subset 0\Cat =\mS.$
Hence a functor between spaces is a weak $n$-equivalence if and only if it is an $n$-equivalence
if and only if it induces isomorphisms on homotopy groups in degrees lower or equal $n.$

\end{example}

\begin{lemma}Let $n \geq 0.$
The left adjoint $\tau_0$ of the embedding $\mS =0\Cat \subset \infty\Cat$ sends 
weak $n$-equivalences to $n$-equivalences.
	
\end{lemma}

\begin{proof}
A functor is a weak $n$-equivalence if and only if it is inverted by the left adjoint of the embedding $(n,n)\Cat \subset \infty\Cat,$ which factors as $\Ho_n \circ \tau_n.$
So for every  weak $n$-equivalence $\mC \to \mD$  and $(n,n)$-category $\mE$ the induced map $ \Map_{\infty\Cat}(\mD,\mB) \to \Map_{\infty\Cat}(\mC,\mB) $  is an equivalence.
In particular, for every $(n,0)$-category $X$ the induced map $ \Map_{\infty\Cat}(\mD,X) \to \Map_{\infty\Cat}(\mC,X) $  is an equivalence. The latter map identtifies withe the map
$ \Map_{\mS}(\tau_0(\mD),X) \to \Map_{\mS}(\tau_0(\mC),X) $.
So the map $\tau_0(\mC) \to \tau_0(\mD)$ is local with respect to any $(n,0)$-category and so is an
$n$-equivalence.
	
\end{proof}

\begin{notation}

The embedding $(1,1)\Cat \subset \Cat$ admits a left adjoint $\Ho$
that assigns the homotopy category.

\end{notation}

\begin{definition}
The functor of homotopy category is $\Ho:= \Ho \circ \iota_1: \infty\Cat \to \Cat \to (1,1)\Cat$.

\end{definition}

\begin{definition}
	
Let $\mC$ be an $\infty$-category, $X \in \mC$ and $n \in \bN.$
The $n$-th homotopy monoidal category is 
$$\Pi_n(\mC,X):= \Ho(\Omega^{n}(\mC,X)).$$	
	
\end{definition}

\begin{remark}
	
Since $\Omega^{n}(\mC,X)$ is an $\bE_n$-monoidal $\infty$-category,
$\Pi_n(\mC,X)= \Ho(\Omega^{n}(\mC,X))$ is an $\bE_n$-monoidal $(1,1)$-category, i.e.
a monoidal $(1,1)$-category for $n=1$, a braided monoidal $(1,1)$-category for $n=2$
and a symmetric monoidal $(1,1)$-category for $n> 2.$
	
\end{remark}

\begin{example}
Let $\mC$ be a space, $X \in \mC $ and $\ell \geq 0.$ 
The category $\Pi^\ell(\mC,X)$ is a Picard groupoid and so its underlying groupoid 
is the product of the set of isomorphism classes and the delooping of the group of automorphisms of the tensor unit: there is a canonical equivalence $\Pi^\ell(\mC,X) \simeq 
\pi^\ell(\mC,X) \times B \pi^{\ell+1}(\mC,X).$ 
	
\end{example}

\begin{definition}
Let $\mC$ be a categorical spectrum and $n \in \bZ$. 
The $n$-th stable homotopy monoidal category is 
$$\Pi_n(\mC):= \Ho(\Omega^\infty(\Omega^{n}(\mC)).$$	
	
\end{definition}

\begin{remark}Let $\mC$ be a categorical spectrum and $n \geq 0.$ There is a canonical equivalence $$\Pi_n(\mC) \simeq \Pi_n(\Omega^\infty(\mC),\tu).$$	
	
\end{remark}

\begin{lemma}\label{anop}
Let $n \geq 1$. Every $n$-category that admits left (right) adjoints for morphisms of dimension smaller or equal $n$, is an $\infty$-groupoid.
So every $n$-category that admits left (right) adjoints is an $\infty$-groupoid. 

\end{lemma}

\begin{proof}We prove the case of left adjoints. The case of right adjoints is dual.
We proceed by induction on $n \geq 1$. We assume the statement holds for $n$ and let $\mC$ be an $n+1$-category that admits left adjoints for morphisms of dimension smaller or equal $n+1$. Then $\mC$ has left adjoints for morphisms of dimension smaller 2 and for every objects $ X, Y \in \mC$ the $n$-category $\Mor_\mC(X,Y)$ admits left adjoints for morphisms of dimension smaller or equal $n$ and so by induction hypothesis is a space. Thus $\mC$ is a 1-category that admits left adjoints for morphisms of dimension smaller $2.$
So it remains to prove the case $n=1$.
Let $g: Y \to X$ be a morphism in $\mC$. Then $g$ has a left adjoint $f: X \to Y$. So there are $\eta: \id \to g \circ f $ and $\varepsilon: f  \to \id,$ which are equivalences since $\mC$ is a $1$-category. So $g$ is an equivalence. Thus $\mC$ is a space.

\end{proof}

\begin{lemma}\label{hiko}
Let $\bk,n \geq 1 $ and $\mC$ an $\infty$-category that admits left (right) adjoints
for morphisms of dimension smaller $\bk$. Then $\tau_n(\mC)$ admits left (right) adjoints for morphisms of dimension smaller $\bk$.
	
\end{lemma}

\begin{proof} 
For $\bk=1$ there is nothing to show. We proceed by induction on $\bk \geq 2.$
The functor $\mC \to \tau_n(\mC)$ is essentially surjective and full for every $n \geq 1.$	So $\tau_{n}(\mC)$ admits left (right) adjoints of dimension smaller 2.
It remains to see that for any $X, Y \in \mC$ the morphism $\infty$-category $\Mor_{\tau_{n}\mC}(X,Y)\simeq \tau_{n-1}(\Mor_\mC(X,Y)) $ admits left (right) adjoints
for morphisms of dimension smaller $\bk-1$.
This holds by induction hypothesis because $\Mor_\mC(X,Y)$ admits left (right) adjoints for morphisms of dimension smaller $\bk-1$.
\end{proof}

\begin{corollary}\label{sasa}
	
Let $ n \geq 0 $ and $\mC$ an $\infty$-category that admits left (right) adjoints
for morphisms of dimension smaller $ n+1$. Then $\tau_n(\mC)$ is an $\infty$-groupoid. 

\end{corollary}	

\begin{proof} For $n=0$ there is nothing to prove. So we can assume that $n > 0. $
By \cref{hiko} the $n$-category $\tau_n(\mC)$ admits left (right) adjoints for morphisms of dimension smaller $ n+1$ and so is a space by \cref{anop}.
	
\end{proof}

\begin{corollary}Let $ n \geq 0$ and $n+1 \geq m \geq 0 $ and $\mC$ an $\infty$-category that admits left (right) adjoints
for morphisms of dimension smaller $ n+2$ and $A,B \in \mC.$
The following canonical functor is an equivalence: $$\tau_{n}(\Mor_{\mC}(A,B)) \to \Mor_{\tau_{m}(\mC)}(A,B). $$	
	
\end{corollary}	

\begin{proof}
	
By \cref{sasa} the canonical functors $\tau_{n+1}(\mC) \to \tau_{0}(\mC), \tau_{m}(\mC) \to \tau_{0}(\mC)$ and so also the functor $\tau_{n+1}(\mC) \to \tau_{m}(\mC)$ are equivalences.
The canonical functor $\tau_{n}(\Mor_{\mC}(A,B)) \to \Mor_{\tau_{m}(\mC)}(A,B)$
factors as $\tau_{n}(\Mor_{\mC}(A,B)) \to \Mor_{\tau_{n+1}(\mC)}(A,B) \simeq \Mor_{\tau_{m}(\mC)}(A,B).$
The latter functor is an equivalence and 
by construction of the left adjoint $\tau_{n+1}$ of the embedding $\n+1\Cat \subset \infty\Cat.$
	
\end{proof}

\begin{corollary}\label{korfh} Let $n \geq m \geq 1 $ and $\mC$ an $\infty$-category that admits left (right) adjoints
for morphisms of dimension smaller $n+1$ and $X \in \mC.$
The following canonical functor is an  equivalence $$\tau_{n-1}(\Omega(\mC, X)) \to \Omega(\tau_{m}(\mC),X).$$
	
\end{corollary}	

\begin{corollary}\label{Fulll}

Let $ 1 \leq n \leq \infty $ and $\mC$ an $\infty$-category that admits left (right) adjoints for morphisms of dimension smaller $ n+1$. The canonical functor $\mC \to \tau_0(\mC)$ is $n-1$-connected.
    
\end{corollary}

\begin{proof}

By \cref{sasa} the canonical functor $\mC \to \tau_n(\mC)$ 
identifies with the functor $\mC \to \tau_0(\mC)$.
The functor $$\mC \to \tau_n(\mC)$$ is $n-1$-connected.
\end{proof}

\begin{corollary}

Let $\mC$ be an $\infty$-category that admits adjoints. The canonical functor $$\mC \to \tau_0(\mC)$$ is $\infty$-connected.
    
\end{corollary}

\begin{corollary}\label{loopcon}

Let $0 \leq m \leq n \leq \infty $ and $X$ an $\infty$-category that admits left (right) adjoints of dimension smaller $n+1.$
The canonical functor $$\Omega^m(X) \to \Omega^m(\tau_0(X)) $$ is $n-m-1$-connected.
    
\end{corollary}

\begin{proof}
By \cref{Fulll} the canonical functor $X \to \tau_0(X) $ is $n-1$-connected.
Thus the canonical functor $ \Omega^m(X) \to \Omega^m(\tau_0(X)) $ is $n-m-1$-connected.

\end{proof}

\begin{theorem}

Let $0 \leq m < n \leq \infty $ and $X$ an $\infty$-category that admits left (right) adjoints of dimension smaller $n+1.$
The induced functor 
$$ \Pi_m(X) \to \Pi_m(\tau_0(X)) \simeq \pi_{m-1}(\tau_0(X)) \times \pi_{m}(\tau_0(X)) $$ is full and essentially surjective. 
    
\end{theorem}

\begin{proof}

By \cref{loopcon} the canonical functor $\Omega^m(X) \to \Omega^m(\tau_0(X)) $ is $n-m$-connective and so in particular 1-connective since 
$n-m > 0.$
Hence also the induced functor
$\iota_1\Omega^m(X) \to \iota_1\Omega^m(\tau_0(X)) $ is 0-connected so that also the induced functor
$ \Ho(\iota_1\Omega^m(X)) \to \Ho(\iota_1\Omega^m(\tau_0(X))) $ is 1-connective.

\end{proof}

\begin{corollary}

Let $X$ be an $\infty$-category that admits left (right) adjoints.
For every $m \geq 0 $ the induced functor 
$$ \Pi_m(X) \to \Pi_m(\tau_0(X)) \simeq \pi_{m-1}(\tau_0(X)) \times \pi_{m}(\tau_0(X)) $$ is full and essentially surjective. 
    
\end{corollary}

\begin{notation}

We write $(-)^\g$ for the left adjoint of the embedding
${(1,0)\Cat} \subset {(1,1)\Cat}$ of the full subcategory of groupoids, i.e. $(1,0)$-categories, into the category of discrete categories, i.e. $(1,1)$-categories.
    
\end{notation}

\begin{theorem}

Let $X$ be an $\infty$-category that admits left (right) adjoints and $m \geq 0.$
The canonical map of groupoids $$ \Pi_{m}(X)^\g \to \Pi_{m}(\tau_0(X))$$
induces an isomorphism
$$ \pi_0(\Pi_{m}(X)^\g) \to \pi_0(\Pi_{m}(\tau_0(X))) \simeq \pi_{m}(\tau_0(X)).$$
    
\end{theorem}

\begin{proof}

For every category $Z$ there is a canonical equivalence
$$ \Ho(Z)^\g \simeq \Ho\tau_0 Z.$$
Hence for every $\infty$-category $Z$ there is a canonical equivalence
$$ \Pi_m(Z)^\g = \Ho(\iota_1\Omega^m(Z))^\g \simeq \Ho\tau_0 \iota_1\Omega^m(Z).$$

So there is a canonical bijection
$$ \pi_0(\Pi_m(X)^\g) \simeq \pi_0 \Ho \tau_0 \iota_1\Omega^m(X) \simeq \pi_0 \tau_0 \iota_1\Omega^m(X) \simeq $$$$ \pi_0 \tau_0 \iota_1\Omega^m(\colim_{n >m} \iota_\n(X)) \simeq \colim_{n >m} \pi_0 \tau_0 \iota_1\Omega^m \iota_\n(X).$$

By \cref{exfull} for every $n > m$ the induced functor
$\iota_{n}(X) \to \iota_{n+1}(X) $ is $n$-connective and so $m+1$-connective so that the induced functor
$\Omega^m \iota_{n}(X) \to \Omega^m \iota_{n+1}(X) $ is 1-connective.
Hence also the induced functor
$$\iota_1 \Omega^m \iota_{n}(X) \to \iota_1 \Omega^m \iota_{n+1}(X) $$ is 1-connective.
By \cref{fuleqo} this implies that the following induced map is a bijection: $$\pi_0 \tau_0\iota_1 \Omega^m \iota_{n}(X) \to \pi_0 \tau_0\iota_1 \Omega^m \iota_{n+1}(X).$$

Hence for every $n >m$ the following canonical map is a bijection:
$$\pi_0 \tau_0 \iota_1\Omega^m \iota_\n(X) \to \colim_{n >m} \pi_0 \tau_0 \iota_1\Omega^m \iota_\n(X).$$

We obtain a canonical bijection
$$ \pi_0(\Pi_m(X)^\g) \simeq \colim_{n >m} \pi_0 \tau_0 \iota_1\Omega^m \iota_\n(X) $$$$ \simeq \pi_0 \tau_0 \iota_1\Omega^m \iota_{m+1}(X) \simeq \pi_0 \tau_0 \Omega^m \iota_{m+1}(X).$$

There is a canonical bijection
$$\pi_{m} \tau_0(X) \simeq \pi_{m} \tau_0 (\colim_{n >m}\iota_\n(X)) \simeq $$$$ \colim_{n >m} \pi_{m} \tau_0 \iota_\n(X) \simeq \colim_{n >m} \pi_{m} \tau_{m} \iota_\n(X) \simeq $$$$ \colim_{n >m} \pi_{0} \Omega^{m} \tau_{m} \iota_\n(X) \simeq \colim_{n >m} \pi_{0} \tau_0 \Omega^{m} \iota_\n(X).$$

By \cref{exfull} for every $n > m$ the induced functor
$\iota_{n}(X) \to \iota_{n+1}(X) $ is $n$-connective and so $m+1$-connective so that the induced functor
$\Omega^m \iota_{n}(X) \to \Omega^m \iota_{n+1}(X) $ is 1-connective.
By \cref{fuleqo} this implies that the induced map $$\pi_0 \tau_0 \Omega^m \iota_{n}(X) \to \pi_0 \tau_0 \Omega^m \iota_{n+1}(X) $$ is a bijection.
Hence for every $n >m$ the following canonical map is a bijection:
$$\pi_{0} \tau_0 \Omega^{m} \iota_\n(X) \to \colim_{n >m} \pi_{0} \tau_0 \Omega^{m} \iota_\n(X).$$

We obtain a canonical bijection
$$ \pi_{m} \tau_0(X) \simeq \pi_{0} \tau_0 \Omega^{m} \iota_{m+1}(X).$$

\end{proof}

For the following proposition note that for every $\infty$-category
$\mC$ the canonical map $\pi_0\iota_0(\mC) \to \pi_0\tau_0(\mC)$ exhibits the target as the quotient of the set of equivalence classes of objects of $\mC$ modulo the equivalence relation $\sim$ that two objects are linked by a zig-zag of morphisms.

\begin{proposition}
Let $\ell \geq 0 $ and $\mC$ an $\infty$-category that admits left (right) adjoints, $X \in \mC$
and $A,B \in \Pi_\ell(\mC,X)$.
The canonical map 
$$ \Pi_\ell(\mC,X)(A,B) \to \Pi_\ell(\tau_0\mC,X)(A,B)$$
identifies with the quotient map 
$$\{ \mathrm{Equivalence} \ \mathrm{classes} \ \mathrm{in} \ \Omega^\ell(\mC,X)(A,B) \}  \to \{ \mathrm{Equivalence} \ \mathrm{classes} \ \mathrm{in} \ \Omega^\ell(\mC,X)(A,B) \}/\sim .$$

\end{proposition}

\begin{proof}
Let $n \geq m \geq 1 $ and $\mC$ an $\infty$-category that admits left (right) adjoints
for morphisms of dimension smaller $n+1$ and $X \in \mC.$
By \cref{korfh} the canonical functor $$\tau_{1}(\Omega^\ell(\mC, X)) \to \Omega^\ell(\tau_{0}(\mC),X)$$	
is an equivalence. The latter induces an equivalence
$$\Ho(\tau_{1}(\Omega^\ell(\mC, X))) \to \Ho(\Omega^\ell(\tau_{0}(\mC),X)).$$	
Hence the canonical functor
$ \Ho(\Omega^\ell(\mC,X)) \to \Ho(\Omega^\ell(\tau_{0}(\mC),X))$
identifies with the functor
$$ \Ho(\Omega^\ell(\mC,X)) \to \Ho(\tau_{1}(\Omega^\ell(\mC, X))) $$
that induces on hom sets between two objects $A,B \in \Ho(\Omega^\ell(\mC,X))$
the canonical map $$ \pi_0\iota_0(\Omega^\ell(\mC,X)(A,B)) \to \pi_0\tau_0(\Omega^\ell(\mC,X)(A,B)).$$

\end{proof}

\subsection{A categorical Whitehead theorem}

\begin{lemma}\label{whiteheadlemma}

Let $\mC$ be an $\infty$-category, $\F: \mC \to \mD$ a functor and
$f, g: X \to Y$ morphisms in $\mC$ such that $f$ admits a right adjoint
$f^R: Y \to X.$

There is a canonical commutative square of $\infty$-categories:
\begin{equation}\label{ohast}
\begin{xy}
\xymatrix{
\Mor_{\Mor_\mC(X,Y)}(f,g) \ar[d] \ar[r]^\simeq
& \Mor_{\Omega(\mC, X)}(\id,f^R \circ g) \ar[d]
\\ 
\Mor_{\Mor_\mD(F(X),F(Y))}(F(f), F(g)) \ar[r]^\simeq & \Mor_{\Omega(\mD, F(X))}(\id, F(f^R) \circ F(g)),
}
\end{xy}\end{equation}
where the horizontal functors are equivalences.

\end{lemma}

\begin{proof}

The adjunction $ f : X \rightleftarrows Y: f^R $ in $\mC$ is sent by $F: \mC \to \mD$ to an adjunction $ F(f) : F(X) \rightleftarrows F(Y): F(f^R) $ in $\mD$. 
These adjunctions induce adjunctions
$$ \Mor_\mC(X,f) : \Mor_\mC(X,X) \rightleftarrows \Mor_\mC(X,Y): \Mor_\mC(X,f^R) $$ 
and 
$$ \Mor_\mD(F(X),F(f)) : \Mor_\mD(F(X),F(X)) \rightleftarrows \Mor_\mD(F(X),F(Y)): \Mor_\mD(F(X),F(f^R)). $$ 
The functor $F: \mC \to \mD$ gives rise to a natural transformation
$\Mor_\mC(X,-) \to \Mor_\mD(F(X),-) \circ F$ of functors $\mC \to \infty\Cat$, which sends the adjunction $ f : X \rightleftarrows Y: f^R $ in $\mC$ to a map of adjunctions
\[
\begin{xy}
\xymatrix@C=9em@R=4em{
\Mor_\mC(X,X)
\ar@<0.6ex>[r]^{\Mor_\mC(X,f)}
\ar[d]
&
\Mor_\mC(X,Y)
\ar@<0.6ex>[l]^{\Mor_\mC(X,f^R)}
\ar[d]
\\
\Mor_\mD(F(X),F(X))
\ar@<0.6ex>[r]^{\Mor_\mD(F(X),F(f))}
&
\Mor_\mD(F(X),F(Y))
\ar@<0.6ex>[l]^{\Mor_\mD(F(X),F(f^R))}
}
\end{xy}
\]
This map of adjunctions induces for every $h \in \Mor_\mC(X,X), g \in \Mor_\mC(X,Y) $ a commutative square
\begin{equation*}
\begin{xy}
\xymatrix{
\Mor_{\Mor_\mC(X,Y)}(f \circ h, g) \ar[d] \ar[r]^\simeq
& \Mor_{\Mor_\mC(X,X)}(h,f^R \circ g) \ar[d]
\\ 
\Mor_{\Mor_\mD(F(X),F(Y))}(F(f) \circ F(h), F(g)) \ar[r]^\simeq & \Mor_{\Mor_\mD(F(X),F(X))}(F(h), F(f^R) \circ F(g)),
}
\end{xy}\end{equation*}
where the horizontal functors are equivalences.
Take $h=\id_X.$

\end{proof}

\begin{proposition}\label{whiteheads} Let $1 \leq n \leq \infty$ and $\mC$ an $\infty$-category such that every 1-morphism admits a right adjoint. 

A functor $\F: \mC \to \mD$ is an $n$-equivalence if and only if it is an 1-equivalence and for every $X \in \mC$ the induced functor $$\Omega(\mC,X) \to \Omega(\mD, F(X))$$ is an $n$-1-equivalence.

\end{proposition}

\begin{proof}

Any $n$-equivalence $F: \mC \to \mD$ is an 1-equivalence and induces for every $X \in \mC$ an $n$-1-equivalence $$\Omega(\mC,X) \to \Omega(\mD, F(X)).$$ 
We prove the other direction.
Let $F: \mC \to \mD$ be an $1$-equivalence that induces for every $X \in \mC$ an $n$-1-equivalence $$\Omega(\mC,X) \to \Omega(\mD, F(X)).$$
Since $F$ is an 1-equivalence, it induces a bijection on equivalence classes of objects and induces on morphism $\infty$-categories a bijection on equivalence classes of objects.
Thus $F$ is an $n$-equivalence if it induces on morphism $\infty$-categories of morphism $\infty$-categories an $n$-2-equivalence.

Since the induced functor $$\Omega(\mC,X) \to \Omega(\mD, F(X))$$ is an $n$-1-equivalence, for every morphisms $f, g: X \to Y$ in $\mC$ the induced functor
$$\Mor_{\Omega(\mC, X)}(\id,f^R \circ g) \to \Mor_{\Omega(\mD, F(X))}(\id, F(f^R) \circ F(g))$$ is an $n$-1-equivalence, 
where $f^R$ is the right adjoint of $f.$
By \cref{whiteheadlemma} the latter functor identifies with the functor 
$$ \Mor_{\Mor_\mC(X,Y)}(f,g) \to \Mor_{\Mor_\mD(F(X),F(Y))}(F(f), F(g)), $$ which is therefore an $n$-2-equivalence.

\end{proof}

\begin{proposition}\label{dimest}
Let $\mC$ be an $\infty$-category whose 1-morphisms admit a right adjoint and $n > 1$. The $\infty$-category $\mC$ is an $n$-category
if and only if for every $X \in \mC $ the $\infty$-category
$\Omega(\mC,X)$ is an $n-1$-category.

\end{proposition}

\begin{proof}

We prove the non-trivial direction.
We assume that for every $X \in \mC $ the $\infty$-category
$\Omega(\mC,X)$ is an $n-1$-category.
We want to see that for every $X,Y \in \mC, f, g \in \Mor_\mC(X,Y)$ the $\infty$-category
$\Mor_{\Mor_\mC(X,Y)}(f,g) $ is an $n-2$-category. 
Since $f: X \to Y $ admits a right adjoint $f^R: Y \to X, $
by \cref{whiteheads} there is a canonical equivalence
$\Mor_{\Mor_\mC(X,Y)}(f,g) \simeq \Mor_{\Omega(\mC,X)}(\id, f^R \circ g), $
which is an $n-2$-category. 

\end{proof}

\begin{proposition}\label{dimest0}
Let $\mC$ be a monoidal $\infty$-category such that every object of $\mC$ admits a right dual and every morphism $\tu_\mC \to X$ of $\mC$ admits a right adjoint.
Let $n > 1$. The $\infty$-category $\mC$ is an $n$-category
if and only if the $\infty$-category
$\Omega(\mC)$ is an $n-1$-category.

\end{proposition}

\begin{proof}

We prove the non-trivial direction.
We assume that the $\infty$-category
$\Omega(\mC)$ is an $n-1$-category.
We want to see that for every $X,Y \in \mC$ the $\infty$-category
$\Mor_\mC(X,Y) $ is an $n-1$-category. 
Since $X$ admits a right dual $X^\vee, $ there is a canonical equivalence
$\Mor_\mC(X,Y) \simeq \Mor_\mC(\tu_\mC, X^\vee \ot Y).$
So it suffices to see that for every $X \in \mC$ the $\infty$-category
$\Mor_\mC(\tu_\mC, X) $ is an $n-1$-category. 
The latter is equivalent to say that for every $ f, g \in \Mor_\mC(\tu_\mC, X)$ the $\infty$-category
$\Mor_{\Mor_\mC(\tu_\mC, X)}(f,g) $ is an $n-2$-category. 
Since $f: \tu_\mC \to X $ admits a right adjoint $f^R: X \to \tu_\mC, $
by \cref{whiteheads} there is a canonical equivalence
$\Mor_{\Mor_\mC(\tu_\mC,X)}(f,g) \simeq \Mor_{\Omega(\mC)}(\id, f^R \circ g), $
which is an $n-2$-category by assumption. 

\end{proof}

\begin{corollary}
Let $\mC$ be an $\infty$-category such that every 1-morphism of $\mC$ admits a right adjoint and for every $X \in \mC$ and $1 \leq m < n-1 $ every morphism
$\tu \to T$ in $\Omega^m(\mC,X) $ admits a right adjoint.

Let $0 < m < n$. The $\infty$-category $\mC$ is an $n$-category
if and only if for every $X \in \mC $ the $\infty$-category
$\Omega^{m}(\mC,X)$ is an $n-m$-category.

\end{corollary}

\begin{proof}
We prove the non-trivial direction. Let $n > 1.$

We proceed by induction on $m$, where $ n > m > 0.$ 
For $m=1$ the statement is \cref{dimest}. 
Let $m > 1 $ and we assume the statement for $m-1$.
Let $\mC$ be an $\infty$-category such that every 1-morphism of $\mC$ admits a right adjoint and for every $X \in \mC$ and $1 \leq m < n-1 $ every morphism
$\tu \to T$ in $\Omega^m(\mC,X) $ admits a right adjoint.
We assume that for every $X \in \mC $ the $\infty$-category
$\Omega^{m}(\mC,X)$ is an $n-m$-category.

Then every object of $\Omega^{m-1}(\mC,X)$ admits a right dual and every morphism $\tu \to T$ in $\Omega^{m-1}(\mC,X)$ admits a right adjoint. So by \cref{dimest0} the $\infty$-category $\Omega^{m-1}(\mC,X)$ is an $n-m+1$-category.
Hence by induction hypothesis, $\mC$ is an $n$-category.

\end{proof}

\begin{proposition}\label{whicon}
Let $\mC$ be a monoidal $\infty$-category such that every object admits a right dual and for every $m \geq 0$ every morphism $\tu \to T$ in $\Omega^m(\mC) $ admits a right adjoint.

Let $n \geq 1.$
A monoidal functor $F: \mC \to \mD$ is an $n$-equivalence if and only if the induced functor 
\begin{equation*}
\Pi_\m(\mC) \to \Pi_\m(\mD)
\end{equation*} is an equivalence for every $0 \leq \m < n.$

In particular, $F: \mC \to \mD$ is an equivalence if and only if the induced graded functor $$\Pi_*(\mC) \to \Pi_*(\mD) $$ is an equivalence.

\end{proposition}

\begin{proof}
We prove the statement by induction on $n \geq 1.$
For $n=1$ there is nothing to show since a functor is an 1-equivalence if and only if it induces an equivalence on the homotopy categories.
Let $n > 1$. We assume the statement holds for $n-1$ and that the induced functor $
\Pi_\m(\mC) \to \Pi_\m(\mD)$ is an equivalence for every $0 \leq \m < n.$

We want to see that $F: \mC \to \mD$ is an $n$-equivalence.
By the induction start, we know that $F: \mC \to \mD$ is an $1$-equivalence. Hence $F: \mC \to \mD$ is an $n$-equivalence if and only if for every $X, Y \in \mC$ the induced functor
$$ \Mor_\mC(X,Y) \to \Mor_\mD(F(X), F(Y)) $$ induces $n-2 $-equivalences on morphism $\infty$-categories.
Since every object of $\mC$ admits a dual, the functor 
$$ \Mor_\mC(X,Y) \to \Mor_\mD(F(X), F(Y)) $$ identifies with the functor
$$ \Mor_\mC(\tu_\mC, X^\vee \ot Y) \to \Mor_\mD(\tu_\mD, F(X^\vee \ot Y)).$$
So it suffices to see that the latter functor induces $n-2 $-equivalences on morphism $\infty$-categories.

By \cref{whiteheadlemma} for every $f,g \in \Mor_\mC(\tu_\mC, X^\vee \ot Y)$ the induced functor
$$ \Mor_{\Mor_\mC(\tu_\mC, X^\vee \ot Y)}(f,g) \to \Mor_{\Mor_\mD(\tu_\mD, F(X^\vee \ot Y))}(F(f),F(g)) $$
identifies with the functor
$$ \Mor_{\Omega(\mC)}(\id,f^R \circ g) \to \Mor_{\Omega(\mD)}(\id, F(f^R \circ g)) $$
because $f : \tu_\mC \to X^\vee \ot Y$ admits a right adjoint $f^R: X^\vee \ot Y \to \tu_\mC $ in $\mC.$ 
So it suffices to see that the monoidal functor $\Omega(\mC) \to \Omega(\mD)$ is an $n-1$-equivalence.

The induced functor $
\Pi_\m(\mC) \to \Pi_\m(\mD)$ identifies with the functor
\begin{equation*}
\Pi_{\m-1}(\Omega(\mC)) \to \Pi_{\m-1}(\Omega(\mD))
\end{equation*}
for every $0 < \m < n.$
Hence the functor \begin{equation*}
\Pi_{\m}(\Omega(\mC)) \to \Pi_{\m}(\Omega(\mD))
\end{equation*}
is an equivalence for every $0 \leq \m < n-1. $
So by induction hypothesis, the monoidal functor 
$\Omega(\mC) \to \Omega(\mD)$ is an $n-1$-equivalence.

\end{proof}

Next we prove the following Whitehead theorem:

\begin{theorem}\label{homten} Let $\mC$ be an $\infty$-category such that every morphism of $\mC$ admits a right adjoint and for every $X \in \mC$ and $m \geq 1$ every morphism
$\tu \to T$ in $\Omega^m(\mC,X) $ admits a right adjoint.

Let $n \geq 1.$ 
A functor $F: \mC \to \mD$ is an $n$-equivalence if and only if for every $X \in \mC$ the induced functor 
\begin{equation*}\label{retku}
\Pi_\m(\mC, X) \to \Pi_\m(\mD, F(X))
\end{equation*} is an equivalence for every $\m < n.$

In particular, $F: \mC \to \mD$ is an equivalence if and only if for every $X \in \mC$ the induced graded functor $$\Pi_*(\mC, X) \to \Pi_*(\mD, F(X))$$ is an equivalence.
	
\end{theorem}

\begin{proof}

For $n=1$ the statement is tautological. So let $n > 1.$
If $F: \mC \to \mD$ is an $n$-equivalence, it is an $\m+1$-equivalence for any $\m < n.$ 
Thus for every $X \in \mC$ the induced functor $\Omega^{\m}(\mC,X) \to \Omega^{\m}(\mD, F(X))$ is a 1-equivalence, i.e. induces an equivalence on homotopy categories, so that the functor \begin{equation*}\label{retku}
\Pi_\m(\mC, X) \to \Pi_\m(\mD, F(X))
\end{equation*} is an equivalence for every $\m < n.$

We prove the converse.

We assume that for every $X \in \mC$ the induced functor 
$$ \Pi_\m(\mC,X) \to \Pi_\m(\mD, F(X))
$$ is an equivalence for every $\m < n$ and have to see that $F$ is an $n$-equivalence.
Since $\Pi_0(F)$ is an equivalence, the functor $F$ is a 1-equivalence.
Hence $F$ is an $n$-equivalence if for every $X,Y \in \mC$ and $f,g \in \Mor_\mC(X,Y)$ the induced functor $$ \Mor_{\Mor_\mC(X,Y)}(f,g) \to \Mor_{\Mor_\mD(F(X), F(Y))}(F(f), F(g)) $$ is an $n-2 $-equivalence.
By assumption $f: X \to Y $ admits a right adjoint $f^R: Y \to Y$ and by \cref{whiteheads} the latter functor identifies with the functor 
$$ \Mor_{\Omega(\mC,X)}(\id,f^R \circ g) \to \Mor_{\Omega(\mD,F(X))}(\id,F(f^R) \circ F(g)).$$

So it suffices to see that the induced functor 
$$ \Omega(\mC,X) \to \Omega(\mD,F(X)) $$ is an $n-1$-equivalence.

By \cref{whicon} this functor is an $n-1$-equivalence if and only if for every $0 \leq \m < n-1 $ the induced functor 
\begin{equation*}
\Pi_\m(\Omega(\mC,X)) \to \Pi_\m(\Omega(\mD,F(X)))
\end{equation*} is an equivalence.
The latter functor is the functor
\begin{equation*}
\Pi_{\m+1}(\mC,X) \to \Pi_{\m+1}(\mD,F(X)),
\end{equation*}
which is an equivalence by assumption.

\end{proof}

\begin{corollary}\label{maf} Let $\mC$ be a categorical spectrum with left (right) adjoints.
A map $\F: \mC \to \mD$ of categorical spectra is an equivalence if and only if the induced $\bZ$-graded functor $$ \Pi_*(\mC) \to \Pi_*(\mD)$$ is an equivalence.

\end{corollary}

\begin{proof}
	
The only-if implication is clear. A map of categorical spectra $\F: \mC \to \mD$ is an equivalence 
if and only if for every $n \in \bZ$ the induced functor $ \Omega^\infty(\Omega^n(\F))$ is an equivalence. The latter functor underlies a symmetric monoidal functor whose source  admits left (right) duals.
Hence by \cref{whicon} it suffices to verify that for every $\ell \geq 0$ the functor $$\Pi_\ell(\Omega^\infty(\Omega^n(\mC)), \tu) \to \Pi_\ell(\Omega^\infty(\Omega^n(\mD)),\tu)$$ is an equivalence. The latter functor identifies with the functor $\Pi_{n+\ell}(\mC) \to \Pi_{n+\ell}(\mD)$.
	
\end{proof}

\subsection{The long exact sequence on homotopy monoidal categories}

\subsubsection{Oriented fiber sequences}

\begin{definition}\label{oplaxexact} Let $\mC$ be a reduced oriented category.
An oriented exact sequence in $\mC$ is a sequence
$$ ...  \xrightarrow{\beta_1} C_1 \xrightarrow{\delta_0} A_0 \xrightarrow{\alpha_0} B_0 \xrightarrow{\beta_0} C_0 $$
equipped with 2-morphisms $$ 0 \to \beta_n \circ \alpha_n, $$$$ 0 \xrightarrow{\simeq} \alpha_n \circ \delta_n, $$$$ 0 \to \delta_n \circ \beta_{n+1} $$ such that 
\begin{itemize}
\item for every $n \geq 0 $ the morphism $\delta_\n $ is the fiber of $\alpha_\n$,
\item for every $n \geq 0 $ the morphism $ \alpha_\n$ is the oriented fiber of $ \beta_{\n}$.
\end{itemize}

\end{definition}

\begin{proposition}\label{oplaxexacto} 
Let $\phi: A \to B$ be a morphism in $\infty\Cat_*^\adj.$
There is an oriented exact sequence in $\infty\Cat_*^\adj:$
$$... \to \Omega^2(B) \to 0 \vectimes_{\Omega(B)} \Omega(A) \to \Omega(A) \to \Omega(B) \to 0 \vectimes_B A \to A \to B $$

\end{proposition}

\begin{proof}

By the pasting law for oriented pullbacks \cref{pasting} there is a fiber sequence
$$  \Omega(B) \to 0 \vectimes_B A \to A . $$
The right adjoint bioriented functor $\Omega: \infty\Cat^\cop \to \infty\Cat$ preserves oriented pullbacks and so sends the sequence $\Omega(B) \to 0 \vectimes_B A \to A \to B$
to a sequence $\Omega^2(B) \to \Omega(A) \vectimes_{\Omega(B)} 0 \to \Omega(A) \to \Omega(B), $
where the first functor is a fiber and the second functor is an oriented fiber. We apply the canonical equivalence $$\Omega(A) \vectimes_{\Omega(B)} 0 \simeq 0 \vectimes_{\Omega(B)} \Omega(A) $$
of \cref{symsym}.

By \cite[Lemma 4.3.16.]{heine2026stable}
There is a canonical base point preserving functor $$\xi: 0 \overset{\to}{+}_{\Omega(A)} \Omega(B) \to 0 \overset{\to}{\times}_B A  $$
corresponding to an oriented commutative square
\[
\begin{tikzcd}
\Omega(A) \ar{r}{} \ar{d}[swap]{} & \Omega(B)  \ar{d}{} \\
0 \ar[double]{ur}{}  \ar{r}[swap]{} & 0 \overset{\to}{\times}_B A 
\end{tikzcd}
\]
Hence we obtain a canonical oriented exact sequence
$$... \to \Omega^2(B) \to 0 \vectimes_{\Omega(B)} \Omega(A) \to \Omega(A) \to \Omega(B) \to 0 \vectimes_B A \to A \to B.$$

\end{proof}

For the next lemma we remark that the canonical functor $\bD^1 \boxtimes \bD^1 \to \bD^1 \times \bD^1$ is an epimorphism since for every $\infty$-category $X$ the induced map
$$ \Map_{\infty\Cat}(\bD^1, \Fun(\bD^1, X)) \simeq \Map_{\infty\Cat}(\bD^1 \times \bD^1, X) \to $$$$ \Map_{\infty\Cat}(\bD^1 \boxtimes \bD^1, X) \simeq \Map_{\infty\Cat}(\bD^1, \Fun^\oplax(\bD^1, X)) $$ 
is an embedding as it is induced by the inclusion $\Fun(\bD^1, X) \to \Fun^\oplax(\bD^1, X).$ 

\begin{lemma}
Let $X$ be a category and $X \boxtimes \bD^1 \to Y$ a functor.
There is a subcategory $X^\nat \subset X$ with the same objects and whose morphisms correspond to functors $\bD^1 \to X$ such that the composition
$\bD^1 \boxtimes \bD^1 \to X \boxtimes \bD^1 \to Y$ factors through
$\bD^1 \times \bD^1.$

\end{lemma}

\begin{proof}

For every object $Z$ of $X$ the constant functor $\bD^1 \to * \to \bD^1$ with value $Z$ satisfies this condition, i.e. the composition
$\bD^1 \boxtimes \bD^1 \to X \boxtimes \bD^1 \to Y$ factors through
$\bD^1 \times \bD^1$ since it factors through $\bD^0 \boxtimes \bD^1.$
For every composable morphisms of $X$ corresponding to functors $\bD^1 \to X$ such that the composition
$\bD^1 \boxtimes \bD^1 \to X \boxtimes \bD^1 \to Y$ factors through
$\bD^1 \times \bD^1,$ the induced functor
$$ (\bD^1 \boxtimes \bD^1)\coprod_{(\bD^0 \boxtimes \bD^1)} (\bD^1 \boxtimes \bD^1) \simeq (\bD^1 \coprod_{\bD^0} \bD^1) \boxtimes \bD^1 \to X \boxtimes \bD^1 \to Y $$ factors through $(\bD^1 \times \bD^1)\coprod_{(\bD^0 \times \bD^1)} (\bD^1 \times \bD^1) \simeq (\bD^1 \coprod_{\bD^0} \bD^1) \times \bD^1 $.
Hence the composition
$$ \bD^1 \boxtimes \bD^1 \to (\bD^1 \coprod_{\bD^0} \bD^1) \boxtimes \bD^1 \to X \boxtimes \bD^1 \to Y, $$ induced by the functor
$\bD^1 \simeq \{0< 2 \} \subset \{0< 1 < 2\} \simeq \bD^1 \coprod_{\bD^0} \bD^1$, factors through $\bD^1 \times \bD^1 $.

\end{proof}

\begin{remark}

Let $X$ be a category. A functor $X \boxtimes \bD^1 \to Y$ corresponds to a lax natural transformation $\xi$ of functors $X \to Y.$
The subcategory $X^\nat \subset X$ is the maximal subcategory on which
the lax natural transformation $\xi$ of functors $X \to Y$ restricts to a natural transformation of functors $X \to Y.$

\end{remark}

\begin{theorem}\label{seqfib}

Let $\phi: X \to Y$ be a fibration of based $\infty$-categories with adjoints. 
The induced long sequence
$$ ... \to \Pi_2(Y) \xrightarrow{\Pi_1(\delta)} \Pi_1(\oker(\phi)) \xrightarrow{\Pi_1(p)} \Pi_1(X) \xrightarrow{\Pi_1(\phi)} \Pi_1(Y) \xrightarrow{\Pi_0(\delta)} \Pi_0(\oker(\phi)) \xrightarrow{\Pi_0(p)} \Pi_0(X) \xrightarrow{\Pi_0(\phi)} \Pi_0(Y) $$
in $\Cat_*$ has the following exactness properties:

\begin{enumerate}
\item The canonical functor $$ \Pi_{n+1}(Y) \to \ker(\Pi_n(p))$$ is essential surjective, full and isofull.

\item The canonical functor $$ \Pi_{n}(\oker(\phi)^\nat) \to \oker(\Pi_n(\phi))$$ is essential surjective, full and isofull.

\item The canonical functor $$ \Pi_{n+1}(X^\nat) \to \oker(\Pi_n(\delta))$$ is essential surjective and isofull.
    
\end{enumerate}

\end{theorem}

\begin{proof}

It suffices to consider the case $n=0.$

(1): We start with essential surjectivity.
Let $([(Z, \kappa: \tu \to \phi(Z))], [\rho]: [Z] \simeq [\tu]) \in \ker(\Pi_0(p)).$
Then $ [\phi(\rho) \circ \kappa] \in \Pi_1(Y) $ and 
$$ \Pi_0(\delta)([\phi(\rho) \circ \kappa]) = [(\tu, \tu \xrightarrow{\phi(\rho) \circ \kappa} \phi(\tu))].$$
Moreover $$ ([(\rho: \Z \simeq \tu, \id_{\phi(\rho) \circ \kappa}: \phi(\rho) \circ \kappa \to \phi(\rho) \circ \kappa)], [\rho]: [Z] \simeq [\tu])$$ is an equivalence
$$ ([(Z, \kappa: \tu \to \phi(Z))], [\rho]: [Z] \simeq [\tu]) \to ([(\tu, \tu \xrightarrow{\phi(\rho) \circ \kappa} \phi(\tu))], [\id_\tu]: [\tu] \to [\tu]) $$ in $\ker(\Pi_0(p)).$

We continue with fullness.
Let $ [(\alpha: \tu \to \tu)], [(\beta: \tu \to \tu)] \in \Pi_1(Y)$
and let $$([(f: \tu \to \tu, \sigma: \phi(f) \circ \alpha \to \beta)], [\tau]:[f] \simeq [\id_\tu])$$ be a morphism
$$ ([(\tu, \alpha)], [\id_\tu]: [\tu] \to [\tu]) \to ([(\tu, \beta)], [\id_\tu]: [\tu] \to [\tu]) $$
in $\ker(\Pi_0(p)).$

Then $$ [\alpha \xrightarrow{\phi(\tau^{-1})\circ \alpha} \phi(f) \circ \alpha \xrightarrow{\sigma} \beta] $$
is a morphism $ [(\alpha: \tu \to \tu)] \to [(\beta: \tu \to \tu)] $ in $ \Pi_1(Y)$ whose image in $\ker(\Pi_0(p))$ is the morphism
$$ ([(\id_\tu, \sigma \circ (\phi(\tau^{-1})\circ \alpha) : \alpha \to \beta)], [\id_{\id_\tu}]).$$
Moreover $$([(\tau,\id_\sigma)],[(\tau, \id_{\id_\tu})]) $$ is an equivalence
$$([(f: \tu \to \tu, \sigma: \phi(f) \circ \alpha \to \beta)], [\tau]:[f] \simeq [\id_\tu])$$
to 
$$ ([(\id_\tu, \sigma \circ (\phi(\tau^{-1})\circ \alpha) : \alpha \to \beta)], [\id_{\id_\tu}])$$
in $\ker(\Pi_0(p))$.
This shows fullness and isofullness.

\vspace{1mm}

(2): We start with essential surjectivity.
Let $([Z], [\kappa]: [\tu] \to [\phi(Z)]) \in \oker(\Pi_0(\phi)) $.
Then $$[(Z, \kappa: \tu \to \phi(Z))] \in \Pi_0(\oker(\phi)) $$ is sent to
$([Z], [\kappa]: [\tu] \to [\phi(Z)]) \in \oker(\Pi_0(\phi)).$

We continue with fullness. Let $[(Z, \kappa: \tu \to \phi(Z))], [(Z', \kappa': \tu \to \phi(Z'))] \in \Pi_0(\oker(\phi)) $ and 
$$([f]:[Z] \to [Z'], [\sigma]: [\phi(f)] \circ [\kappa] = [\kappa']) $$ a morphism $$([Z], [\kappa]: [\tu] \to [\phi(Z)]) \to ([Z'], [\kappa']: [\tu] \to [\phi(Z')])$$ in $\oker(\Pi_0(\phi)).$ 
Then $[(f:Z \to Z', \sigma: \phi(f) \circ \kappa \simeq \kappa')] $ is a morphism $$[(Z, \kappa: \tu \to \phi(Z))] \to [(Z', \kappa': \tu \to \phi(Z'))]$$ in $\Pi_0(\oker(\phi)^\nat).$ 
This shows fullness and isofullness by the description of equivalences in oriented fibers.

\vspace{1mm}

(3): Let $$ (  [(\alpha: \tu \to \tu)], [ \beta: \tu \to \tu, \theta : \phi(\beta) \to \alpha) ] : [(\tu, \id_\tu)] \to [(\tu, \alpha: \tu \to \tu)]  ) $$
be an object of $\oker(\Pi_0(\delta)).$
Since $\phi$ is a fibration, the 2-morphism $\theta : \phi(\beta) \to \alpha)$ in $Y$ admits a lift $\beta \to \alpha'$ in $\Mor_X(\tu,\tu).$
Hence $\phi(\alpha') \simeq \alpha,$ where $\alpha' \in \Mor_X(\tu,\tu). $

Next we prove isofullness.
Let $$ [(\alpha: \tu \to \tu)], [(\beta: \tu \to \tu)] \in \Pi_1(X)$$
and $$ \zeta: ([(\sigma: \phi(\alpha) \to \phi(\beta)) ], [(\kappa: \alpha \simeq \beta, \lambda: \phi(\kappa) \simeq \sigma))]:[(\alpha, \sigma)]= [(\id_\tu, \sigma)] \circ [(\alpha, \id_{\phi(\alpha})] \simeq [(\beta, \id_{\phi(\beta})])$$
a morphism in $\oker(\Pi_0(\delta))$
from $$ ([(\phi(\alpha): \tu \to \tu)], [(\alpha, \id_{\phi(\alpha})] : [(\tu, \id_\tu)] \to [(\tu, \phi(\alpha): \tu \to \tu)]  ) $$
to $$ ([(\phi(\beta): \tu \to \tu)], [(\beta, \id_{\phi(\beta})] : [(\tu, \id_\tu)] \to [(\tu, \phi(\beta): \tu \to \tu)]  ).$$

The image of the morphism $ [(\kappa : \alpha \simeq \beta)] $ of $ \Pi_1(X^\nat) $ in $\oker(\Pi_0(\delta)) $ is the morphism
$$ ([(\phi(\kappa): \phi(\alpha) \to \phi(\beta)) ], [(\kappa: \alpha \simeq \beta, \id_{\phi(\kappa)})]:[(\alpha, \id_{\phi(\alpha})] \simeq [(\beta, \id_{\phi(\beta})]),$$
which is equivalent to $\zeta.$
This proves isofullness.

\end{proof}

\subsection{Homotopical richness}

\begin{proposition}\label{presef}
Let $\phi: A \to B$ be a cocartesian fibration of $\infty$-categories whose fiber transports are weak equivalences.
Let $B' \to B$ be a functor.
The induced map
$$ \tau_0(B' \times_B A) \to \tau_0(B') \times_{\tau_0(B)} \tau_0(A) $$ is an equivalence.

\end{proposition}

\begin{proof}

By \cite[Proposition 3.5.10.]{gepner2026colimits} the functor $\tau_0: \infty\Cat \to \mS$ preserves the fiber of $\phi: A \to B$ over every object of $B.$
We want to see that the induced map
$$ \tau_0(B' \times_B A) \to \tau_0(B') \times_{\tau_0(B)} \tau_0(A) $$ is an equivalence. The latter map is a map over $\tau_0(B') $
and so an equivalence if it induces an equivalence on the fiber over every $t \in \tau_0(B').$ The induced map on the fiber over $t \in \tau_0(B')$ is the canonical map
$$ \{t\}\times_{\tau_0(B')} \tau_0(B' \times_B A) \to \{t\}\times_{\tau_0(B')} \tau_0(B') \times_{\tau_0(B)} \tau_0(A) \simeq \{t\} \times_{\tau_0(B)} \tau_0(A) .$$
The canonical maps $$ \tau_0(\{t\} \times_B A) \simeq \tau_0(\{t\}\times_{B'} B' \times_B A) \to \{t\}\times_{\tau_0(B')} \tau_0(B' \times_B A) $$
and $$ \tau_0(\{t\} \times_B A) \to \{t\} \times_{\tau_0(B)} \tau_0(A)$$ are equivalences because the fiber transports of the cocartesian fibration $\phi: A \to B$ and so also of the pullback $B' \times_B A \to B'$ are weak equivalences. 
The canonical map
$$ \tau_0(\{t\} \times_B A) \simeq \{t\}\times_{\tau_0(B')} \tau_0(B' \times_B A) \to \{t\}\times_{\tau_0(B')} \tau_0(B') \times_{\tau_0(B)} \tau_0(A) \simeq \{t\} \times_{\tau_0(B)} \tau_0(A) \simeq \tau_0(\{t\} \times_B A)$$ is the identity. 

\end{proof}

\begin{proposition}\label{laxadji}

Let $Z$ be an $\infty$-category and $Y \to Z$ a functor.
The diagonal functor $\delta: Z \to \Fun^\oplax(\bD^1,Z)$ is a section of evaluation at the target $\ev_1: \Fun^\oplax(\bD^1,Z) \to Z.$
There is a lax natural transformation $ \alpha: \id \to \delta \circ \ev_1$ over $Z$.

\end{proposition}

\begin{proof}

There is a canonical equivalence $$ \Fun^\lax(\Fun^\oplax(\bD^1,Z),  \Fun^\oplax(\bD^1,Z)) \simeq \Fun^\oplax(\bD^1,\Fun^\lax(\Fun^\oplax(\bD^1,Z), Z)),$$
under which the identity corresponds to a morphism
$f: \ev_0 \to \ev_1 $ in $\Fun^\lax(\Fun^\oplax(\bD^1,Z), Z).$

The commutative square
$$\begin{xy}
\xymatrix{
\ev_0 \ar[r]^{f}\ar[d]^f & \ev_1  \ar[d]^=
\\ 
\ev_1 \ar[r]^{=} & \ev_1
}
\end{xy}$$
in $\Fun^\lax(\Fun^\oplax(\bD^1,Z), Z)$
gives a morphism in
$\Fun^\oplax(\bD^1,\Fun^\lax(\Fun^\oplax(\bD^1,Z), Z))$,
which under the equivalence
$$ \Fun^\oplax(\bD^1,\Fun^\lax(\Fun^\oplax(\bD^1,Z), Z)) \simeq \Fun^\lax(\Fun^\oplax(\bD^1,Z),  \Fun^\oplax(\bD^1,Z))$$
corresponds to a morphism $ \alpha: \id \to \delta \circ \ev_1$ in
$\Fun^\lax(\Fun^\oplax(\bD^1,Z),  \Fun^\oplax(\bD^1,Z)),$
i.e. a lax natural transformation of functors $\Fun^\oplax(\bD^1,Z) \to \Fun^\oplax(\bD^1,Z).$
Moreover the commutative square implies that the lax natural transformation $\ev_1 \circ \alpha: \ev_1 \to \ev_1 \circ \delta \circ \ev_1\simeq \ev_1$ of functors $\Fun^\oplax(\bD^1,Z) \to Z $ is the identity.
Thus $ \alpha: \id \to \delta \circ \ev_1$ is a lax natural transformation over $Z.$

\end{proof}

\begin{corollary}\label{laxadjo}

Let $Z$ be an $\infty$-category and $Y \to Z$ a functor.
The diagonal functor $Z \to \Fun^\oplax(\bD^1,Z)$ induces an equivalence
$$ \tau_0(Y) \simeq \tau_0(Y \times_Z Z) \to \tau_0(Y \times_{\Fun^\oplax(\{1\},Z)} \Fun^\oplax(\bD^1,Z)) .$$

\end{corollary}

\begin{proof}

Let $p: Y \times_{\Fun^\oplax(\{1\},Z)} \Fun^\oplax(\bD^1,Z) \to Y$ be the projection. Let $$d: Y \simeq Y \times_Z Z \to  Y \times_{\Fun^\oplax(\{1\},Z)} \Fun^\oplax(\bD^1,Z)$$ be the pullback of the diagonal functor.

Pulling back the lax natural transformation $ \alpha: \id \to \delta \circ \ev_1$ of functors $\Fun^\oplax(\bD^1,Z) \to \Fun^\oplax(\bD^1,Z)$ over $Z$ of \cref{laxadji} along $Y \to Z$ gives a lax natural transformation $ \alpha': \id \to d \circ p $ of functors $Y \times_Z \Fun^\oplax(\bD^1,Z) \to Y \times_Z \Fun^\oplax(\bD^1,Z)$ over $Y.$
Thus $\tau_0(\alpha') : \id \to \tau_0(d) \circ \tau_0(p)$ is an equivalence and so $\tau_0(p)$ is an inverse to $\tau_0(d)$
since trivially also $\tau_0(p) \circ \tau_0(d) \simeq \id$.

\end{proof}

\begin{corollary}\label{laxadj2}

Let $Z$ be an $\infty$-category and $s \in Z$.
The $\infty$-category $Z_{//^\oplax s} $ is weakly contractible.

\end{corollary}

\begin{theorem}\label{pullclass}

Let $X \to Z, Y \to Z$ be $\infty$-categories such that for every morphism $s \to t$ in $Y$ the canonical functor $ Z_{t //^\oplax} \times_Z Y \to Z_{s //^\oplax} \times_Z Y $ is a weak equivalence.
The canonical map of spaces
$$ \tau_0(X \overset{\to}{\times}_Z Y) \to \tau_0(X) \times_{\tau_0(Y)} \tau_0(Z)$$
is an equivalence.
    
\end{theorem}

\begin{proof}

By \cref{presef} for every cocartesian fibration $\phi: A \to B$ of $\infty$-categories whose fiber transports are weak equivalences the functor $\tau_0: \infty\Cat \to \mS$ preserves pullbacks of $\phi$.
Dually, by applying $(-)^{\co\op},$ for every cartesian fibration $\phi: A \to B$ of $\infty$-categories whose fiber transports are weak equivalences the functor $\tau_0: \infty \Cat \to \mS$ preserves pullbacks of $\phi$.

By \cite[Corollary 4.1.8.]{gepner2026fibrations} the canonical functor $\Fun^\oplax(\bD^1,Z) \to Z \times Z$ 
is a bifibration, i.e. cartesian fibration on the first factor and cocartesian fibration on the second factor.
For every morphism $s \to t$ in $Z$ the induced functor
$$Z_{//^\oplax s} \to Z_{//^\oplax t} $$ is a weak equivalence because its source and target are weakly contractible by \cref{laxadj2}. 
Hence any pullback of evaluation at the target
$\Fun^\oplax(\bD^1,Z) \to Z$ is preserved by $\tau_0.$
Thus the following canonical map is an equivalence
$$ \tau_0(\Fun^\oplax(\bD^1,Z)\times_{\Fun^\oplax(\{1\},Z)} Y) \to  \tau_0(\Fun^\oplax(\bD^1,Z)) \times_{\tau_0(\Fun^\oplax(\{1\},Z))} \tau_0(Y).$$
The latter map is a map over $\tau_0(\Fun^\oplax(\{0\},Z)).$
Hence the following map is an equivalence
$$ \tau_0(X) \times_{\tau_0(\Fun^\oplax(\{0\},Z))} \tau_0( \Fun^\oplax(\bD^1,Z)\times_{\Fun^\oplax(\{1\},Z)} Y) \to $$$$ \tau_0(X) \times_{\tau_0(\Fun^\oplax(\{0\},Z))}  \tau_0(\Fun^\oplax(\bD^1,Z)) \times_{\tau_0(\Fun^\oplax(\{1\},Z))} \tau_0(Y).$$

The functor $$ \Fun^\oplax(\bD^1,Z) \times_{\Fun^\oplax(\{1\},Z)} Y \to \Fun^\oplax(\bD^1,Z) \to \Fun^\oplax(\{0\},Z)$$ is a cartesian fibration whose fiber transport of any morphism $s \to t$ in $Y$ is the canonical functor $ Z_{t //^\oplax} \times_Z Y \to Z_{s //^\oplax} \times_Z Y.$
The latter is a weak equivalence by assumption.
Hence the canonical functor
$$ \tau_0(X \times_{\Fun^\oplax(\{0\},Z)} \Fun^\oplax(\bD^1,Z) \times_{\Fun^\oplax(\{1\},Z)} Y) \to $$$$ \tau_0(X) \times_{\tau_0(\Fun^\oplax(\{0\},Z))} \tau_0(\Fun^\oplax(\bD^1,Z) \times_{\Fun^\oplax(\{1\},Z)} Y) $$ 
is an equivalence.

Hence the following canonical map is an equivalence
$$ \tau_0(X \overset{\to}{\times}_Z Y) \simeq \tau_0(X \times_{\Fun^\oplax(\{0\},Z)} \Fun^\oplax(\bD^1,Z) \times_{\Fun^\oplax(\{1\},Z)} Y) \to $$$$ \tau_0(X) \times_{\tau_0(\Fun^\oplax(\{0\},Z))} \tau_0( \Fun^\oplax(\bD^1,Z)\times_{\Fun^\oplax(\{1\},Z)} Y) \to $$$$ \tau_0(X) \times_{\tau_0(\Fun^\oplax(\{0\},Z))} \tau_0(\Fun^\oplax(\bD^1,Z)) \times_{\tau_0(\Fun^\oplax(\{1\},Z))} \tau_0(Y).$$

By \cref{laxadjo} the diagonal functor $Z \to \Fun^\oplax(\bD^1,Z)$
induces an equivalence $$\tau_0(Z) \to \tau_0(\Fun^\oplax(\bD^1,Z)).$$
Hence the induced functor 
$$ \tau_0(X) \times_{\tau_0(Z)} \tau_0(Y) \simeq \tau_0(X) \times_{\tau_0(Z)} \tau_0(Z) \times_{\tau_0(Z)} \tau_0(Y) \to $$$$ \tau_0(X) \times_{\tau_0(\Fun^\oplax(\{0\},Z))} \tau_0(\Fun^\oplax(\bD^1,Z)) \times_{\tau_0(\Fun^\oplax(\{1\},Z))} \tau_0(Y)$$ is an equivalence.

\end{proof}

\begin{corollary}

Let $X \to Z, Y \to Z$ be $\infty$-categories such that every 1-morphism of $Z$ admits a right adjoint.
The canonical map of spaces
$$ \tau_0(X \overset{\to}{\times}_Z Y) \to \tau_0(X) \times_{\tau_0(Y)} \tau_0(Z)$$
is an equivalence.
    
\end{corollary}

We obtain that the classifying space functor preserves oriented pullbacks:

\begin{theorem}\label{pullclass}

Let $X \to Z, Y \to Z$ be $\infty$-categories with adjoints.
The canonical map of spaces
$$ \tau_0(X \overset{\to}{\times}_Z Y) \to \tau_0(X) \times_{\tau_0(Y)} \tau_0(Z)$$
is an equivalence.
    
\end{theorem}

\begin{theorem}\label{superfull}

\begin{enumerate}
\item The canonical functor $$ (\tau_0)_!(\Cat\Sp^\adj) \to \Sp $$
is an equivalence.

\item The canonical functor $$ (\tau_0)_!(\infty\Cat^\adj) \to \mS $$
is an equivalence.
\end{enumerate}

\end{theorem}

\begin{proof}

We prove (1). (2) is similar.

We prove first that $(\tau_0)_!(\Cat\Sp^\adj)$
is a stable category that admits small colimits.

Since for every $\infty$-category $X$ the canonical functor
$X \to \tau_0(X)$ is essentially surjective, the canonical functor $\Cat\Sp^\adj \to (\tau_0)_!(\Cat\Sp^\adj)$ is full.

So every pair of morphisms $\alpha: C \to A, \beta: C \to B $ in $(\tau_0)_!(\Cat\Sp^\adj)$ lifts to morphisms $\bar{\alpha}: C \to A, \bar{\beta}: C \to B $ in 
$\Cat\Sp^\adj$ and the oriented pushout 
$ A \vec{+}_C B $ in $\Cat\Sp^\adj$ induces an equivalence
$$ \Mor_{\Cat\Sp^\adj}(A \vec{+}_C B, D) \to \Mor_{\Cat\Sp^\adj}(A, D) \vec{\times}_{\Mor_{\Cat\Sp^\adj}(C, D)} \Mor_{\Cat\Sp^\adj}(B, D).$$ 
By \cref{pullclass} the latter induces an equivalence
$$ \tau_0(\Mor_{\Cat\Sp^\adj}(A \vec{+}_C B, D)) \to \tau_0(\Mor_{\Cat\Sp^\adj}(A, D)) \vec{\times}_{\tau_0(\Mor_{\Cat\Sp^\adj}(C, D))} \tau_0(\Mor_{\Cat\Sp^\adj}(B, D)),$$
which is the canonical equivalence
$$ \Mor_{\tau_!(\Cat\Sp^\adj)}(A \vec{+}_C B, D) \to \Mor_{\tau_!(\Cat\Sp^\adj)}(A, D) \vec{\times}_{\Mor_{\tau_!(\Cat\Sp^\adj)}(C, D)} \Mor_{\tau_!(\Cat\Sp^\adj)}(B, D).$$
The latter exhibits the oriented pushout $A \vec{+}_C B$ of
$\bar{\alpha}: C \to A, \bar{\beta}: C \to B $ in 
$\Cat\Sp^\adj$ as the pushout of $\alpha: C \to A, \beta: C \to B $ in $(\tau_0)_!(\Cat\Sp^\adj)$.

Similarly, every pair of morphisms $\alpha: A \to C, \beta: B \to C $ in $(\tau_0)_!(\Cat\Sp^\adj)$ lifts to morphisms $\bar{\alpha}: A \to C, \bar{\beta}: B \to C $ in 
$\Cat\Sp^\adj$ and the oriented pullback
$ A \vec{\times}_C B $ in $\Cat\Sp^\adj$ induces an equivalence
$$ \Mor_{\Cat\Sp^\adj}(D, A \vec{\times}_C B) \to \Mor_{\Cat\Sp^\adj}(D,A) \vec{\times}_{\Mor_{\Cat\Sp^\adj}(D, C)} \Mor_{\Cat\Sp^\adj}(D, B).$$ 
By \cref{pullclass} the latter induces an equivalence
$$ \tau_0(\Mor_{\Cat\Sp^\adj}(D, A \vec{\times}_C B)) \to \tau_0(\Mor_{\Cat\Sp^\adj}(D, A)) \vec{\times}_{\tau_0(\Mor_{\Cat\Sp^\adj}(D, C))} \tau_0(\Mor_{\Cat\Sp^\adj}(D, B)),$$
which is the canonical equivalence
$$ \Mor_{\tau_!(\Cat\Sp^\adj)}(D, A \vec{\times}_C B) \to \Mor_{\tau_!(\Cat\Sp^\adj)}(D, A) \vec{\times}_{\Mor_{\tau_!(\Cat\Sp^\adj)}(D, C)} \Mor_{\tau_!(\Cat\Sp^\adj)}(D, B).$$
The latter exhibits the oriented pullback $A \vec{\times}_C B$ of
$\bar{\alpha}: A \to C, \bar{\beta}: B \to C $ in 
$\Cat\Sp^\adj$ as the pullback of $\alpha: A \to C, \beta: B \to C $ in $(\tau_0)_!(\Cat\Sp^\adj)$.

In particular, for every $X \in \Cat\Sp^\adj$ the endomorphisms $\Omega(X)$ are the loops of $X$ in $(\tau_0)_!(\Cat\Sp^\adj)$.
Thus the functor of loops 
$$ (\tau_0)_!(\Cat\Sp^\adj) \to (\tau_0)_!(\Cat\Sp^\adj)$$
arises by applying $(\tau_0)_!$ to the functor of loops 
$$ \Cat\Sp^\adj \to \Cat\Sp^\adj,$$ and so is an equivalence.
Hence $(\tau_0)_!(\Cat\Sp^\adj)$ is stable.
To see that $(\tau_0)_!(\Cat\Sp^\adj)$ admits small colimits, it suffices to see that it admits small coproducts.
For every set $\mJ$ and family $X_j, j \in \mJ$
the induced equivalence
$$ \Mor_{\Cat\Sp^\adj}(\coprod_{j \in \mJ} X_j, D) \to \prod_{j \in \mJ} \Mor_{\Cat\Sp^\adj}(X_j, D) $$ 
induces an equivalence
$$ \tau_0(\Mor_{\Cat\Sp^\adj}(\coprod_{j \in \mJ} X_j, D) \to \prod_{j \in \mJ} \tau_0(\Mor_{\Cat\Sp^\adj}(X_j, D)), $$ 
which is the canonical equivalence
$$ \Mor_{\tau_!(\Cat\Sp^\adj)}(\coprod_{j \in \mJ} X_j, D) \to \prod_{j \in \mJ} \Mor_{\tau_!(\Cat\Sp^\adj)}(X_j, D).$$
The latter exhibits the coproduct of $X_j, j \in \mJ $ in $\Cat\Sp^\adj$ as the coproduct of $X_j, j \in \mJ $ in $(\tau_0)_!(\Cat\Sp^\adj)$.
This proves that $(\tau_0)_!(\Cat\Sp^\adj)$
is a stable category that admits small colimits.

By the construction of pushouts and small coproducts, the 
canonical functor $(\tau_0)_!(\Cat\Sp^\adj) \to \Sp $
preserves pushouts and small coproducts and so preserves small colimits. 

Since $(\tau_0)_!(\Cat\Sp^\adj)$
is a stable category that admits small colimits, by \cite[Corollary 1.4.4.5.]{lurie.higheralgebra} there is a unique left adjoint functor $\theta: \Sp \to (\tau_0)_!(\Cat\Sp^\adj)$ that sends the sphere spectrum to the categorical sphere spectrum.
The composition $\Sp \to (\tau_0)_!(\Cat\Sp^\adj) \to \Sp$
preserves small colimits and preserves the sphere spectrum and so by \cite[Corollary 1.4.4.5.]{lurie.higheralgebra} is the identity.

The composition $$\Cat\Sp^\adj \xrightarrow{\tau_0} \tau_0^*(\Sp) \xrightarrow{\tau_0^*(\theta)} \tau_0^*((\tau_0)_!(\Cat\Sp^\adj)) $$ of left adjoint oriented functors preserves the categorical sphere spectrum and so by \cite[Theorem 5.3.19.]{heine2026stable} is the canonical oriented functor.
Hence applying $(\tau_0)_!$ also the composition $$(\tau_0)_!(\Cat\Sp^\adj) \to \Sp \xrightarrow{\theta} (\tau_0)_!(\Cat\Sp^\adj) $$ is the identity. 

\end{proof}

\begin{corollary}\label{tauequ}
For every $A, B \in \infty\Cat^\adj$ 
the canonical functor $$ \tau_0(\Mor_{\infty\Cat^\adj}(A,B)) \to \Map_{\mS}(\tau_0(A),\tau_0(B)) $$
is an equivalence.

For every $A, B \in \Cat\Sp^\adj$ 
the canonical functor $$ \tau_0(\Mor_{\Cat\Sp^\adj}(A,B)) \to \Map_{\Sp}(\tau_0(A),\tau_0(B)) $$
is an equivalence.

\end{corollary}

\begin{corollary}\label{superfull}

The oriented functor $$\infty\Cat^\adj \to \mS$$ is $\infty$-connected, i.e. it is essentially surjective and 
for every $A, B \in \infty\Cat^\adj$ the canonical functor $$ \Mor_{\infty\Cat^\adj}(A,B) \to \Map_{\mS}(\tau_0(A),\tau_0(B)) $$
is $\infty$-connected.

The oriented functor $$\Cat\Sp^\adj \to \Sp$$ is $\infty$-connected, i.e. it is essentially surjective and 
for every $A, B \in \Cat\Sp^\adj$ the canonical functor $$ \Mor_{\Cat\Sp^\adj}(A,B) \to \Map_{\Sp}(\tau_0(A),\tau_0(B)) $$
is $\infty$-connected.

\end{corollary}

\begin{proof}

By \cref{Fulll} the canonical functor $$ \Mor_{\infty\Cat^\adj}(A,B) \to  \tau_0(\Mor_{\infty\Cat^\adj}(A,B)) $$ is $\infty$-connected.
So the result follows from \cref{tauequ}.

By \cref{Fulll} the canonical functor $$ \Mor_{\Cat\Sp^\adj}(A,B) \to  \tau_0(\Mor_{\Cat\Sp^\adj}(A,B)) $$ is $\infty$-connected.
So the result follows from \cref{tauequ}.

\end{proof}

\subsection{Categorical chain complexes}

\subsubsection{Skeletal sequences}

For the following we use \cref{defcon}.

\begin{definition}
Let $1 \leq n \leq \infty$ and $2 \leq i \leq n+1.$
An $i$-skeletal sequence of length $n$ is a sequence
$$ X_0 \to ... \to X_{n-1} $$ of symmetric monoidal $\infty$-categories
(which for $n=\infty$ is understood as $ X_0 \to X_1 \to ... $) such that 
\begin{enumerate}
\item for every $0 \leq j < \min\{i-1,n \} $ the $\infty$-category
$X_{j}$ is an $j+1$-category 
\item the symmetric monoidal category $X_0$ has duals, 
\item for every $1 \leq j < i-1$ every morphism $\tu \to T$ in $\Omega^{j-1}(X_{j})$ admits a left adjoint,
\item for every $0 \leq j < n-1$
the functor $X_j \to X_{j+1}$ is $\min\{i-2,j\}$-connective.

\end{enumerate}

For $n=\infty$ we call $i$-skeletal sequences of length $\infty$ simply
$i$-skeletal sequences.

\end{definition}

\begin{remark}

Let $1 \leq n \leq \infty$ and $2 \leq i \leq i' \leq n+1.$
Every $i'$-skeletal sequence is in particular $i$-skeletal.

\end{remark}

\begin{definition}Let $1 \leq n \leq \infty$.
A skeletal sequence of length $n$ is an $n+1$-skeletal sequence of length $n$.
For $n=\infty$ we call skeletal sequences of length $\infty$ simply
skeletal sequences.

\end{definition}

\begin{definition}Let $2 \leq n \leq \infty$.
A weak skeletal sequence of length $n$ is an $n$-skeletal sequence of length $n$.

\end{definition}

\begin{remark}Let $2 \leq n \leq \infty$.
Every skeletal sequence of length $n$ is a weak skeletal sequence of length $n$. 
For $n=\infty$ there is no difference between weak skeletal sequences of length $\infty$ and skeletal sequences of length $\infty$,
which we call skeletal sequences.

\end{remark}

\begin{example}\label{exampweaksk}

\begin{itemize}
\item A weak skeletal sequence of length 2 is an essentially surjective symmetric monoidal functor $X \to Y$
such that $X$ is a symmetric monoidal category that admits duals. 

\item A skeletal sequence of length 2 is an essentially surjective symmetric monoidal functor $X \to Y$
such that $X$ is a symmetric monoidal category that admits duals
and $Y$ is a symmetric monoidal 2-category such that every morphism
$\tu \to T$ in $Y$ admits a left adjoint.

\end{itemize}

\end{example}

\begin{remark}
Let $2 \leq n \leq \infty$ and $2 < i \leq n+1.$

Let $$ X_0 \to ... \to X_{n-1} $$ be a sequence of symmetric monoidal $\infty$-categories.
Then we can form the sequence $$ \Omega(X_1) \to ... \to \Omega(X_{n-1}) $$ of symmetric monoidal $\infty$-categories.

If $$ X_0 \to ... \to X_{n-1} $$ is an $i$-skeletal sequence of length $n$, then $ \Omega(X_1) \to ... \to \Omega(X_{n-1}) $ is an $i-1$-skeletal sequence of length $n-1$. 
Moreover if $$ X_0 \to ... \to X_{n-1} $$ is an $i$-skeletal sequence of length $n$, then $X_0 \to X_1$ is a weak skeletal sequence of length $2$.

\vspace{1mm}
Hence for every $0 \leq j < n-1 $ and $j+1 < i \leq n+1$ for every
$i$-skeletal sequence $$ X_0 \to ... \to X_{n-1} $$ of length $n$
the symmetric monoidal functors $$ X_0 \to X_1, \Omega(X_1) \to \Omega(X_2), ..., \Omega^{j-1}(X_{j-1}) \to \Omega^{j-1}(X_{j})$$
are weak skeletal sequences of length $2$ and the symmetric monoidal functor $$\Omega^j(X_j) \to \Omega^j(X_{j+1})$$
is a
\[
\begin{cases}
\text{weak skeletal sequence of length } 2, & \text{if } i-j \geq 2,\\
\text{skeletal sequence of length } 2, & \text{if } i-j \geq 3.
\end{cases}
\]

\end{remark}

\begin{proposition}\label{weakstrict}

Let $X \to Y$ be a weak skeletal sequence of length 2.

The weak skeletal sequence $X \to Y$ of length 2 is a skeletal sequence of length 2 if and only if the weak categorical chain complex $0 \vectimes_Y X \to X $ of length 2 is a categorical chain complex of length 2.

\end{proposition}

\begin{proof}

We have to see the following:
The $\infty$-category $Y$ is a 2-category and every morphism $\tu \to T$ in $Y$ admits a left adjoint if and only $0 \overset{\to}{\times}_{Y} X$ is a symmetric monoidal category that admits duals.
Since the functor $0 \overset{\to}{\times}_{Y} X \to X$ 
is a cocartesian fibration and $X$ is a category, $0 \overset{\to}{\times}_{Y} X$ is a category if and only if the fibers
of the cocartesian fibration $0 \overset{\to}{\times}_{Y} X \to X$ are categories. Since the functor $X \to Y$ is essentially surjective, this is equivalent to say that for every object $Z \in Y$
the $\infty$-category $\Mor_{Y}(\tu, Z) $ is a 1-category.
Since the functor $X \to Y$ is essentially surjective and $X$ admits duals, also the symmetric monoidal $\infty$-category $Y$ admits duals. Hence $\Mor_{Y}(\tu, Z) $ is a 1-category of every $Z \in Y$
if and only if all morphism $\infty$-categories of $Y$ are 1-categories. Hence $0 \overset{\to}{\times}_{Y} X$ is a category if and only if $Y$ is a 2-category. 
So it remains to see that every morphism $\tu \to T$ in $Y$ admits a left adjoint if and only $0 \overset{\to}{\times}_{Y} X$ admits duals.
Since the functor $X \to Y$ is essentially surjective, also the functor $$ 0 \overset{\to}{\times}_{Y} X \simeq (0 \overset{\to}{\times}_{Y} Y) \times_{Y} X \to 0 \overset{\to}{\times}_{Y} Y $$ is essentially surjective. 
Thus $0 \overset{\to}{\times}_{Y} Y$ admits duals if
$0 \overset{\to}{\times}_{Y} X$ admits duals.
On the other hand, $0 \overset{\to}{\times}_{Y} X$ admits duals if
$0 \overset{\to}{\times}_{Y} Y$ admits duals since the pullback of symmetric monoidal categories with duals, has duals. 
Consequently, we have to see that every morphism $\tu \to T$ in $Y$ admits a left adjoint if and only $0 \overset{\to}{\times}_{Y} Y$ admits duals. By \cref{adjjo} the dual $(\beta: \tu \to T^\vee) \in 0 \overset{\to}{\times}_{Y} Y$ of an object $(\alpha: \tu \to T) \in 0 \overset{\to}{\times}_{Y} Y$ corresponds under the duality to a left adjoint $\beta^\vee: T \to \tu $ of $\alpha: \tu \to T.$

\end{proof}

For the next notation we set $$ [\infty]:= (\bN, \leq).$$

\begin{notation}
Let $1 \leq n \leq \infty$ and $2 \leq i \leq n+1.$
Let $$ \mathrm{Sk}^i_n \subset \Fun([n-1], \Cmon(\infty\Cat)) $$ be the full subcategory of $i$-skeletal sequence of length $n$ and let
$$ \mathrm{Sk}_n := \mathrm{Sk}^n_n $$ the category of weak skeletal sequences of length $n.$

Moreover we write $$ \mathrm{Sk} := \mathrm{Sk}_\infty$$
for the category of skeletal sequences.

\end{notation}

\begin{definition}Let $1 \leq n \leq \infty.$ 
A dimension-increasing sequence of length $n$ is a sequence 
$$ X_0 \to ... \to X_{n-1} $$ of symmetric monoidal $\infty$-categories
(which for $n=\infty$ is understood as $ X_0 \to X_1 \to ... $)
such that for every $0 \leq j < n$ the $\infty$-category
$X_{j}$ is an $j+1$-category.

For $n=\infty$ we call dimension-increasing sequences of length $\infty$ simply dimension-increasing sequences.

\end{definition}

\begin{definition}Let $1 \leq n \leq \infty$.
A strong skeletal sequence of length $n$ is a dimension-increasing sequence 
$$ X_0 \to ... \to X_{n-1} $$ of length $n$ such that 
\begin{enumerate}
\item for every $0 \leq j < n $ the symmetric monoidal $j+1$-category
$X_{j}$ admit duals for every object and left adjoints for every morphism below dimension $j+1$,

\vspace{1mm}
\item for every $0 \leq j < n $
the functor $X_j \to X_{j+1}$ is $j+1$-connective.

\end{enumerate}

For $n=\infty$ we call strong skeletal sequences of length $\infty$ simply strong skeletal sequences.

\end{definition}

\begin{remark}

Let $2 \leq n \leq \infty$.
Every strong skeletal sequence of length $n$ is a skeletal sequence of length $n$.
Every skeletal sequence of length $n$ is a dimension-increasing sequence of length $n$.

On the other hand, weak skeletal sequences of length $n$ are not required to be dimension-increasing sequences of length $n$. See \cref{exampweaksk}.

\end{remark}

\begin{example} Let $n \geq 2$ and $X$ a symmetric monoidal $\infty$-category with duals.
The sequence $$ \iota_1(X) \to \iota_2(X) \to ... \to \iota_{n}(X) $$ is a strong skeletal sequence of length $n$, which we call the standard (strong) skeletal sequence of length $n.$

The sequence $$ \iota_1(X) \to \iota_2(X) \to ... $$ is a strong skeletal sequence, which we call the standard (strong) skeletal sequence.

\end{example}

\begin{notation}

Let $1 \leq n \leq \infty.$
Let $$ \mathrm{Seq}_n \subset \Fun([n-1], \Cmon(\infty\Cat)) $$ be the full subcategory of dimension-increasing sequences of length $n$.
Moreover we write
$$ \mathrm{Seq} := \mathrm{Seq}_\infty $$
for the category of dimension-increasing sequences.

\end{notation}

\begin{proposition}\label{skeladj} Let $1 \leq n \leq \infty.$ 
There is a localization $$ \mathrm{Seq}_n \rightleftarrows \Cmon(n\Cat) $$ whose left adjoint takes the colimit and whose fully faithful right adjoint sends a symmetric monoidal $n$-category $X$ to the standard skeletal sequence $$ \iota_1(X) \to ... \to \iota_n(X).$$

\end{proposition}

\begin{proof}
Let $R: n\Cat \to \mathrm{Seq}_n $ be the functor sending a symmetric monoidal $n$-category $X$ to the standard skeletal sequence $$ \iota_1(X) \to ... \to \iota_n(X).$$
For every symmetric monoidal $n$-category $X$ and dimension-increasing sequences $Y_0 \to ... \to Y_{n-1}$ of symmetric monoidal $\infty$-categories of length $n$ the induced map
$$ \Map_{\Cmon(n\Cat)}(\colim(Y_0 \to ... \to Y_{n-1}),X) \to \Map_{\mathrm{Seq}_n}(Y_0 \to ... \to Y_{n-1}, \iota_1(X) \to ... \to \iota_n(X)) $$
is an equivalence.

\end{proof}

\begin{definition}\label{dime} Let $n \in \bZ.$

A based $\infty$-category is $n$-dimensional if
\[
\begin{cases}
\text{it is an } n\text{-category}, & n > 0,\\
\text{it is contractible}, & n \leq 0.
\end{cases}
\]

\end{definition}

\begin{definition}

Let $\infty \leq n \leq \infty$. A categorical spectrum is $n$-dimensional if for every $m \geq 0$ the $\infty$-category $\Omega^\infty(\Sigma^m(X))$ is $n+m$-dimensional.

\end{definition}

\begin{example}

An $\infty$-dimensional categorical spectrum is precisely a categorical spectrum.

Let $n \geq 0.$
A categorical spectrum $X$ is $- n$-dimensional if and only if $\Omega^\infty(\Sigma^i(X))=*$ for $ 0 \leq i \leq n.$

\end{example}

\begin{notation}
Let $n \in \bZ.$
Let $n\Sp \subset \Sp$ be the full subcategory of $n$-dimensional categorical spectra.

\end{notation}

\begin{definition}

Let $-\infty < n \leq \infty$. A categorical spectrum admits adjoints of dimension smaller $n$ if for every $m \geq 0$ the $\infty$-category $\Omega^\infty(\Sigma^m(X))$ admits adjoints of dimension smaller $n+m.$

\end{definition}

\begin{definition}
Let $\infty < n \leq \infty$. An $n$-dimensional categorical spectrum with adjoints is a $n$-dimensional categorical spectrum that admits adjoints of dimension smaller $n.$

\end{definition}

\begin{remark}\label{standardsk}
Let $n \in \bZ.$ The embedding of the full subcategory of $n$-dimensional based $\infty$-categories into $\infty\Cat_*$ admits a left adjoint
$\widetilde{\iota}_{n}$ that agrees with $\iota_n$ for $n >0$
and sends everything to 0 for $n \leq 0.$

\vspace{1mm}

The embedding $n\Sp \subset \Sp$ admits a left adjoint
$\widetilde{\iota}_{n}$ that sends a categorical spectrum $X$ to the
categorical spectrum $\{ \widetilde{\iota}_{n+m}(X_m), m \geq 0 \}$
with evident bonding equivalences.

If $X$ is a categorical spectrum with adjoints, $\widetilde{\iota}_{n}(X)$
is an $n$-dimensional categorical spectrum with adjoints.

For every categorical spectrum $X$ the induced map
$$ \colim_{n \in \bZ} \widetilde{\iota}_{n}(X) \to X $$ is an equivalence
since it induces under the functor $\Omega^\infty \circ \Sigma^m$
the induced map
$$ \colim_{n \in \bZ} \widetilde{\iota}_{n+m}(X_m) \to X_m, $$
which is an equivalence.

\end{remark}

\begin{definition}
A dimension-increasing sequence of categorical spectra is a sequence of
categorical spectra $... \to Y_{-1} \to Y_0 \to Y_1 \to ... $ indexed by the integers such that $Y_n $ is an $n$-dimensional categorical spectrum
for every $n \in \bZ.$

\end{definition}

\begin{notation}
Let $$ \mathrm{Seq}\Cat\Sp \subset \Fun([\infty], \Cat\Sp)$$
be the full subcategory of dimension-increasing sequences of categorical spectra.

\end{notation}

\begin{proposition}\label{standardsk0}

There is an localization $$\mathrm{Seq}\Cat\Sp \rightleftarrows \Cat\Sp $$
whose left adjoint takes the colimit and whose fully faithful right adjoint sends a categorical spectrum $X$ to the standard skeletal sequence $$ ...\to \widetilde{\iota}_{-1}(X)  \to \widetilde{\iota}_0(X) \to \widetilde{\iota}_1(X) \to \widetilde{\iota}_2(X) \to ... $$

\end{proposition}

\begin{proof}

Let $Y$ be a dimension-increasing sequences of categorical spectra
and $X$ a categorical spectrum.
\cref{standardsk} implies that induced map
$$ \Map_{\Cat\Sp}(\colim(Y),X) \to  \Map_{\mathrm{Seq}\Cat\Sp}(Y,...\to \widetilde{\iota}_{-1}(X)  \to \widetilde{\iota}_0(X) \to \widetilde{\iota}_1(X) \to ...) $$
is an equivalence.

\end{proof}

\begin{definition}

Let $n \geq 0$. A map $X \to Y$ of categorical spectra is $n$-connective if for every $m \geq 0$ the induced functor $$\Omega^\infty(\Sigma^{m}(X)) \to \Omega^\infty(\Sigma^{m}(Y)) $$ is $n+m$-connective.
By convention every functor is $n$-connective for any $n < 0.$

\end{definition}

\begin{definition}

A skeletal sequence of categorical spectra is a dimension-increasing sequence of categorical spectra
$$ ...\to X_{-1} \to X_0 \to X_1 \to ... $$ such that 
\begin{itemize}
\item for every $i \in \bZ $ the map $X_i \to X_{i+1}$ is $i$-connective, 
\item the symmetric monoidal $\infty$-category $\Omega^\infty(\Sigma^i(X_{-i})$ admits duals for $ i \geq 0$,

\item every morphism $\tu \to T$ in 
$\Omega^\infty(\Sigma^{k-i+1}(X_{i-k}))$ admits a left adjoint for every $ i \geq k \geq 0.$

\end{itemize}

\end{definition}

\begin{remark}

A skeletal sequence of categorical spectra is a diagram
\[
\begin{tikzcd}[column sep=small, row sep=small]
\vdots
&
\vdots
&
\vdots
&
\vdots
&
\\
X_{-2,2}
  \arrow[r]
&
X_{-1,2}
  \arrow[r]
&
X_{0,2}
  \arrow[r]
&
X_{1,2}
  \arrow[r]
&
\cdots
\\
X_{-1,1}
  \arrow[r]
&
X_{0,1}
  \arrow[r]
&
X_{1,1}
  \arrow[r]
&
X_{2,1}
  \arrow[r]
&
\cdots
\\
X_{0,0}
  \arrow[r]
&
X_{1,0}
  \arrow[r]
&
X_{2,0}
  \arrow[r]
&
X_{3,0}
  \arrow[r]
&
\cdots
\end{tikzcd}
\]
where the diagonal sequences $(X_{i,0}, X_{i,1}, X_{i,2},...) $
form categorical spectra, the $i$-th column (starting with the 0th column)
consists of $i$-categories that admit adjoints of morphisms below dimension $i$ and each functor from the $i$-th column to the $i+1$-th column is $i$-connective.

In other words, each row in this diagram is a skeletal sequence,
which arises from the row before by applying the functor 
$s: \mathrm{Sk} \to \mathrm{Sk}$
sending $X_0 \to X_1 \to ... $ to $\Omega(X_1) \to \Omega(X_2) \to ... $.

\end{remark}

\begin{definition}
Let $$ \Sk\Cat\Sp := \lim( ... \xrightarrow{s} \mathrm{Sk} \xrightarrow{s} \mathrm{Sk})$$
be the category of skeletal sequences of categorical spectra.

\end{definition}

\subsubsection{Categorical chain complexes}

\begin{proposition}\label{dualindu}

Let $\mC, \mD$ be symmetric monoidal $\infty$-categories and 
$\phi: \mC \to \mD$ a symmetric monoidal cocartesian fibration.
Assume that every object of $\mC$ admits a dual.

Let $X, Y \in \mC$ and $f: \phi(X) \to \phi(Y)$ a morphism in $\mD$.
There is a canonical equivalence
$$ \Mor_\mC(X,Y) \times_{\Mor_\mD(\phi(X), \phi(Y))} \{f\} \simeq \Mor_{\mC \times_\mD \{\tu_\mD\}}(f'_!(X \ot Y^\vee), \tu_\mC). $$ 

\end{proposition}

\begin{proof}
Since $Y \in \mC$ admits a dual $Y^\vee$, also $\phi(Y^\vee)$ is a dual of $\phi(Y).$

Let $f': \phi(X \ot Y^\vee) \to \tu_\mD $ be the image of $f$ under the canonical equivalence
$$ \Mor_\mD(\phi(X), \phi(Y)) \simeq \Mor_\mD(\phi(X) \ot \phi(Y^\vee), \tu_\mD) \simeq \Mor_\mD(\phi(X \ot Y^\vee), \tu_\mD).$$ 

There is a canonical equivalence
$$ \Mor_\mC(X,Y) \times_{\Mor_\mD(\phi(X), \phi(Y))} \{f\} \simeq \Mor_\mC(X \ot Y^\vee, \tu_\mC) \times_{\Mor_\mD(\phi(X \ot Y^\vee), \tu_\mD)} \{f'\} $$$$ \simeq \Mor_{\mC \times_\mD \{\tu_\mD\}}(f'_!(X \ot Y^\vee), \tu_\mC). $$ 

\end{proof}

\begin{corollary}\label{dimel} Let $n > 0.$
Let $\mC, \mD$ be symmetric monoidal $\infty$-categories and 
$\phi: \mC \to \mD$ a symmetric monoidal cocartesian fibration.
Assume that every object of $\mC$ admits a dual.
If $\mD$ and the fiber of $\phi: \mC \to \mD$ over the tensor unit are $n$-categories, then also $\mC$ is an $n$-category.

\end{corollary}

\begin{proof}

We have to see that for every $X,Y \in \mC$ the $\infty$-category $\Mor_\mC(X,Y) $ is an $n-1$-category.
By assumption the $\infty$-category $\Mor_\mD(\phi(X), \phi(Y)) $ is an $n-1$-category.
Since $\phi: \mC \to \mD$ is a cocartesian fibration, for every $X,Y \in \mC$ the induced functor $\Mor_\mC(X,Y) \to \Mor_\mD(\phi(X), \phi(Y))$
is a cartesian fibration. Therefore by \cite[Corollary 3.3.4.]{gepner2026fibrations}, the $\infty$-category $\Mor_\mC(X,Y) $ is an $n-1$-category if the fibers of the cartesian fibration $\Mor_\mC(X,Y) \to \Mor_\mD(\phi(X), \phi(Y))$ are $n-1$-categories. By \cref{dualindu} the fiber of this cartesian fibration over 
any morphism $f: \phi(X) \to \phi(Y)$ of $\mD$ is equivalent to 
$ \Mor_{\mC \times_\mD \{\tu_\mD\}}(f'_!(X \ot Y^\vee), \tu_\mC),$
which is an $n-1$-category.

\end{proof}

\begin{definition}Let $X, Y$ be symmetric monoidal $\infty$-categories. 
A symmetric monoidal functor $\phi: X \to Y$ is a cocartesian fibration of symmetric monoidal $\infty$-categories if the underlying functor of $\phi$ is a cocartesian fibration of $\infty$-categories and the collection of $\phi$-cocartesian morphisms is closed under the tensor product.

\end{definition}

\begin{definition}\label{weakcha}

Let $1 < n \leq \infty.$
A weak categorical chain complex of length $n$ is a sequence
$$((\phi_m: X_m \to Z_{m-1})_{0 < m < n},(\rho_m)_{0 < m < n-1} ),$$
where $\phi_m: X_m \to Z_{m-1}$ is a cocartesian fibration of symmetric monoidal $\infty$-categories, $\rho_m$ is a symmetric monoidal equivalence
$$ \rho_m: \ker(\phi_m):= 0 \times_{Z_{m-1}} X_m \simeq Z_{m}$$
such that 
\begin{enumerate}
\item for every $0 < m < n $ the symmetric monoidal $\infty$-category $Z_{m-1}$ is a symmetric monoidal category that admits duals and
\item for every $0 < m < n-1$ the the symmetric monoidal $\infty$-category $ X_m $ is a symmetric monoidal category that admits duals. 

\end{enumerate}

\end{definition}

\begin{definition}Let $1 < n < \infty.$
A categorical chain complex of length $n$ is a weak categorical chain complex of length $n$ such that the symmetric monoidal $\infty$-category $ X_{n-1}$ is a symmetric monoidal category that admits duals.

\end{definition}

\begin{remark}
By \cref{dimel} condition (2) of \cref{weakcha} can be replaced by the following weaker condition: for every $0 < m < n-1$ the the symmetric monoidal $\infty$-category $ X_m $ admits duals. 

\end{remark}

We are mainly interested in the case $n=\infty,$ for which we reserve the following terminology:

\begin{definition}
A connective categorical chain complex is a weak categorical chain complex of length $\infty$.

\end{definition}

\begin{remark}
A weak categorical chain complex of length $n$ is likewise a sequence
$$((\phi_m: X_m \to Z_{m-1})_{0 < m < n},(\rho_m)_{0 < m < n-1} )$$
where $\phi_m: X_m \to Z_{m-1}$ is a cocartesian fibration of symmetric monoidal $\infty$-categories and $\rho_m$ is a symmetric monoidal equivalence
$ \rho_m: 0 \times_{Z_{m-1}} X_m \simeq Z_{m}$
such that $X_m$ is a symmetric monoidal category that admit duals for every $0 \leq m < n-1$ and $Z_{m-1}$ is a symmetric monoidal category that admit duals for every $0 \leq m < n$ and $Z_{-1} \simeq 0$.
The equivalence
$$ \rho_0: X_0 \simeq 0 \times_{Z_{-1}} X_0 \simeq Z_{0} $$
is equivalent to the identity of $Z_0$ and so redundant if we set $X_0:= Z_0.$

\end{remark}

\begin{example} We have the following important base cases:

\begin{itemize}

\vspace{1mm}
\item A weak categorical chain complex of length $2$ is a 
symmetric monoidal functor $X \to Y$ of symmetric monoidal $\infty$-categories such that $Y$ is a symmetric monoidal category that admits duals.

\vspace{1mm}
\item A categorical chain complex of length $2$ is a 
symmetric monoidal functor $X \to Y$ of symmetric monoidal categories that admit duals.

\end{itemize}

\end{example}

\begin{notation}
Let
$$ \Cat\Ch^{\mathrm{w}}_2 \subset \Cmon(\coCart)$$
be the full subcategory of weak categorical chain complexes of length $2$.

Let
$$ \Cat\Ch_2 \subset \Cat\Ch^{\mathrm{w}}_2$$
be the full subcategory of categorical chain complexes of length $2$.

\end{notation}

\begin{definition}Let $n > 2.$

\begin{itemize}
\item The category of categorical chain complexes of length $n$
is the pullback 
$$ \Cat\Ch_n := \Cat\Ch_2 \times_{\Cmon(\infty\Cat)} ... \times_{\Cmon(\infty\Cat)}\Cat\Ch_2 $$
of $n$-1 copies of $\Cat\Ch_2,$ where the pullback 
is formed along the functors taking the fiber and evaluating at the target.

\item The category of weak categorical chain complexes of length $n$
is the pullback 
$$ \Cat\Ch_n := \Cat\Ch^{\mathrm{w}}_2 \times_{\Cmon(\infty\Cat)} ... \times_{\Cmon(\infty\Cat)}\Cat\Ch_2 $$
of $n$-2 copies of $\Cat\Ch_2$ and $ \Cat\Ch^{\mathrm{w}}_2$, where the pullback 
is formed along the functors taking the fiber and evaluating at the target.

\end{itemize}

\end{definition}

\begin{notation}

Let $n > 2.$
Let $$\xi_n: \Cat\Ch_n \to \Cat\Ch_{n-1} $$ be the functor sending 
$$ ((\phi_m: X_m \to Z_{m-1})_{0 < m < n},(\rho_m)_{0 < m < n-1} )$$ to $$ ((\phi_m: X_m \to Z_{m-1})_{0 < m < n-1},(\rho_m)_{0 < m < n-2} ).$$

\end{notation}

\begin{definition}
The category of connective categorical chain complexes $ \Cat\Ch_{\geq 0} $ is the limit of the tower
$$ ...  \xrightarrow{\xi_4} \Cat\Ch_3 \xrightarrow{\xi_3} \Cat\Ch_2.$$

\end{definition}

\begin{definition}

A categorical chain complex is a sequence
$$((\phi_m: X_m \to Z_{m-1})_{m \in \bZ},(\rho_m)_{m \in \bZ}),$$
where $\phi_m: X_m \to Z_{m-1}$ is a cocartesian fibration of symmetric monoidal categories that admit duals and $\rho_m$ is a symmetric monoidal equivalence
$$ \rho_m: \ker(\phi_m) = 0 \times_{Z_{m-1}} X_m \simeq Z_{m}. $$

\end{definition}

\vspace{1mm}

\begin{notation}

Let $t: \Cat\Ch_{\geq 0} \to \Cat\Ch_{\geq 0}$
be the shift functor sending 
$((\phi_m: X_m \to Z_{m-1})_{m > 0},(\rho_m)_{m > 0})$
to 
$((\phi_{m+1}: X_{m+1} \to Z_{m})_{m > 0},(\rho_{m+1})_{m > 0}).$

\end{notation}

\begin{definition}\label{shifta}

The category of categorical chain complexes $ \Cat\Ch $ is the limit of the tower 
$$ ... \xrightarrow{t} \Cat\Ch_{\geq 0} \xrightarrow{t} \Cat\Ch_{\geq 0}. $$ 

\end{definition}

\vspace{2mm}

\begin{definition}

Let $n \in \bZ.$
The $n$-th homology monoidal category of a categorical chain complex 
$$ A= ((\phi_m: X_m \to Z_{m-1})_{m \in \bZ},(\rho_m)_{m \in \bZ})$$
is $$ H_n(A) := \Ho(Z_{n}).$$

\end{definition}

We prove the following theorem, whose proof was outlined by Lurie
\cite[Theorem 3.3.33.]{lurie2008classification} and whose rigorous proof
was announced by Haugseng in joint work with Nikolaus \cite{HaugsengCT2025}:

\begin{theorem}\label{DK}
Let $ 2 \leq n \leq \infty.$
The canonical functor
$$ \mathrm{Sk}_n \to  \Cat\Ch^\mathrm{w}_n $$
sending $X_0 \to ... \to X_{n-1} $ to $$(0 \overset{\to}{\times}_{X_1} X_0 \to X_0, 0 \overset{\to}{\times}_{\Omega(X_2)} \Omega(X_1) \to \Omega(X_1), ..., 0 \overset{\to}{\times}_{\Omega^{n-1}(X_{n-1})} \Omega^{n-1}(X_{n-2}) \to \Omega^{n-1}(X_{n-2})) $$
is an equivalence.

\end{theorem}

We obtain the following:

\begin{corollary}\label{DK2}
The canonical functor
$$ \Cmon(\infty\Cat^\adj) \to  \Cat\Ch_{\geq 0} $$
sending $X $ to the categorical chain complex
$$ \{\iota_1(\Omega^{k}(0 \overset{\to}{\times}_X X)) \to \iota_1(\Omega^{k}(X))\}_{k \geq 0} $$
is fully faithful.

\end{corollary}

\begin{proof}

By \cref{skeladj} there is an adjunction $\mathrm{Sk} _\infty\rightleftarrows \infty\Cat^\adj$
whose left adjoint takes the colimit and whose right adjoint is fully faithful and sends an $\infty$-category $X$ to the standard skeletal sequence $$ \iota_1(X) \to \iota_2(X) \to ... $$

We restrict the equivalence of \cref{DK} for $n =\infty$ along the right adjoint.

\end{proof}

We consider further consequences.

\begin{corollary}\label{DK4}
The canonical functor
$$ \mathrm{Sk}\Cat\Sp \to \Cat\Ch $$
sending a skeletal sequence of categorical spectra with adjoints $X_0 \to X_1 \to ... $ to the categorical chain complex
$$ \{\iota_1(0 \overset{\to}{\times}_{\Omega^\infty(\Omega^{i}(X_{i+1}))} \Omega^\infty(\Omega^i(X_i))) \to \iota_1(\Omega^\infty(\Omega^i(X_i))), i \in \bZ \}, $$
is an equivalence.

\end{corollary}

\begin{proof}

By \cref{DK2} the canonical functor
$ \Psi: \mathrm{Sk} \to \Cat\Ch_{\geq 0} $
sending $X $ to the categorical chain complex
$$ \{\iota_1(0 \overset{\to}{\times}_{\Omega^i(X_{i+1})} \Omega^i(X_i)) \to \iota_1(\Omega^i(X_i)), i \geq 0\} $$
is fully faithful.

Let $t:  \Cat\Ch_{\geq 0} \to \Cat\Ch_{\geq 0}$
be the functor of \cref{shifta}.
Let $s: \mathrm{Sk} \to \mathrm{Sk}$
be the functor sending $X_0 \to X_1 \to ... $ to $\Omega(X_1) \to \Omega(X_2) \to ... $.
For every $X \in \mathrm{Sk}$
there is a canonical equivalence $$ \Psi(s(X)) \simeq t(\Psi(X)). $$

Hence there is a canonical equivalence $$\Sk\Cat\Sp = \lim(... \xrightarrow{s} \mathrm{Sk} \xrightarrow{s} \mathrm{Sk}) \simeq \lim( ... \xrightarrow{t} \Cat\Ch_{\geq 0} \xrightarrow{t} \Cat\Ch_{\geq 0}) \simeq \Cat\Ch.$$ 

\end{proof}

\begin{corollary}\label{DK3}
The canonical functor
$$ \Cat\Sp^\adj \to  \Cat\Ch $$
sending a categorical spectrum with adjoints $X $ to the categorical chain complex
$$ \{\iota_1(\Omega^{\infty}(\Omega^k(0 \overset{\to}{\times}_X X))) \to \iota_1(\Omega^{\infty}(\Omega^k(X)))\}_{k \in \bZ} $$
is fully faithful and admits a left adjoint.

\end{corollary}

\begin{proof}

By \cref{standardsk0} there is an adjunction $\Sk\Cat\Sp \rightleftarrows \Cat\Sp^\adj$
whose left adjoint takes the colimit and whose right adjoint is fully faithful and sends a categorical spectrum with adjoints $X$ to the standard skeletal sequence $$ ...\to \widetilde{\iota}_{-1}(X)  \to \widetilde{\iota}_0(X) \to \widetilde{\iota}_1(X) \to \widetilde{\iota}_2(X) \to ... $$
We restrict the equivalence of \cref{DK4} along the right adjoint.

\end{proof}

\subsubsection{Cocartesian fibrations}

\begin{proposition}\label{mainlem0}

There is a canonical equivalence of functors
$$ \Cocart \to \Fun([1],\infty\Cat) $$ that sends a cocartesian fibration $ \phi: X \to Y $ to an equivalence 
$$ (X \vectimes_Y 0) \vectimes_{(X \times_Y 0)} 0 \simeq X \vectimes_{X}  0 $$
over $ X \vectimes_Y 0,$
which sends 
$$(x, g: \phi(x) \to \tu, g_!(x) \to \tu)$$
to $$(x \to g_!(x) \to \tu).$$


\end{proposition}

\begin{proof}

Let $\phi: X \to Y$ be a cocartesian fibration.
Then by the axioms of $\phi$-cocartesian morphisms, there is a canonical bijection $\rho$ on equivalence classes of objects
between 
$ (X \vectimes_Y 0) \vectimes_{(X \times_Y 0)} 0 $ and $ X \vectimes_{X} 0 $ sending
$$(x, g: \phi(x) \to \tu, g_!(x) \to \tu)$$
to $$(x \to g_!(x) \to \tu).$$

Next we prove functoriality. Let $S$ be a category.
By \cite[Theorem 4.3.2.]{gepner2026colimits} a functor $S \to \infty\Cat)$ is classified by a cocartesian fibration $ Y \to S.$
A functor $S \to \Fun([1],\infty\Cat)$ over $\infty\Cat$ (via evaluation at the target) is classified by a map $X \to Y$ of cocartesian fibrations over $S$. This functor lands in $\Cocart$ if and only if the map
$X \to Y$ of cocartesian fibrations over $S$ is a cocartesian fibration.

Hence a functor $S \to \Cocart$ over $\infty\Cat$ is classified by a cocartesian fibration $X \to Y$ (which is automatically a map of cocartesian fibrations over $S$). 

Let $K \to S$ a functor. By \cite[Corollary 4.2.7.]{gepner2026fibrations} the induced functor $$ X':=\Fun_S^\lax(K,X) \to Y':= \Fun_S^\lax(K,Y) $$ is a cocartesian fibration.
By what we have proven, there is a canonical bijection $\rho$ on equivalence classes of objects between 
$$ (X' \vectimes_{Y'} 0) \vectimes_{(X' \times_{Y'} 0)} 0 \simeq \Fun_S^\lax(K, (X \vectimes_Y 0) \vectimes_{(X \times_Y 0)} 0) $$ and $$ X' \vectimes_{X'} 0 \simeq \Fun_S^\lax(K,X \vectimes_X 0) $$
over the set of equivalence classes of $ X' \vectimes_{Y'} 0 \simeq \Fun_S^\lax(K,X \vectimes_Y 0)$.
The latter represents an equivalence 
$$ (X \vectimes_Y 0) \vectimes_{(X \times_Y 0)} 0 \simeq X \vectimes_{X} 0 $$
of cocartesian fibrations over $S$ over $ X \vectimes_Y 0 \to S$.

By \cite[Theorem 5.1.8.]{gepner2026fibrations} for every $s \in \rS$ the functor
$ \Cocart_S \to \infty\Cat $ taking the fiber over $s$ is a right adjoint oriented functor and so preserves oriented pullbacks. Hence the equivalence $$ (X \vectimes_Y 0) \vectimes_{(X \times_Y 0)} 0 \simeq X \vectimes_{X}  0 $$
of cocartesian fibrations over $S$ over $ X \vectimes_Y 0 \to S$ induces on the fiber over any $s \in \rS$
an equivalence $$ (X_s \vectimes_{Y_s} 0) \vectimes_{(X_s \times_{Y_s} 0)} 0 \simeq X_s \vectimes_{X_s} 0 $$
over $X_s \vectimes_{Y_s} 0, $ which sends 
$(x, g: \phi(x) \to \tu, g_!(x) \to \tu)$ to $(x \to g_!(x) \to \tu).$

\end{proof}

\begin{corollary}\label{mainlem1}

There is a canonical equivalence of functors
$$ \Fun([1],\infty\Cat) \to \Fun([1],\infty\Cat) $$ that sends a functor $ \phi: A \to B $ to an equivalence 
$$ ((0 \vectimes_B A)\vectimes_A 0) \vectimes_{\Omega(B)} 0 \simeq   (0 \vectimes_B A) \vectimes_{(0 \vectimes_{B} A)} 0 $$
over $ (0 \vectimes_B A)\vectimes_A 0,$
which sends 
$$(g: A \to \tu, X=(A, f: \tu \to \phi(A)), g_!(X) \to \tu)$$
to $$ X \to g_!(X) \to \tu.$$

\end{corollary}

\begin{corollary}\label{mainlem}

There is a canonical equivalence of functors
$$ \Fun([2],\infty\Cat) \to \Fun([1],\infty\Cat) $$ that sends a composable pair of functors $ \phi: A \to B, \psi : B \to C $ to an equivalence 
$$ ((0 \vectimes_B A)\vectimes_A 0) \vectimes_{\Omega(C)} 0 \simeq   (0 \vectimes_B A) \vectimes_{(0 \vectimes_{C} A)} 0 $$
over $ (0 \vectimes_B A)\vectimes_A 0.$
 
\end{corollary}

\begin{proof}

Composing the canonical equivalence of functors
$ \Fun([1],\infty\Cat) \to \Fun([1],\infty\Cat) $ of \cref{mainlem1}
with the functor $ \Fun([2],\infty\Cat) \to \Fun([1],\infty\Cat) $
induced by the functor $[1]= \{0,2\} \subset [2],$ we obtain an equivalence of functors $ \Fun([2],\infty\Cat) \to \Fun([1],\infty\Cat), $
which sends a composable pair of functors $ \phi: A \to B, \psi : B \to C $
to a canonical equivalence
$$ ((0 \vectimes_C A)\vectimes_A 0) \vectimes_{\Omega(C)} 0 \simeq   (0 \vectimes_C A) \vectimes_{(0 \vectimes_{C} A)} 0 $$
over $ (0 \vectimes_C A)\vectimes_A 0.$
The pullback of this equivalence along the canonical map
$$ (0 \vectimes_B A)\vectimes_A 0 \to (0 \vectimes_C A)\vectimes_A 0 $$ 
is the desired equivalence
$$ ((0 \vectimes_B A)\vectimes_A 0) \vectimes_{\Omega(C)} 0 \simeq (0 \vectimes_B A) \vectimes_{(0 \vectimes_{C} A)} 0 $$
over $ (0 \vectimes_B A)\vectimes_A 0.$

\end{proof}

The following is \cref{chainsk}:

\begin{theorem}\label{chainsk1}

The functor
$$\rho: \Sk_2\to \infty\Cat\Ch_2$$
sending a weak skeletal sequence $X_0 \to X_1$ of length 2 to $$ 0 \overset{\to}{\times}_{X_1} X_0 \to X_0, $$ is an equivalence, under
which skeletal sequences of length 2 correspond to categorical chain complexes of length 2.

\end{theorem}

\begin{proof}

In view of \cref{exampweaksk} the statement follows from \cref{chainsk}
since every symmetric monoidal category that admits duals, is closed.
The second part is \cref{weakstrict}.

\end{proof}

\begin{notation}Let $2 \leq \n \leq \infty $.

Let $$ \Fun([\n],\infty\Cat\Ch_2)' := \Fun([\n],\infty\Cat\Ch_2) \times_{\Fun([\n],\Cmon(\infty\Cat)^\vee)}\Cmon(\infty\Cat)^\vee,$$
where the pullback is taken along the functor evaluating at the target  and the diagonal embedding.

\vspace{2mm}

Let $$ \Fun([\n],\Cmon(\coCart))' := \Fun([\n],\Cmon(\coCart)) \times_{\Fun([\n],\Cmon(\infty\Cat))}\Cmon(\infty\Cat),$$
where the pullback is taken along the functor evaluating at the target  and the diagonal embedding.

\end{notation}

\begin{notation}Let $2 \leq \n \leq \infty $ and $2 \leq i \leq n+1.$
Let $$ \Sk_n^i(\infty\Cat\Ch_2):=\Fun([\n-1],\infty\Cat\Ch_2)' \times_{\Fun([\n-1],\Cmon(\infty\Cat)^\vee)} \Sk^{i}_{n},$$
where the pullback is taken along the functor evaluating at the source and the canonical embedding.
\end{notation}

\begin{corollary}\label{eqolz} Let $2 \leq \n \leq \infty$ and $2 \leq i \leq n+1.$
The functor
$$ \Fun([\n],\Cmon(\infty\Cat)) \to \Fun([\n-1],\Cmon(\coCart))'$$
over $\Cmon(\infty\Cat)$
sending $X_0 \to ... \to X_n$ to $ 0 \overset{\to}{\times}_{X_1} X_0 \to ... \to 0 \overset{\to}{\times}_{X_\n} X_0 $ over $X_0$ restricts to an equivalence 
$$ \Sk_{n+1}^{i} \to \Sk_{n}^{i}(\infty\Cat\Ch_2).$$

\end{corollary}

\begin{proof}

It suffices to prove the case $i=2:$
if the functor $$ \Sk_{n+1}^2 \to \Fun([\n-1],\infty\Cat\Ch_2)' $$
is an equivalence, it restricts to an equivalence
$$ \Sk_{n+1}^{i} \to \Sk_{n}^{i}(\infty\Cat\Ch_2)$$
for every $2 \leq i \leq n+1.$

The functor $$ \Sk_{n+1}^2 \to \Fun([\n-1],\infty\Cat\Ch_2)' $$
of the statement is the functor
$$ \Sk_{n+1}^2 \simeq \Fun([n-1],\Sk_2) \times_{\Fun([n-1],\Cmon(\infty\Cat)^\vee)} \Cmon(\infty\Cat)^\vee \to $$$$ \Fun([n-1],\infty\Cat\Ch_2) \times_{\Fun([n-1],\Cmon(\infty\Cat)^\vee)} \Cmon(\infty\Cat)^\vee $$$$ =\Fun([\n-1],\infty\Cat\Ch_2)'$$
induced by the functor 
$$ \Sk_2 \to \infty\Cat\Ch_2$$
of \cref{chainsk1}, where the left hand pullback is taken along the functor evaluating at the source and the right hand pullback is taken along the functor evaluating at the target.
So the result follows from \cref{chainsk1}.

\end{proof}

\begin{lemma}\label{hiris} Let $n \geq 1.$
There is a canonical equivalence
$$\Fun([1],\infty\Cat) \times_{\infty\Cat} \Fun([\n],\Fun([1],\infty\Cat))' \to \infty\Cat
\times_{\Fun([1],\infty\Cat)} \Fun([\n],\Fun([1] \times [1],\infty\Cat))' $$ 
sending a pair $(E \to A, Y_0 \to ... \to Y_{n-1} \to 0 \vectimes_A E \simeq (0 \vectimes_A A) \times_A E $
to the commutative square
$$\begin{xy}
\xymatrix{
(Y_0 \to ... \to Y_{n-1}) \ar[d]  \ar[r] & E \ar[d]
\\ 
0 \vectimes_A A \ar[r] & A,
}
\end{xy}$$ where the pullback is taken along the functor 
$ (E \to A) \mapsto 0 \vectimes_A E $
and evaluation at the target.

\end{lemma}

\begin{proof}

For every category $\mC$ that admits pullbacks the canonical functor
$$ \Fun([1] \times [1], \mC) \to (\Fun([1], \mC) \times_\mC \Fun([1], \mC))\times_\mC \Fun([1], \mC) $$ is an equivalence, which sends 
a commutative square
$$\begin{xy}
\xymatrix{
A \ar[d]  \ar[r] & B \ar[d]
\\ 
C \ar[r] & D
}
\end{xy}$$
to $(C \to D, B \to D, A \to C \times_D B).$

For $\mC = \Fun([\n],\infty\Cat)$ the resulting equivalence 
$$ \Fun([\n], \Fun([1] \times [1],\infty\Cat)) \simeq \Fun([1] \times [1], \Fun([\n],\infty\Cat)) \to $$$$ (\Fun([1], \Fun([\n],\infty\Cat)) \times_{\Fun([\n],\infty\Cat)} \Fun([1], \Fun([\n],\infty\Cat)))\times_{\Fun([\n],\infty\Cat)} \Fun([1], \Fun([\n],\infty\Cat)) $$
$$\simeq (\Fun([\n], \Fun([1],\infty\Cat)) \times_{\Fun([\n],\infty\Cat)} \Fun([\n], \Fun([1],\infty\Cat)))\times_{\Fun([\n],\infty\Cat)} \Fun([\n], \Fun([1],\infty\Cat))$$
restricts to an equivalence
$$ \Fun([\n], \Fun([1] \times [1],\infty\Cat))' \to $$$$  (\Fun([1],\infty\Cat) \times_{\infty\Cat} \Fun([1],\infty\Cat)) \times_{\infty\Cat} \Fun([\n], \Fun([1],\infty\Cat))'.$$
The latter is an equivalence over
$\Fun([1],\infty\Cat)$ via projection to the first factor, whose pullback along the functor $\infty\Cat \to \Fun([1],\infty\Cat), A \mapsto 0 \vectimes_A A $ is an equivalence
$$ \infty\Cat \times_{\Fun([1], \infty\Cat)} \Fun([\n], \Fun([1] \times [1],\infty\Cat))' \to \Fun([1], \infty\Cat)\times_{\infty\Cat} \Fun([\n], \Fun([1],\infty\Cat))'$$
that sends an $\infty$-category $A$ equipped with a commutative square
$$\begin{xy}
\xymatrix{
(Y_0 \to ...\to Y_n) \ar[d]  \ar[r] & E \ar[d]
\\ 
0 \vectimes_A A \ar[r] & A
}
\end{xy} $$
to $(E \to A, (Y_0 \to ...\to Y_n) \to (0 \vectimes_A A) \times_A E \simeq 0 \vectimes_A E).$

\end{proof}

\begin{proposition}\label{huled} Let $2 \leq n \leq \infty$ and $ 2 \leq i \leq n+1.$
The functor $$\lambda: \Fun([\n],\infty\Cat) \to \Fun([\n-1],\Cocart)'$$
induces a functor 
$$ \lambda': \Fun([\n], \Fun([1],\infty\Cat))' \simeq $$$$ \infty\Cat \times_{\Fun([\n],\infty\Cat)}\Fun([1], \Fun([\n],\infty\Cat)) \to $$$$ \infty\Cat \times_{\Fun([\n-1],\Cocart)'} \Fun([1],\Fun([\n-1],\Cocart)') \simeq $$
$$ \infty\Cat \times_{\Cocart} (\Cocart \times_{\Fun([\n-1],\Cocart)'} \Fun([1],\Fun([\n-1],\Cocart)')) \subset $$
$$ \infty\Cat \times_{\Fun([1],\infty\Cat)} (\Fun([1],\infty\Cat) \times_{\Fun([\n-1],\Fun([1],\infty\Cat))'} \Fun([1],\Fun([\n-1],\Fun([1],\infty\Cat))')) $$
$$ \simeq \infty\Cat
\times_{\Fun([1],\infty\Cat)} \Fun([\n-1],\Fun([1] \times [1],\infty\Cat))', $$
which induces a functor 
$$\lambda'': \Fun([\n],\Cocart)' \to \Cocart
\times_{\infty\Cat} \Fun([\n-1],\Cocart)'$$
sending $ (X_0 \to ... \to X_n \to A) $ to $$ (X_0 \to A, 0 \vectimes_{X_1} X_0 \to ... \to 0 \vectimes_{X_n} X_0 \to 0 \vectimes_{A} X_0) $$
and fitting into a commutative square
$$\begin{xy}
\xymatrix{
\Fun([\n],\Cocart)' \ar[dd]  \ar[rrr]^{\lambda''} &&& \Cocart
\times_{\infty\Cat} \Fun([\n-1],\Cocart)' \ar[d]
\\ 
&&& \Fun([1],\infty\Cat) \times_{\infty\Cat} \Fun([\n-1],\Fun([1],\infty\Cat))' \ar[d]^\simeq
\\
\Fun([\n], \Fun([1],\infty\Cat))' \ar[rrr]^{\lambda'}  &&& \infty\Cat
\times_{\Fun([1],\infty\Cat)} \Fun([\n-1],\Fun([1] \times [1],\infty\Cat))'
}
\end{xy} $$
whose vertical functors are inclusions and whose bottom right vertical equivalence is by \cref{hiris}.

The functor $\lambda''$ induces on commutative monoids a functor
$$ \Fun([\n],\Cmon(\Cocart))' \to \Cmon(\Cocart)
\times_{\Cmon(\infty\Cat)} \Fun([\n-1],\Cmon(\Cocart))',$$
which restricts to an equivalence
$$ \Sk_{n+1}^{i+1}(\infty\Cat\Ch_2) \to \infty\Cat\Ch_2
\times_{\Cmon(\infty\Cat)^\vee} \Sk_{n}^{i}(\infty\Cat\Ch_2) .$$

\end{proposition}

\begin{proof}

The first part of the proof is clear. We prove the second part about the equivalence.

The functor $$\lambda: \Fun([\n],\infty\Cat) \to \Fun([\n-1],\Cocart)'$$
induces on commutative monoids a functor
$$\Fun([\n],\Cmon(\infty\Cat)) \to \Fun([\n-1],\Cmon(\Cocart))',$$
which restricts to an equivalence
$$ \Sk_{n+1}^{i+1} \to \Sk_{n}^{i}(\infty\Cat\Ch_2)$$ by \cref{eqolz}.

Hence the functor 
$$\lambda'': \Fun([\n],\Cocart)' \to \Cocart
\times_{\infty\Cat} \Fun([\n-1],\Cocart)'$$
induces on commutative monoids a functor
$$ \Fun([\n],\Cmon(\Cocart))' \to \Cmon(\Cocart)
\times_{\Cmon(\infty\Cat)} \Fun([\n-1],\Cmon(\Cocart))',$$
which restricts to an equivalence
$$ \Sk_{n+1}^{i+1}(\infty\Cat\Ch_2) \to \infty\Cat\Ch_2
\times_{\Cmon(\infty\Cat)^\vee} \Sk_{n}^{i}(\infty\Cat\Ch_2) .$$

\end{proof}

\begin{proposition}\label{retr} Let $\n \geq 1.$
There is a canonical functor $$ \theta : \Cocart
\times_{\infty\Cat} \Fun([\n],\infty\Cat) \to \Cocart
\times_{\infty\Cat} \Fun([\n],\infty\Cat) $$ 
over $\Cocart$ sending 
$$ (A \to B, \Omega(B) \to Y_1 \to ... \to Y_{n}) $$
to 
$$ (A \to B, (0 \overset{\to}{\times}_{B} A) \vectimes_{A} 0 \to Y_1 \to ... \to Y_{n})$$
postcomposing with the canonical functor
$$ (0 \overset{\to}{\times}_{B} A) \vectimes_{A} 0 \to \Omega(B).$$ 

There is a canonical functor $$ \rho : \Cocart
\times_{\infty\Cat} \Fun([\n],\infty\Cat) \to \Cocart
\times_{\infty\Cat} \Fun([\n],\infty\Cat) $$ 
over $\Cocart$ sending 
$$ (A \to B, (0 \overset{\to}{\times}_{B} A) \vectimes_{A} 0 \to Y_1 \to ... \to Y_{n}) $$
to 
$$ (A \to B, \Omega(B) \to Y_1 \to ... \to Y_{n})$$
postcomposing with the canonical functor
$$ \Omega(B) \to (0 \overset{\to}{\times}_{B} A) \vectimes_{A} 0.$$ 

There is a canonical equivalence $\rho \circ \theta \simeq \id.$

\end{proposition}

\begin{proof}

Let $S$ be a category. A functor $S \to \Cocart$
is classified by a map $A \to B$ of cocartesian fibrations over $S$,
which is itself a cocartesian fibration by \cite[Corollary 3.3.6.]{gepner2026fibrations}.
Hence by \cite[Corollary 4.2.7.]{gepner2026fibrations} for every functor $T \to S$ the induced functor $ \Fun^\lax_S(T, A) \to \Fun^\lax_S(T, B)$ is a coartesian fibration.
Sending an object $ (X, X \xrightarrow{\psi} \tu) $ of $$ (0 \vectimes_{\Fun^\lax_S(T,B)} \Fun^\lax_S(T,A) ) \vectimes_{\Fun^\lax_S(T,A)} 0 \simeq \Fun^\lax_S(T,(0 \overset{\to}{\times}_{B} A) \vectimes_{A} 0) $$
to $$\psi_!(X) \in \Omega(\Fun^\lax_S(T,B)) \simeq \Fun^\lax_S(T,\Omega(B))  $$ represents a map $$\alpha: (0 \overset{\to}{\times}_{B} A) \vectimes_{A} 0 \to \Omega(B) $$ of cocartesian fibrations over $S$.
Sending an object $ (X, X \simeq \tu) $ of $ \Omega(\Fun^\lax_S(T,B)) \simeq \Fun^\lax_S(T,\Omega(B))$
to itself viewed as an object of 
$$(0 \vectimes_{\Fun^\lax_S(T,B)} \Fun^\lax_S(T,A) ) \vectimes_{\Fun^\lax_S(T,A)} 0 \simeq \Fun^\lax_S(T,(0 \overset{\to}{\times}_{B} A) \vectimes_{A} 0) $$ represents a map $$\beta: \Omega(B) \to (0 \overset{\to}{\times}_{B} A) \vectimes_{A} 0 $$ of cocartesian fibrations over $S$ such that $\alpha \circ \beta \simeq \id.$

A functor $$ S \to \Cocart
\times_{\infty\Cat} \Fun([\n],\infty\Cat)$$
is classified by a map $A \to B$ of cocartesian fibrations over $S$,
which is itself a cocartesian fibration, together with maps $\Omega(B) \to Y_1 \to ... \to Y_{n}$
of cocartesian fibrations over $S$.

A functor $S \to \Cocart
\times_{\infty\Cat} \Fun([\n],\infty\Cat)$
is classified by a map $A \to B$ of cocartesian fibrations over $S$,
which is itself a cocartesian fibration, together with maps $(0 \overset{\to}{\times}_{B} A) \vectimes_{A} 0 \to Z_1 \to ... \to Z_{n}$
of cocartesian fibrations over $S$.

Using the maps $ \alpha: (0 \overset{\to}{\times}_{B} A) \vectimes_{A} 0 \to \Omega(B)$
and $ \beta: \Omega(B) \to (0 \overset{\to}{\times}_{B} A) \vectimes_{A} 0$
of cocartesian fibrations over $S$, we can go from a functor
$$ S \to \Cocart
\times_{\infty\Cat} \Fun([\n],\infty\Cat)$$ to a functor $$ S \to \Cocart
\times_{\infty\Cat} \Fun([\n],\infty\Cat), $$ and vice versa.

We obtain the desired functors
$$ \theta : \Cocart
\times_{\infty\Cat} \Fun([\n],\infty\Cat) \to \Cocart
\times_{\infty\Cat} \Fun([\n],\infty\Cat) $$ 
and 
$$ \rho : \Cocart
\times_{\infty\Cat} \Fun([\n],\infty\Cat) \to \Cocart
\times_{\infty\Cat} \Fun([\n],\infty\Cat) $$ 
over $\Cocart$ such that $\rho \circ \theta \simeq \id.$

\end{proof}

\begin{proposition}\label{propma} Let $\n \geq 2.$
There is a commutative square of categories
$$\begin{xy}
\xymatrix{
\Fun([\n],\infty\Cat) \ar[dd]^\lambda \ar[rrr]^-{\chi} &&& \Fun([1],\infty\Cat)
\times_{\infty\Cat} \Fun([\n-1],\infty\Cat) \ar[d]^\theta
\\ &&& \Fun([1],\infty\Cat)
\times_{\infty\Cat} \Fun([\n-1],\infty\Cat) \ar[d]^{\lambda \times_{\infty\Cat} \lambda}
\\ 
\Fun([\n-1],\Cocart)' \ar[rrr]^-{\lambda''} &&& \Cocart
\times_{\infty\Cat} \Fun([\n-2],\Cocart)'
}
\end{xy}$$

\end{proposition}

\begin{proof}

In view of the definitions of the functors in the commutative square, 
we have to see that for every sequence
$ X_0 \to ... \to X_n $ in $\infty\Cat$ there is a canonical equivalence
$$ 0 \vectimes_{(0 \vectimes_{X_{i+1}} X_0)} (0 \vectimes_{X_{i}} X_0) \simeq 0 \vectimes_{\Omega(X_{i+1})} ((0 \vectimes_{X_{i}} X_0) \vectimes_{X_0} 0) $$ natural in $0 \leq i < n $ and in $ (X_0 \to ... \to X_n) \in \Fun([\n],\infty\Cat).$
This is \cref{mainlem}.

\end{proof}

We obtain the following: 

\begin{theorem}\label{TH} Let $ 2 \leq n \leq \infty$ and $2 \leq i \leq n+1.$
The canonical functor
$$ \chi: \mathrm{Sk}^{i+1}_{n+1} \to \Sk_2 \times_{\Cmon(\infty\Cat)} \mathrm{Sk}^{i}_{n} $$
sending $X_0 \to ... \to X_n $ to $(X_0 \to X_1, \Omega(X_1) \to ... \to \Omega(X_n))$
is an equivalence.

\end{theorem}

\begin{proof}

In the commutative square of \cref{propma} all functors preserve finite products. So we obtain an induced commutative square on commutative monoids
$$\begin{xy}
\xymatrix{
\Fun([\n],\Cmon(\infty\Cat)) \ar[dd]^\lambda \ar[r]^-{\chi} & \Fun([1],\Cmon(\infty\Cat))
\times_{\Cmon(\infty\Cat)} \Fun([\n-1],\Cmon(\infty\Cat)) \ar[d]^\theta
\\ & \Fun([1],\Cmon(\infty\Cat))
\times_{\Cmon(\infty\Cat)} \Fun([\n-1],\Cmon(\infty\Cat)) \ar[d]^{\lambda \times_{\Cmon(\infty\Cat)} \lambda}
\\ 
\Fun([\n-1],\Cmon(\Cocart))' \ar[r]^-{\lambda''} & \Cmon(\Cocart)
\times_{\Cmon(\infty\Cat)} \Fun([\n-2],\Cmon(\Cocart))'
}
\end{xy}$$

By \cref{retr} the functor $\theta$ admits a left inverse $\rho$ that preserves finite products. Hence $\theta$ induces on commutative monoids a functor right inverse to the functor on commutative monoids induced by 
$\rho$. 

This commutative square restricts on full subcategories to a commutative square of categories
$$\begin{xy}
\xymatrix{
\Sk_{n+1}^{i+1} \ar[dd]^\lambda \ar[rr]^-{\chi} && \Sk^2_2
\times_{\Cmon(\infty\Cat)^\vee} \Sk^{i}_{n} \ar[d]^\theta
\\ && \Sk^2_2
\times_{\Cmon(\infty\Cat)^\vee} \Sk^{i}_{n}  \ar[d]^{\lambda \times_{\Cmon(\infty\Cat)^\vee} \lambda}
\\ 
\Sk_{n}^{i}(\infty\Cat\Ch_2) \ar[rr]^-{\lambda''} && \infty\Cat\Ch_2
\times_{\Cmon(\infty\Cat)^\vee} \Sk_{n-1}^{i-1}(\infty\Cat\Ch_2)
}
\end{xy}$$

The left vertical functor and right bottom vertical functor are equivalences by \cref{eqolz}.
The bottom horizontal functor is an equivalence by \cref{huled}.
\cref{retr} implies that for the functor $\theta$ of the latter commutative square there is a functor $\rho$ such that $\rho \circ \theta$ is fully faithful. 
This commutative square provides a right inverse for $\theta.$
This implies that $\theta$ is an equivalence.

\end{proof}

\begin{corollary}\label{fold}
Let $ 2 \leq n \leq \infty.$
The canonical functor
$$ \zeta: \mathrm{Sk}_{n+1} \to \Sk_2 \times_{\Cmon(\infty\Cat)} \Sk_2 \times_{\Cmon(\infty\Cat)} ...\times_{\Cmon(\infty\Cat)} \Sk_2 $$
sending $X_0 \to ... \to X_n $ to $$(X_0 \to X_1, \Omega(X_1) \to \Omega(X_2),..., \Omega^{n-1}(X_{n-1}) \to \Omega^{n-1}(X_{n})) $$
is an equivalence.

\end{corollary}

\begin{proof}

We proceed by induction on $n \geq 2.$ The induction start $n=2$ is \cref{TH}. Let $n > 2.$ We assume the statement holds for $n-1$.
The canonical functor 
$$ \zeta: \mathrm{Sk}_{n+1} \to \Sk_2 \times_{\Cmon(\infty\Cat)} \Sk_2 \times_{\Cmon(\infty\Cat)} ...\times_{\Cmon(\infty\Cat)} \Sk_2$$
factors as 
$$ \mathrm{Sk}_{n+1} \xrightarrow{\chi} \Sk_2 \times_{\Cmon(\infty\Cat)} \mathrm{Sk}_{n} \xrightarrow{ \Sk_2 \times_{\Cmon(\infty\Cat)} \zeta} \Sk_2 \times_{\Cmon(\infty\Cat)} \Sk_2 \times_{\Cmon(\infty\Cat)} ...\times_{\Cmon(\infty\Cat)} \Sk_2.$$

The first functor $\chi$ is an equivalence by \cref{TH}.
The second functor is an equivalence by induction hypothesis.

\end{proof}

We obtain the following:
\begin{theorem}\label{DK00}
Let $ 2 \leq n \leq \infty.$
The canonical functor
$$ \mathrm{Sk}_n \to \Cat\Ch^{\mathrm{w}}_n $$
sending $X_0 \to ... \to X_{n-1} $ to $$(0 \overset{\to}{\times}_{X_1} X_0 \to X_0, 0 \overset{\to}{\times}_{\Omega(X_2)} \Omega(X_1) \to \Omega(X_1), ..., 0 \overset{\to}{\times}_{\Omega^{n-1}(X_{n-1})} \Omega^{n-1}(X_{n-2}) \to \Omega^{n-1}(X_{n-2})) $$
is an equivalence.
Under this equivalence skeletal sequences of length $n$ correspond to categorical chain complexes of length $n$.

\end{theorem}

\begin{proof}

The functor of the statement factors as
$$ \mathrm{Sk}_n \xrightarrow{\zeta} \Sk_2 \times_{\Cmon(\infty\Cat)} \Sk_2 \times_{\Cmon(\infty\Cat)} ...\times_{\Cmon(\infty\Cat)} \Sk_2 
$$$$ \xrightarrow{\rho \times_{\Cmon(\infty\Cat)} ... \times_{\Cmon(\infty\Cat)} \rho} \Cat\Ch^{\mathrm{w}}_2 \times_{\Cmon(\infty\Cat)} ... \times_{\Cmon(\infty\Cat)} \Cat\Ch^{\mathrm{w}}_2 = \Cat\Ch^{\mathrm{w}}_n.$$

The first functor is an equivalence by \cref{fold}.
The second functor is an equivalence by \cref{chainsk1}.

\end{proof}

\appendix

\section{Enriched categories}

\begin{lemma}\label{liftzuz}

Let $\phi: \mC \to \mD$ be a functor, $\mA$ a category
and $\phi': \Fun(\mA,\mC) \to \Fun(\mA,\mD)$ the induced functor.
Let $F: \mA \to \mC, G: \mA \to \mD$ be functors and $\alpha : \phi \circ F \to G $ a natural transformation such that for every $A \in \mA$ the induced morphism $\alpha_A: \phi(F(A)) \to G(A) $ admits a $\phi$-cocartesian lift
$F(A) \to B$. Then $\alpha$ admits a $\phi'$-cocartesian lift $F \to G'$
that is objectwise $\phi$-cocartesian.

\end{lemma}

\begin{proof}

Let $\Xi \subset \infty\Cat$ be the full subcategory of categories $\mA$, for which the statement holds.
By uniqueness of cocartesian lifts, the full subcategory $\Xi$ is closed in $\infty\Cat$ under coproducts and arbitrary coproducts and so closed under small colimits.
Since $\infty\Cat$ is generated under small colimits by the final $\infty$-category and $[1],$ we can reduce to the case that $\mA= [0], [1].$
For $\mA= [0]$ there is nothing to show. So it suffices to assume that 
$\mA= [1].$
For every $i=0,1$ let $\beta_i: F(i) \to B_i$ be a $\phi$-cocartesian lift of $\alpha_i: \phi(F(i)) \to G(i).$
Since $\beta_0$ is $\phi$-cocartesian, there is a commutative square
$$\begin{xy}
\xymatrix{
F(0) \ar[r] \ar[d]^{\beta_0} & F(1) \ar[d]^{\beta_1}
\\ 
B_0 \ar[r] & B_1
}
\end{xy}$$
that lies over the commutative square
$$\begin{xy}
\xymatrix{
\phi(F(0)) \ar[r] \ar[d]^{\alpha_0} & \phi(F(1)) \ar[d]^{\alpha_1}
\\ 
G(0) \ar[r] & G(1)
}
\end{xy}$$

For every morphism $\theta: X \to Y$ in $\mC$ the induced commutative square
$$\begin{xy}
\xymatrix{
\Map_{\Fun([1],\mC)}(B_0 \to B_1,\theta) \ar[r] \ar[d] & \Map_{\Fun([1],\mC)}(F(0) \to F(1),\theta) \ar[d]
\\ 
\Map_{\Fun([1],\mD)}(G(0) \to G(1), \phi \circ \theta) \ar[r] & \Map_{\Fun([1],\mD)}(\phi(F(0)) \to \phi(F(1)), \phi \circ \theta)
}
\end{xy}$$
identifies with the induced commutative square
$$\begin{xy}
\xymatrix{
\Map_{\mC}(B_0,X) \times_{\Map_{\mC}(B_0,Y)} \Map_{\mC}(B_1,Y) \ar[r] \ar[d] & \Map_{\mC}(F(0),X) \times_{\Map_{\mC}(F(0),Y)} \Map_{\mC}(F(1),Y) \ar[d]
\\ 
\Map_{\mD}(G(0),X') \times_{\Map_{\mD}(G(0),Y')} \Map_{\mD}(G(1),Y') \ar[r] & \Map_{\mD}(\phi(F(0)),X') \times_{\Map_{\mD}(\phi(F(0)),Y')} \Map_{\mD}(\phi(F(1)),Y'),
}
\end{xy}$$
where $X':= \phi(X), Y':= \phi(Y).$
The latter is a pullback square since pullbacks commute with pullbacks.

\end{proof}

\begin{lemma}\label{cartolk}Let $X \to S, Y \to S$ be cocartesian fibrations and $\phi: X \to Y $ a map of cocartesian fibrations over $S$.
Let $s \in S$. A morphism of $X$ lying over $s$ is $\phi$-cartesian if and only if it is $\phi_s$-cartesian.
\end{lemma}

\begin{proof}

We prove the non-trivial direction.
Let $f: A \to B$ a morphism of $X$ lying over $s \in S$, which is $\phi_s$-cartesian.
We want to see that $f$ is $\phi$-cartesian, i.e. that for every $C \in X$
the induced map $$ \Map_X(B,C) \to \Map_Y(s,\phi(C)) \times_{\Map_Y(s,\phi(C))} \Map_X(A,C)  $$
is an equivalence.
The latter induces on the fiber over any $g \in \Map_Y(\phi(B),\phi(C))$
the canonical map 
$$ \Map_{X_{s}}(B,g^*(C)) \to \Map_{X_{s}}(A,g^*(C)), $$
which is an equivalence since $f$ is $\phi_s$-cartesian.

\end{proof}

\begin{proposition}\label{cocal}

The functor 
$$\theta: \Enr\Cat \to \Op $$ of 2-categories is a cocartesian fibration.
The collection of $\theta$-cocartesian morphisms is closed under finite products.

\end{proposition}

\begin{proof}

By \cite[Theorem 3.67.]{heine2024bienrichedinftycategories} the functor of 2-categories of $\omega\Enr\Cat \to \Op $ induces a cocartesian fibration on underlying 1-categories whose collecton of cocartesian morphisms is closed under finite products.
By \cite[Proposition 3.69.]{heine2024bienrichedinftycategories} the restriction $\Enr\Cat \subset \omega\Enr\Cat \to \Op $ induces a cocartesian fibration on underlying 1-categories and the embedding
$\Enr\Cat \subset \omega\Enr\Cat$ is a map of cocartesian fibrations over $\Op,$ which preserves finite products. 
By \cite[Remark 2.56.]{heine2024bienrichedinftycategories} the 2-category $\omega\Enr\Cat$ admits cotensors which are preserved by the functor $\omega\Enr\Cat \to \Op. $ Hence the functor of 2-categories $\omega\Enr\Cat \to \Op $ is a 1-cocartesian fibration and so also the functor of 2-categories $\Enr\Cat \to \Op $ is a 1-cocartesian fibration.
Moreover the collections of cocartesian morphisms are closed under finite products.

It remains to see that for every non-symmetric $\infty$-operads $\mV, \mW$
and any $\mV$-enriched category $\mM$ and $\mW$-enriched category $\mN$
the induced functor $$ \Enr\Fun(\mM, \mN) \to \Alg_\mV(\mW) $$ is a cartesian fibration whose cartesian morphisms are those that lie over an equivalence in $\Fun(\mM, \mN).$
By \cite[Proposition 2.100.]{heine2024bienrichedinftycategories} there is an enriched embedding $\mN \subset \mN'$ lying over an embedding of non-symmetric operads $\mW \subset \mW'$ such that $\mW'$ is a monoidal category and $\mN'$ is a left $\mW'$-tensored category.
If we prove that the functor $$ \Enr\Fun(\mM, \mN') \to \Alg_\mV(\mW') $$ is a cartesian fibration whose cartesian morphisms are those that lie over an equivalence in $\Fun(\mM, \mN'),$ then also $$ \Enr\Fun(\mM, \mN) \to \Alg_\mV(\mW) $$ is a cartesian fibration whose cartesian morphisms are those that lie over an equivalence in $\Fun(\mM, \mN).$
Consequently, we can assume that $\mW$ is a monoidal category and $\mN$ is a left $\mW$-tensored category.

The induced functor $$ \Enr\Fun(\mM, \mN) \to \Alg_\mV(\mW) $$
is the restriction of the induced functor 
$$ \Fun_{\Delta^\op}(\mM^\circledast, \mN^\circledast) \times_{\Fun_{\Delta^\op}(\mM^\circledast, \mW^\ot)} \Fun_{\Delta^\op}(\mV^\ot, \mW^\ot) \to \Fun_{\Delta^\op}(\mV^\ot, \mW^\ot).$$
If we prove that the functor $$  \Fun_{\Delta^\op}(\mM^\circledast, \mN^\circledast) \times_{\Fun_{\Delta^\op}(\mM^\circledast, \mW^\ot)} \Fun_{\Delta^\op}(\mV^\ot, \mW^\ot) \to \Fun_{\Delta^\op}(\mV^\ot, \mW^\ot) $$ is a cartesian fibration whose cartesian morphisms are those that lie over an equivalence in $\Fun(\mM, \mN)$ by taking the fiber over $[0]$, then also $$ \Enr\Fun(\mM, \mN) \to \Alg_\mV(\mW) $$ is a cartesian fibration whose cartesian morphisms are those that lie over an equivalence in $\Fun(\mM, \mN).$

To see that the functor $$  \Fun_{\Delta^\op}(\mM^\circledast, \mN^\circledast) \times_{\Fun_{\Delta^\op}(\mM^\circledast, \mW^\ot)} \Fun_{\Delta^\op}(\mV^\ot, \mW^\ot) \to \Fun_{\Delta^\op}(\mV^\ot, \mW^\ot) $$ is a cartesian fibration whose cartesian morphisms are those that lie over an equivalence in $\Fun(\mM, \mN),$
it suffices to see that the functor $$ \Fun_{\Delta^\op}(\mM^\circledast, \mN^\circledast) \to \Fun_{\Delta^\op}(\mM^\circledast, \mW^\ot) $$ is a cartesian fibration whose cartesian morphisms are those that lie over an equivalence in $\Fun(\mM, \mN)$ by taking the fiber over $[0]$.

The functor $\mN^\circledast \to \mW^\ot$
induces on the fiber over every object of $[n]$ the functor
$\mN_{[n]}^\circledast \simeq \mW_{[n]}^\ot \times \mN \to \mW_{[n]}^\ot, $
which is a cartesian fibration. 
Since $\mN$ is a left $\mW$-tensored category, the functor
$\mN^\circledast \to \mW^\ot$ is a map of cocartesian fibrations over 
$\Delta^\op.$ 
Hence by \cref{cartolk} the functor $\rho: \mN^\circledast \to \mW^\ot$ is a cartesian fibration relative to the collection of morphisms of $\mW^\ot$ lying over equivalences in $\Delta^\op.$
A morphism in $\Fun_{\Delta^\op}(\mM^\circledast, \mW^\ot)$ is a natural transformation over $\Delta^\op$ and so sends any object of $\mM^\circledast$ to a morphism of $\mW^\ot$ lying over an equivalence,
which admits a $\rho$-cartesian lift after specifying the target.
Hence the result follows from \cref{liftzuz}.

\end{proof}

\begin{notation}Let $\mC$ be an $\infty$-category that admits finite products. 

Let $$ \Cmon(\mC) \subset \Fun(\Fin_*, \mC)$$ be the full subcategory of 
functors $X: \Fin_* \to \mC$ such that for every $\langle n \rangle \in \Fin_*$ the induced functor $$ X(\langle n \rangle) \to \prod_{i=1}^n X(\langle 1 \rangle) $$ is an equivalence.

\end{notation}

\begin{lemma}\label{levelfib}

Let $\mC, \mD$ be $\infty$-categories that admit finite products and 
$\phi: \mC \to \mD$ a cocartesian fibration of $\infty$-categories that preserves finite products such that the collection of $\phi$-cocartesian morphisms is closed under finite products.

The induced functor $\Cmon(\mC) \to \Cmon(\mD)$ is a cocartesian fibration of $\infty$-categories.

\end{lemma}

\begin{proof}

By \cite[Proposition 3.3.9.]{gepner2026fibrations} the functor $ \phi_*: \Fun(\Fin_*, \mC) \to \Fun(\Fin_*, \mD) $ is a cocartesian fibration.
It suffices to see that for every $\phi_*$-cocartesian morphism
$A \to B$ in $\Fun(\Fin_*, \mC)$ such that $A \in \Cmon(\mC)$ and
$\phi_*(A), \phi_*(B) \in \Cmon(\mD)$ also $B \in \Cmon(\mC).$
There is a commutative square
$$\begin{xy}
\xymatrix{
A(\langle n \rangle) \ar[r] \ar[d] & \prod_{i=1}^n A(\langle 1 \rangle) \ar[d]
\\ 
B(\langle n \rangle) \ar[r] & \prod_{i=1}^n B(\langle 1 \rangle)
}
\end{xy}$$
in $\mC$ whose top horizontal morphism is an equivalence, whose vertical morphisms are $\phi$-cocartesian and whose bottom horizontal morphism
is sent to an equivalence by $\phi. $ Hence the bottom horizontal morphism is also $\phi$-cocartesian and is sent to an equivalence by $\phi, $ and so is an equivalence.

\end{proof}

\begin{notation}

Let $$ \Cmon(\Cat)^\cl \subset \Cmon(\Cat) $$
be the full subcategory of closed symmetric monoidal categories.

Let $$ \Cmon(\Enr\Cat)^\cl \subset \Cmon(\Enr\Cat) $$
be the full subcategory of closed symmetric monoidal enriched categories.

\end{notation}

\begin{notation}
Let $$ \Cmon(\Cat)^\lax \subset \Op^\Sigma $$
be the full subcategory of symmetric monoidal categories.

\end{notation}

\begin{remark}
By \cite[Proposition 2.100.]{heine2024bienrichedinftycategories} the embedding $\Cmon(\Cat) \subset \Cmon(\Cat)^\lax$ admits a left adjoint $\Env.$
\end{remark}

By \cref{cocal} the forgetful functor $$ \Enr\Cat \to \Mon(\Cat) $$ is a cocartesian fibration of 2-categories that preserves finite products.
Hence by \cref{levelfib} the induced functor $$\Cmon(\Enr\Cat) \to \Cmon(\Mon(\Cat))) \simeq \Cmon(\Cat) $$ is a cocartesian fibration of 2-categories.

\begin{notation}
Let $$ \Cmon(\Enr\Cat)^{\lax} \subset \Cmon(\Cat)^\lax \times_{\Cmon(\Cat)} \Cmon(\Enr\Cat) $$
be the full subcategory of the pullback along $\Env: \Cmon(\Cat)^\lax \to \Cmon(\Cat)$ spanned by the pairs $(\mV, \mC)$ such that
the morphism objects of $\mC$ lie in $\mV \subset \Env(\mV).$

\end{notation}

\begin{remark}
The restriction $$ \Cmon(\Enr\Cat)^{\lax} \to \Cmon(\Cat)^\lax $$
is a cocartesian fibration and the embedding
$$ \Cmon(\Enr\Cat)^{\lax} \subset \Cmon(\Cat)^\lax \times_{\Cmon(\Cat)} \Cmon(\Enr\Cat) $$ 
is a map of cocartesian fibrations over $\Cmon(\Cat)^\lax$ because every lax symmetric monoidal functor
$\mV \to \mW$ extends the induced lax symmetric monoidal functor
$\Env(\mV) \to \Env(\mW).$

\end{remark}

\begin{lemma}\label{remenr}

There is a canonical equivalence 
$$ \Cmon(\Enr\Cat) \simeq \Cmon(\Cat) \times_{\Cmon(\Cat)^\lax} \Cmon(\Enr\Cat)^{\lax}$$
over $\Cmon(\Cat).$

\end{lemma}

\begin{proof}

The counit $\xi: \Env \circ \bj \to \id $ of the adjunction 
$\Env: \Cmon(\Cat)^\lax \rightleftarrows \Cmon(\Cat):\bj $
and the cocartesian fibration $\Cmon(\Enr\Cat) \to \Cmon(\Cat)$
induce a map 
$$ \Cmon(\Cat) \times_{\Cmon(\Cat)^\lax} \Cmon(\Enr\Cat)^{\lax} \to \Cmon(\Enr\Cat) $$
of cocartesian fibrations over $\Cmon(\Cat).$
The latter induces on the fiber over every monoidal category $\mV$
the restriction of the functor $(\xi_\mV)_!: \Calg(\Env(\mV)\mathrm{-}\Cat) \to \Calg(\mV\mathrm{-}\Cat)$
to the full subcategory of symmetric monoidal $\mV$-enriched categories whose morphism objects lie in $\mV \subset \Env(\mV).$
We prove that this restriction is an equivalence.

By \cite[Lemma 2.105.]{heine2024bienrichedinftycategories} the lax symmetric monoidal embedding $\mV \subset \Env(\mV)$
admits a symmetric monoidal left adjoint $\Env(\mV) \to \mV$ that extends the identity and so agrees with $\xi_\mV. $
Hence $\xi_\mV$ admits a fully faithful right adjoint and so induces a symmetric monoidal localization 
$$ (\xi_\mV)_!: \Env(\mV) \mathrm{-}\Cat \rightleftarrows \mV\mathrm{-}\Cat$$
whose local objects are the $\mV$-enriched categories whose morphism objects lie in $\mV \subset \Env(\mV).$
The latter induces on commutative algebras a symmetric monoidal localization 
$$(\xi_\mV)_!: \Calg(\Env(\mV) \mathrm{-}\Cat) \rightleftarrows \Calg(\mV\mathrm{-}\Cat) $$
whose local objects are the symmetric monoidal $\mV$-enriched categories whose morphism objects lie in $\mV \subset \Env(\mV).$
Hence $$(\xi_\mV)_!: \Calg(\Env(\mV) \mathrm{-}\Cat) \to \Calg(\mV\mathrm{-}\Cat) $$
restricts to an equivalence on the full subcategory of symmetric monoidal $\mV$-enriched categories whose morphism objects lie in $\mV \subset \Env(\mV).$

\end{proof}

\cref{mainenr} follows immediately from the following results:

\begin{definition}

Let $\mV, \mW$ be symmetric monoidal categories such that $\mV$ is closed.
A symmetric monoidal functor $\alpha: \mV \to \mW$ preserves internal homs if
for every objects $X,Y \in \mV$ the induced morphism
$$\alpha(Y^X) \ot \alpha(X) \simeq \alpha(Y^X \ot X) \to \alpha(Y) $$
exhibits $\alpha(Y^X)$ as the internal hom of $\alpha(X), \alpha(Y),$
i.e. for every $W \in \mW$ the induced map
$$ \Map_\mW(W, \alpha(Y^X)) \to \Map_\mW(W \ot \alpha(X),\alpha(Y))$$
is an equivalence.

\end{definition}

\begin{notation}
Let $$ \Fun(\bD^1,\Cmon(\Cat)^\cl)^\surj \subset \Fun(\bD^1,\Cmon(\Cat)^\cl) $$
be the full subcategory of symmetric monoidal functors that are essentially surjective and preserve internal homs.

\end{notation}

The following might be viewed as a generalization of \cite[Proposition 3.3.21.]{lurie2008classification}:

\begin{theorem}\label{mainenr}

\begin{enumerate}
\item There is a canonical functor
$$ \Phi: \Cmon(\Cat)^\cl \times_{\Cmon(\Cat)^\lax} \Fun^\oplax(\bD^1, \Cmon(\Cat)^\lax) \to $$$$ \Cmon(\Enr\Cat)^\cl \times_{\Cmon(\Cat)^\cl} \Fun(\bD^1,\Cmon(\Cat)^\cl)^\surj $$
that sends a lax symmetric monoidal functor $\phi: \mV \to \mW $ to $(\phi_!(\mV), \mV \to \phi_!(\mV)).$ 

\vspace{1mm}

\item There is a canonical functor 
$$\Psi: \Cmon(\Enr\Cat)^\cl \times_{\Cmon(\Cat)^\cl} \Fun(\bD^1,\Cmon(\Cat)^\cl)^\surj \to $$$$ \Cmon(\Cat)^\cl \times_{\Cmon(\Cat)^\lax} \Fun^\oplax(\bD^1, \Cmon(\Cat)^\lax)
$$
that sends a symmetric monoidal $\mW$-enriched category $\mC$ 
equipped with a symmetric monoidal functor $\phi: \mV \to \mC$
to a lax symmetric monoidal functor $$ \mV \xrightarrow{\phi} \mC \xrightarrow{\Mor_\mC(\tu_\mC,-)} \mW.$$ 

\item There is a canonical equivalence $$ \Phi \circ \Psi \simeq \id. $$

\item There is a canonical equivalence $$ \Psi \circ \Phi \simeq \id.$$

\end{enumerate}

\end{theorem}

\begin{proof}

(1): Let $\nu: \Cmon(\Enr\Cat) \to \Cmon(\Cat) $ be the functor forgetting the enrichment. The forgetful functor $\Enr\Cat \to \Mon(\Cat)$ preserves finite products and so induces a functor $$\gamma: \Cmon(\Enr\Cat) \to \Cmon(\Cat).$$ Let $\xi: \Cmon(\Cat)^\cl \to \Cmon(\Enr\Cat) $ be the canonical section of $\gamma: \Cmon(\Enr\Cat) \to \Cmon(\Cat).$

Let $ \mu: S \to \Cmon(\Cat)^\cl \times_{\Cmon(\Cat)^\lax} \Fun^\oplax(\bD^1,\Cmon(\Cat)^\lax)$ be a functor corresponding to a lax natural transformation $\kappa: \mV \to \mW$ of functors
$S \to \Cmon(\Cat)^\lax,$ where $$ \mV : S \to \Cmon(\Cat)^\cl \subset \Cmon(\Cat)^\lax.$$

Since $\gamma : \Cmon(\Enr\Cat) \to \Cmon(\Cat)$ is a cocartesian fibration, by \cite[Corollary 4.2.7.]{gepner2026fibrations} the functor $$\gamma': \Fun^\lax(S, \Cmon(\Enr\Cat)) \to \Fun^\lax(S, \Cmon(\Cat))$$ is a cocartesian fibration.
Hence there is a $\gamma'$-cocartesian lift $$\kappa': \xi \circ \mV \to \kappa_!(\xi\circ \mV) $$ of $\kappa$ in $\Fun^\lax(S, \Cmon(\Enr\Cat)).$
The lax natural transformation $\nu \circ \kappa': \mV \to \nu \circ \kappa_!(\xi \circ \mV))$ is a natural transformation because for morphism $f: s \to t $ in $S$ the induced symmetric monoidal natural transformation between underlying symmetric categories from $ \mV_s \to \gamma((\kappa_s)_!(\xi(\mV_s))) \to \gamma((\kappa_t)_!(\xi(\mV_t))) $ to $\mV_s \xrightarrow{} \mV_t \to \gamma((\kappa_t)_!(\xi(\mV_t))) $ is an equivalence: the component at any $V \in \mV_s$ is the identity of $ \mV(f)(V).$

Since $\mV_s$ is closed, $(\kappa_s)_!(\xi(\mV_s)) $ is a closed $\mW_s$-enriched category and the essentially surjective symmetric monoidal functor $\mV_s \to \gamma((\kappa_s)_!(\xi(\mV_s)))$ preserves internal homs.
Hence the pair $(\kappa_!(\xi \circ \mV), \nu \circ \kappa': \mV \to \nu \circ \kappa_!(\xi \circ \mV)) $ corresponds to a functor $$ \psi_\mu : S \to \Cmon(\Enr\Cat)^\cl \times_{\Cmon(\Cat)^\cl} \Fun(\bD^1,\Cmon(\Cat)^\cl)^\surj.$$

For every functor $f:  T \to S$ the induced functor $\Fun^\lax(S, \Cmon(\Enr\Cat)) \to \Fun^\lax(T, \Cmon(\Enr\Cat)) $
is a map of cocartesian fibrations.
Hence $\psi_{\mu \circ f} \simeq \psi_\mu \circ f.$
Let $$ \Psi:= \psi_{\mu}: \Cmon(\Cat)^\cl \times_{\Cmon(\Cat)^\lax} \Fun^\oplax(\bD^1,\Cmon(\Cat)^\lax) \to $$$$ \Cmon(\Enr\Cat)^\cl \times_{\Cmon(\Cat)^\cl} \Fun(\bD^1,\Cmon(\Cat)^\cl)^\surj$$ for $\mu$ the identity of $\Cmon(\Cat)^\cl \times_{\Cmon(\Cat)^\lax} \Fun^\oplax(\bD^1,\Cmon(\Cat)^\lax).$
Then $\Psi \circ \mu \simeq \psi_\mu.$

(2) : Let $S$ be a category and $$\psi: S \to \Cmon(\Enr\Cat)^\cl \times_{\Cmon(\Cat)^\cl} \Fun(\bD^1,\Cmon(\Cat)^\cl)^{\surj}$$ a functor corresponding to a functor
$\mC: S \to \Cmon(\Enr\Cat)^\cl$ and a natural transformation $\varphi: \mV \to \nu \circ \mC $
of functors $S \to \Cmon(\Cat)^\cl$ whose components are essentially surjective and preserve internal homs. 
Let $\mW:= \gamma \circ \mC: S \to \Cmon(\Cat).$

By \cite[Corollary 4.8.1.4.]{lurie.higheralgebra} the inclusion $\Cat^{\rc\rc} \subset \widehat{\Cat} $ admits a symmetric monoidal left adjoint that sends a small category to its category of presheaves such that the unit $\lambda$ is componentwise an embedding.
Hence the induced inclusion
$$\Calg(\Cat^{\rc\rc})\subset \Cmon(\widehat{\Cat}) $$
admits a symmetric monoidal left adjoint covering the latter left adjoint
such that the unit $\lambda$ is componentwise a monoidal embedding.
By \cite[Theorem 4.43., Example 4.46.]{heine2024bienrichedinftycategories} the inclusion $$\LMod(\Cat^{\rc\rc})\subset \P\Enr\widehat{\Cat} $$ admits a symmetric monoidal left adjoint $\mP$ that covers the latter symmetric monoidal left adjoint such that the unit is componentwise an enriched embedding.

We obtain a lax symmetric monoidal functor 
$$\mP: \P\Enr\widehat{\Cat} \xrightarrow{} \LMod(\Cat^{\rc\rc})\subset \LMod(\widehat{\Cat}) $$ and a symmetric monoidal natural transformation from the identity to $$\P\Enr\widehat{\Cat} \xrightarrow{\mP} \LMod(\widehat{\Cat}) \subset \P\Enr\widehat{\Cat}.$$
The latter induce a functor $\mP: \Cmon(\P\Enr\widehat{\Cat}) \to \Cmon(\LMod(\widehat{\Cat})) $ and a natural transformation $\eta$ from
the identity to $$\Cmon(\P\Enr\widehat{\Cat}) \xrightarrow{\mP} \Cmon(\LMod(\widehat{\Cat})) \subset \Cmon(\P\Enr\widehat{\Cat}).$$

The functor $$\mW': S \xrightarrow{\mW} \Cmon(\Cat) \subset \Cmon(\widehat{\Cat}) \to \Calg(\Cat^{\rc\rc}) \subset \Cmon(\widehat{\Cat}) $$ 
corresponds via the equivalence $$ \Fun(S, \Cmon(\widehat{\Cat})) \simeq \Cmon(\Fun(S,\widehat{\Cat})) $$ to a commutative monoid in 
$\Fun(S,\widehat{\Cat})$ also denoted $\mW'.$

The functor $$ \mC': S \xrightarrow{\mC} \Cmon(\Enr\Cat) \subset \Cmon(\P\Enr\widehat{\Cat}) \xrightarrow{\mP} \Cmon(\LMod(\widehat{\Cat})) $$
corresponds via the equivalence $$ \Fun(S, \Cmon(\LMod(\widehat{\Cat}))) \simeq \Cmon(\LMod(\Fun(S,\widehat{\Cat}))) $$ to an object of
$$ \{ \mW' \}\times_{\Cmon(\Fun(S,\widehat{\Cat}))} \Cmon(\LMod(\Fun(S,\widehat{\Cat}))) \simeq $$$$ \{ \mW' \}\times_{\Cmon(\Mon(\Fun(S,\widehat{\Cat})))} \Cmon(\LMod(\Fun(S,\widehat{\Cat}))) \simeq $$$$ \Calg(\LMod_{\mW'}(\Fun(S,\widehat{\Cat}))) \simeq \Cmon(\Fun(S,\widehat{\Cat}))_{\mW'/}, $$ 
i.e. a map $\mW' \to \nu \circ \mC'$ in $\Cmon(\Fun(S,\widehat{\Cat}))$ such that
$\mC'$ is the restriction of the initial object of $\Calg(\LMod_{\nu \circ \mC'}(\Fun(S,\widehat{\Cat})))$ along the map $\alpha: \mW' \to \nu \circ \mC'$ in $\Cmon(\Fun(S,\widehat{\Cat}))$.
Hence there is a canonical natural transformation $\beta: \mC' \to \xi \circ \nu \circ \mC'$ of functors $S \to \Cmon(\LMod(\widehat{\Cat}))$ that is sent by $\gamma$ to the natural transformation
$\alpha: \mW' \to \nu \circ \mC'$ of functors $S \to \Cmon(\widehat{\Cat}).$

The natural transformation $\alpha: \mW' \to \nu \circ \mC'$ induces at every object of $S$ and after forgetting the symmetric monoidal structure a right adjoint functor and therefore by \cref{Grayadj}
admits a right adjoint $\alpha^R: \nu \circ \mC' \to \mW' $ in $\Fun^\lax(S,\Cmon(\widehat{\Cat}))$.
The natural transformation $\beta$ induces at every object of $S$ and after forgetting the left action a right adjoint functor and therefore by \cref{Grayadj} admits a right adjoint in $\Fun^\lax(S,\Cmon(\LMod(\widehat{\Cat})))$.

The natural transformation $$\eta \circ \mC: \mC \to \mP \circ \mC = \mC', $$ which is componentwise an enriched embedding, lies over a natural transformation $\mW \to \mW'$, which is componentwise a monoidal embedding.

The right adjoint $\alpha^R : \nu \circ \mC' \to \mW'$ of $\alpha$ 
restricts to a lax natural transformation $\sigma: \nu \circ \mC \to \mW $.
The right adjoint $\beta^R: \xi \circ \nu \circ \mC' \to \mC' $ of $\beta$
is sent by $\gamma$ to $\alpha^R.$ The right adjoint $\beta^R: \xi \circ \nu \circ \mC' \to \mC' $ restricts to a lax natural transformation $\tau:\xi \circ \nu \circ \mC \to \mC $ that is sent by $\gamma$ to the lax natural transformation $\sigma: \nu \circ \mC \to \mW $.

Let $\kappa$ be the composition of lax natural transformations 
$$\mV \xrightarrow{\varphi} \nu \circ \mC \xrightarrow{\sigma} \mW. $$

Then $\kappa$ is the image of the lax natural transformation $$\xi \circ \mV \xrightarrow{\xi \circ \varphi} \xi \circ \nu \circ \mC \xrightarrow{\tau} \mC $$ of functors $S \to \Cmon(\Enr\Cat)$ under $\gamma: \Cmon(\Enr\Cat) \to \Cmon(\Cat)^\lax$.

The lax natural transformation $\kappa: \mV \to \mW$ of functors $S \to \Cmon(\Cat) $ corresponds to a functor $$ \mu: S \to \Cmon(\Cat)^\cl \times_{\Cmon(\Cat)^\lax} \Fun^\oplax(\bD^1,\Cmon(\Cat)).$$
By construction for every functor $f: T \to S $
there is a canonical equivalence $$ \mu_{\psi \circ f} \simeq \mu_\psi \circ f $$ of functors $T \to \Fun^\oplax(\bD^1,\Cmon(\Cat)).$
Let $$\Psi: \Cmon(\Enr\Cat)^\cl \times_{\Cmon(\Cat)^\cl} \Fun(\bD^1,\Cmon(\Cat)^\cl)^\surj \to $$$$ \Cmon(\Cat)^\cl \times_{\Cmon(\Cat)^\lax} \Fun^\oplax(\bD^1, \Cmon(\Cat)^\lax)
$$
be the functor $\mu_\psi$ for $\psi$ the identity of 
$\Cmon(\Enr\Cat)^\cl \times_{\Cmon(\Cat)^\cl} \Fun(\bD^1,\Cmon(\Cat)^\cl)^\surj.$
Then $$\Psi \circ \psi \simeq \mu_\psi.$$ 

(3): Since $\gamma$ is a cocartesian fibration, the lax natural transformation $$ \xi \circ \mV \xrightarrow{\xi \circ \varphi} \xi \circ \nu \circ \mC \xrightarrow{\tau} \mC $$ of functors $S \to \Cmon(\Enr\Cat)$ corresponds to a lax natural transformation $\kappa_!(\xi \circ \mV) \to \mC$ of functors $S \to \Cmon(\Enr\Cat)$ whose image under $\gamma$ is the identity.

We prove next that the lax natural transformation $ \kappa_!(\xi \circ \mV) \to \mC$ is an equivalence.
This implies that the canonical map
$$ \Phi \circ \Psi \circ \psi \simeq (\kappa_!(\xi \circ \mV), \mV \to \nu(\kappa_!(\xi \circ \mV))) \to (\mC, \varphi: \mV \to \nu \circ \mC) $$ is an equivalence. 

The lax natural transformation $\kappa_!(\xi \circ \mV) \to \mC$ induces at any $s \in S$ the canonical map $$ \Mor_{\mC_s}(\tu_{\mC_s},-)_!((\varphi_s)_!(\xi(\mV_s))) \to \Mor_{\mC_s}(\tu_{\mC_s},-)_!(\xi(\mC_s)) \to \mC_s.$$
Since for every morphism $s \to t $ in $S$ the induced monoidal functor
$\mV_s \to \mV_t$ preserves internal homs, the lax natural transformation $\kappa_!(\xi \circ \mV) \to \mC$ is a natural transformation.
Hence it suffices to see that all its components are equivalences.

The canonical enriched functor $$ (\varphi_s)_!(\xi(\mV_s)) \to \xi(\mC_s) $$ is an equivalence: it is essentially surjective because the symmetric monoidal functor $\varphi_s: \mV_s \to \mC_s$ is essentially surjective.
It induces on morphism objects between $X,Y \in \mV_s$
the canonical morphism $\varphi_s(Y^X) \to \varphi_s(Y)^{\varphi_s(X)},$
which is an equivalence because $\varphi_s: \mV_s \to \mC_s$ preserves internal homs.
The canonical essentially surjective enriched functor $$ \Mor_{\mC_s}(\tu_{\mC_s},-)_!(\xi(\mC_s)) \to \mC_s $$ is an equivalence since
it induces on morphism objects for every $X,Y \in \mC_s $
the canonical morphism $$ \Mor_{\mC_s}(\tu_{\mC_s},Y^X) \to \Mor_{\mC_s}(X,Y).$$
The latter is an equivalence since $\mC_s$ is a closed symmetric monoidal enriched $\infty$-category. 
Hence the natural transformation $\kappa_!(\xi \circ F) \to G$ is componentwise an equivalence and so an equivalence.

This shows that the functor $$ \psi: S \to \Cmon(\Enr\Cat)^\cl \times_{\Cmon(\Cat)^\cl} \Fun(\bD^1,\Cmon(\Cat)^\cl)^{\surj}$$
is the image under $\Phi$ of the functor
$S \to \Cmon(\Cat)^\cl \times_{\Cmon(\Cat)^\lax} \Fun^\oplax(\bD^1, \Cmon(\Cat)^\lax)$ corresponding to the lax natural transformation $\kappa: \mV \to \mW$ of functors $ S \to \Cmon(\Cat)^\lax. $

(4): We prove next that for every functor $ \theta: S \to \Cmon(\Cat)^\cl \times_{\Cmon(\Cat)^\lax} \Fun^\oplax(\bD^1, \Cmon(\Cat)^\lax)$ corresponding to a lax natural transformation $\lambda: \mV \to \mW$ of functors $ S \to \Cmon(\Cat)^\lax $ the composition $$S \xrightarrow{\theta} \Cmon(\Cat)^\cl \times_{\Cmon(\Cat)^\lax} \Fun^\oplax(\bD^1, \Cmon(\Cat)^\lax) \xrightarrow{\Phi} $$$$ \Cmon(\Enr\Cat)^\cl \times_{\Cmon(\Cat)^\cl} \Fun(\bD^1,\Cmon(\Cat)^\cl)^{\surj} $$
determines $\theta.$
This implies that $$ \Phi \circ (-): \Fun^\lax(S, \Cmon(\Cat)^\cl \times_{\Cmon(\Cat)^\lax} \Fun^\oplax(\bD^1, \Cmon(\Cat)^\lax)) \to $$$$ \Fun^\lax(S, \Cmon(\Enr\Cat)^\cl \times_{\Cmon(\Cat)^\cl} \Fun(\bD^1,\Cmon(\Cat)^\cl)^{\surj}) $$
is injective on equivalence classes.
By (3) we know that $ \Phi \circ \Psi \simeq \id $ and so also
$ \Phi \circ \Psi \circ \Phi \simeq \Phi $. Hence also $\Psi \circ \Phi \simeq \id$ if $\Phi \circ (-)$ is injective.

The canonical cocartesian lax natural transformation 
$\xi \circ \mV \to \lambda_!(\xi \circ \mV) $ of functors $S \to \Cmon(\Enr\Cat)$ whose image under $\gamma: \Cmon(\Enr\Cat) \to \Cmon(\Cat) $ is $\lambda$
gives rise to a commutative square 
$$\begin{xy}
\xymatrix{
\mV \ar[r]^= \ar[d]^\simeq & \mV \ar[d]
\\ 
\nu \circ \xi \circ \mV \ar[r] & \nu \circ \lambda_!(\xi \circ \mV)
}
\end{xy}$$ 
in $\Fun^\lax(S,\Cmon(\Cat)).$
This commutative square and the lax natural transformation 
$\xi \circ \mV \to \lambda_!(\xi \circ \mV) $ determine a morphism
$$ \psi =(\xi \circ \mV, \mV \simeq \nu \circ \xi \circ \mV) \to \Phi(\theta) = (\lambda_!(\xi \circ \mV), \mV \to \nu \circ \lambda_!(\xi \circ \mV)) $$
in the pullback $$\Fun^\lax(S,\Fun(\bD^1, \Cmon(\Cat)^\cl)^\surj) \times_{\Fun^\lax(S,\Cmon(\Cat)^\cl)} \Fun^\lax(S,\Cmon(\Enr\Cat)^\cl).$$
The functor $\Psi \circ (-)$ sends this morphism to a morphism
in $$\Fun^\lax(S,\Cmon(\Cat)^\cl \times_{\Cmon(\Cat)^\lax} \Fun^\oplax(\bD^1, \Cmon(\Cat)^\lax))$$
corresponding to a lax commutative square
$$\begin{xy}
\xymatrix{
\mV \ar[r]^= \ar[d]^{\Psi \circ \psi} & \mV \ar[d]^{\Psi \circ \Phi \circ \theta}
\\ 
\mV \ar[r]^\lambda & \mW
}
\end{xy}$$ 
in $\Fun^\lax(S,\Cmon(\Cat)^\lax).$

By definition the lax natural transformation $\Psi(\psi)$ factors as 
$$\mV \xrightarrow{\simeq} \nu \circ \xi \circ \mV \xrightarrow{\sigma} \mV $$ and $ \sigma: \nu \circ \xi \circ \mV \xrightarrow{} \mV$
is the image under $\gamma$ of $ \tau: \xi \circ \mV \xrightarrow{} \xi \circ \mV.$ By construction, the lax natural transformation $\tau$ induces at any $s \in S$ a symmetric monoidal $\mV_s$-enriched functor whose underlying functor is 
the equivalence $\Mor_{\xi(\mV_s)}(\tu_{\mV_s},-) = (-)^{\tu_{\mV_s}} \simeq \id_{\mV_s}: \xi(\mV_s) \xrightarrow{} \xi(\mV_s).$ 
Hence $\tau$ and so $\sigma$ are equivalences. Moreover the symmetric monoidal natural transformation of the lax commutative square sends $V \in \mV_s$ to the canonical equivalence $ \lambda(\Mor_{\xi(\mV_s)}(\tu_{\mV_s},V)) \to \Mor_{\lambda_!(\xi(\mV_s))}(\tu_{\mV_s},V).$
Hence the lax commutative square is a commutative square and so 
$\lambda$ is determined by $\Psi \circ \Phi \circ \theta$ and so by $\Phi \circ \theta.$

\end{proof}

\begin{corollary}\label{macorr}

Let $\mW$ be a symmetric monoidal category.
There is a canonical equivalence
$$ \Phi: \Cmon(\Cat)^\cl \times_{\Cmon(\Cat)^\lax} \Cmon(\Cat)^\lax_{//^\oplax \mW} \to \Calg(\mW\mathrm{-}\Cat)^\cl \times_{\Cmon(\Cat)^\cl} \Fun(\bD^1,\Cmon(\Cat)^\cl)^\surj $$
that sends a lax symmetric monoidal functor $\phi: \mV \to \mW $ to $(\phi_!(\mV), \mV \to \phi_!(\mV)).$ 

\end{corollary}

\begin{proof}
By \cref{mainenr} there is a canonical equivalence
$$ \Phi: \Cmon(\Cat)^\cl \times_{\Cmon(\Cat)^\lax} \Fun^\oplax(\bD^1, \Cmon(\Cat)^\lax) \to $$$$ \Cmon(\Enr\Cat)^\cl \times_{\Cmon(\Cat)^\cl} \Fun(\bD^1,\Cmon(\Cat)^\cl)^\surj $$
that sends a lax symmetric monoidal functor $\phi: \mV \to \mW $ to $(\phi_!(\mV), \mV \to \phi_!(\mV)).$ 

The functor $\Phi$ is over $\Cmon(\Cat)^\cl \times \Cmon(\Cat)^\lax$
and induces on the fiber over $\mW \in \Cmon(\Cat)^\lax $
an equivalence 
$$ \Cmon(\Cat)^\cl \times_{\Cmon(\Cat)^\lax} \Cmon(\Cat)^\lax_{//^\oplax \mW} \simeq $$$$ \Cmon(\Cat)^\cl \times_{\Cmon(\Cat)^\lax} \Fun^\oplax(\bD^1, \Cmon(\Cat)^\lax) \times_{\Cmon(\Cat)^\lax} \{ \mW \} \to $$
$$ \{ \mW \} \times_{\Cmon(\Cat)^\lax} \Cmon(\Enr\Cat)^\cl \times_{\Cmon(\Cat)^\cl} \Fun(\bD^1,\Cmon(\Cat)^\cl)^\surj \simeq $$
$$ \Calg(\mW\mathrm{-}\Cat)^\cl \times_{\Cmon(\Cat)^\cl} \Fun(\bD^1,\Cmon(\Cat)^\cl)^\surj, $$
where we use the canonical equivalence
$$ \{ \mW \} \times_{\Cmon(\Cat)^\lax} \Cmon(\Enr\Cat)^\cl \simeq $$$$ \{ \mW \} \times_{\Cmon(\Cat)} \Cmon(\Cat) \times_{\Cmon(\Cat)^\lax} \Cmon(\Enr\Cat)^\cl \simeq $$$$ \{ \mW \} \times_{\Cmon(\Cat)} \Cmon(\Enr\Cat)^\cl \simeq \Calg(\mW\mathrm{-}\Cat)^\cl $$
of \cref{remenr}.

\end{proof}

Next we specialize to $\mW =\infty\Cat$ endowed with the cartesian structure:

\begin{corollary}

There is a canonical equivalence
$$ \Phi: \Cmon(\Cat)^\cl \times_{\Cmon(\Cat)} \Cmon(\Cocart) \to \Cmon(\Cat)^\cl \times_{\Cmon(\infty\Cat)^\cl} \Fun(\bD^1,\Cmon(\infty\Cat)^\cl)^\surj $$
that sends a lax symmetric monoidal functor $\phi: \mV \to \mW $ to $(\phi_!(\mV), \mV \to \phi_!(\mV))$.
The inverse 
$$ \Psi: \Cmon(\Cat)^\cl \times_{\Cmon(\infty\Cat)^\cl} \Fun(\bD^1,\Cmon(\infty\Cat)^\cl)^\surj \to \Cmon(\Cat)^\cl \times_{\Cmon(\Cat)} \Cmon(\Cocart) $$
sends $A \to B$ to $ 0 \vectimes_B A \to A.$

\end{corollary}

\begin{proof}

By \cref{macorr} for $\mW = \infty\Cat$ endowed with the cartesian structure, there is a canonical equivalence
$$ \Phi: \Cmon(\Cat)^\cl \times_{\Cmon(\Cat)^\lax} \Cmon(\Cat)^\lax_{//^\oplax \infty\Cat} \to $$$$ \Cmon(\infty\mathrm{-}\Cat)^\cl \times_{\Cmon(\Cat)^\cl} \Fun(\bD^1,\Cmon(\Cat)^\cl)^\surj $$
that sends a lax symmetric monoidal functor $\phi: \mV \to \mW $ to $(\phi_!(\mV), \mV \to \phi_!(\mV)).$ 

The adjunction $$ \Cat \rightleftarrows \infty\Cat: \iota $$
provides a canonical equivalence 
$$ \Cmon(\infty\mathrm{-}\Cat)^\cl \times_{\Cmon(\Cat)^\cl} \Fun(\bD^1,\Cmon(\Cat)^\cl)^\surj \simeq $$$$ \Cmon(\Cat)^\cl \times_{\Cmon(\infty\Cat)^\cl} \Fun(\bD^1,\Cmon(\infty\Cat)^\cl)^\surj, $$
where the right hand pullback is taken over the embedding 
$\Cmon(\Cat)^\cl \subset \Cmon(\infty\Cat)^\cl $ and the functor evaluating at the source.

The functor $ \Cmon(\Cat)^\lax_{//^\oplax \infty\Cat} \to \Cmon(\Cat)^\lax$
is a cartesian fibration and classifies the functor
$$ \Mor_{\Cmon(\Cat)^\lax}(-, \infty\Cat) \simeq \Alg_{(-)}(\infty\Cat): \Cmon(\Cat)^\lax \to \Cat .$$

The functor $ \Cocart \to \Cat $ evaluating at the target is a cartesian fibration and preserves finite products. Hence it induces a functor $$ \Cmon(\Cocart) \to \Cmon(\Cat), $$ which a cartesian fibration.
This cartesian fibration classifies the functor
$$ X \mapsto \Cmon(\Cocart)_{/X} \simeq \Cocart^\mon_{/X^\ot} \simeq \Fun^\mon(X^\ot, \infty\Cat) \simeq \Alg_{X}(\infty\Cat).$$
Hence there is a canonical equivalence
$$ \Cmon(\Cat) \times_{\Cmon(\Cat)^\lax} \Cmon(\Cat)^\lax_{//^\oplax \infty\Cat} \simeq \Cmon(\Cocart) $$ over $\Cmon(\Cat)$ and so a canonical equivalence
$$ \Cmon(\Cat)^\cl \times_{\Cmon(\Cat)^\lax} \Cmon(\Cat)^\lax_{//^\oplax \infty\Cat} \simeq \Cmon(\Cat)^\cl \times_{\Cmon(\Cat)} \Cmon(\Cocart) $$ over $\Cmon(\Cat)^\cl.$

\end{proof}

\begin{corollary}\label{chainsk}

The canonical functor
$$ \Fun(\bD^1,\Cmon(\infty\Cat))\to \Cmon(\Cocart) \subset \Cmon(\Fun(\bD^1,\infty\Cat)) \simeq \Fun(\bD^1,\Cmon(\infty\Cat)) $$
that sends $A \to B$ to $ 0 \vectimes_B A \to A$,
restricts to an equivalence 
$$ \Cmon(\Cat)^\cl \times_{\Cmon(\infty\Cat)^\cl} \Fun(\bD^1,\Cmon(\infty\Cat)^\cl)^\surj \to \Cmon(\Cat)^\cl \times_{\Cmon(\Cat)} \Cmon(\Cocart).$$
\end{corollary}

\bibliographystyle{plain}
\bibliography{ma}

@article{Scholze2025,
  title={Geometry and Higher Category Theory},
  author={Scholze, Peter},
  journal={https://people.mpim-bonn.mpg.de/scholze/Gestalten.pdf},
  year={2025}
}

@article{campion2023gray,
title={The {Gray} tensor product of $(\infty,n)$-categories}, 
author={Timothy Campion},
year={2023},
journal={arXiv: 2311.00205}, 
}

@article{gepner2026colimits,
  title={Colimits in Oriented Category Theory},
  author={Gepner, David and Heine, Hadrian},
  journal={arXiv:2608.02592},
  year={2026}
}

@article{heine2024bienrichedinftycategories,
	title={On bi-enriched $\infty$-categories}, 
	author={Hadrian Heine},
	year={2024},
	journal={arXiv: 2406.09832},
	primaryClass={id='math.CT' full_name='Category Theory' is_active=True alt_name=None in_archive='math' is_general=False description='Enriched categories, topoi, abelian categories, monoidal categories, homological algebra'}
}

@article{heine2024higheralgebraweightedcolimits,
	title={The higher algebra of weighted colimits}, 
	author={Hadrian Heine},
	year={2024},
	journal={arXiv: 2406.08925},
	primaryClass={id='math.CT' full_name='Category Theory' is_active=True alt_name=None in_archive='math' is_general=False description='Enriched categories, topoi, abelian categories, monoidal categories, homological algebra'}
}

@article{Heine2023AnEB,
	title={An equivalence between enriched $\infty$-categories and $\infty$-categories with weak action},
	author={Hadrian Heine},
	journal = {Advances in Mathematics},
	year={2023}
}

@article{gepner2026homotopy,
  title={Homotopy {Posets, Postnikov Towers, and Hypercompletions} of $\infty $-Categories},
  author={Gepner, David and Heine, Hadrian},
  journal={arXiv preprint arXiv:2603.09903},
  year={2026}
}

@misc{HaugsengCT2025,
  author = {Haugseng, Rune},
  title = {Unfolding of symmetric monoidal
           {$(\infty,n)$}-categories},
  note = {Talk at CT2025, Brno,
          16 July 2025. Joint work with Thomas Nikolaus.
          Slides available online},
  year = {2025},
  url = {https://archive.math.muni.cz/conference/ct2025/data/uploads/slides/haugseng.pdf}
}

@article{lurie2008classification,
  title={On the classification of topological field theories},
  author={Lurie, Jacob},
  journal={Current developments in mathematics},
  volume={2008},
  number={1},
  pages={129--280},
  year={2008},
  publisher={International Press of Boston}
}

@article{gepner2026fibrations,
  title={Fibrations in Oriented Category Theory},
  author={Gepner, David and Heine, Hadrian},
  journal={arXiv preprint:2607.18418},
  year={2026}
}

@article{whitehead1949combinatorial,
  title={Combinatorial homotopy I},
  author={Whitehead, John HC},
  journal={Bull. Amer. Math. Soc},
  volume={55},
  number={3},
  pages={213--245},
  year={1949}
}

@article{masuda2026algebra,
      title={The Algebra of Categorical Spectra}, 
      author={Naruki Masuda},
      journal={arXiv: 2605.03114}, 
      year={2026},
}

@book {lurie.HTT,
	AUTHOR = {Lurie, Jacob},
	TITLE = {Higher Topos Theory},
	SERIES = {Annals of Mathematics Studies},
	VOLUME = {170},
	PUBLISHER = {Princeton University Press},
	ADDRESS = {Princeton, NJ},
	YEAR = {2009},
	PAGES = {xviii+925},
	ISBN = {978-0-691-14049-0; 0-691-14049-9},
	MRCLASS = {18-02 (18B25 18E35 18G30 18G55 55U40)},
	MRNUMBER = {2522659 (2010j:18001)},
	MRREVIEWER = {Mark Hovey},
	}

@article{heine2017monadicity,
  title={A monadicity theorem for higher algebraic structures},
  author={Heine, Hadrian},
  journal={arXiv:1712.00555},
  year={2017}
}

@misc{haugseng2020lax,
 	title={On lax transformations, adjunctions, and monads in $(\infty,2)$-categories}, 
 	author={Rune Haugseng},
 	year={2020},
 	eprint={2002.01037},
 	archivePrefix={arXiv},
 	primaryClass={math.CT}
 	}

@article{gepner2025oriented,
  title={Oriented Category Theory},
  author={Gepner, David and Heine, Hadrian},
  journal={arXiv:2510.10504},
  year={2025}
}

@article{may1974,
  title={{$E_\infty$}-spaces, group completions and permutative categories},
  author={May, J Peter},
  journal={New Developments in Topology, London Math. Soc. Lecture Note Series},
  volume={11},
  year={1974}
}

@article{gepner2026oriented,
  title={An {Oriented Street--Roberts} Conjecture},
  author={Gepner, David and Heine, Hadrian},
  journal={arXiv:2606.29373},
  year={2026}
}

@article{heine2026stable,
  title={Stable homotopy theory of higher categories},
  author={Heine, Hadrian},
  journal={arXiv:2605.05195},
  year={2026}
}

@UNPUBLISHED{lurie.higheralgebra,
	AUTHOR = {Lurie, Jacob},
	TITLE = {Higher {A}lgebra},
	NOTE = {http://www.math.harvard.edu/~lurie/},
	}

@article{GEPNER2015575,
 	title = "Enriched $\infty$-categories via non-symmetric $\infty$-operads",
 	journal = {Advances in Mathematics},
 	volume = "279",
 	pages = "575 - 716",
 	year = "2015",
 	issn = "0001-8708",
 	doi = "https://doi.org/10.1016/j.aim.2015.02.007",
 	url = "http://www.sciencedirect.com/science/article/pii/S0001870815000535",
 	author = "David Gepner and Rune Haugseng"
 	}

@article{HINICH2020107129,
 	title = "Yoneda lemma for enriched $\infty$-categories",
 	journal = {Advances in Mathematics},
 	volume = "367",
 	pages = "107129",
 	year = "2020",
 	issn = "0001-8708",
 	doi = "https://doi.org/10.1016/j.aim.2020.107129",
 	url = "http://www.sciencedirect.com/science/article/pii/S0001870820301559",
 	author = "Vladimir Hinich"
 	}

@article{heine2025equivalence,
  title={An equivalence between two models of $\infty$-categories of enriched presheaves},
  author={Heine, Hadrian},
  journal={Applied Categorical Structures},
  volume={33},
  number={1},
  pages={2},
  year={2025},
  publisher={Springer}
}

\end{document}